\documentclass{amsart}

\usepackage{amsmath,amsthm,amssymb,amsfonts,enumerate,color}
\usepackage{mathrsfs,mathtools}
\usepackage{amstext,amsxtra}
\usepackage[margin=1.5in]{geometry}
\usepackage{graphicx}
\usepackage{float}

\usepackage{tikz}
\usetikzlibrary{matrix,arrows,calc,intersections,fit}
\usetikzlibrary{decorations.markings}
\usepackage{tikz-cd}
\usepackage{textcomp}

\usepackage[colorlinks,urlcolor=black,linkcolor=blue,citecolor=blue,hypertexnames=false]{hyperref}
\allowdisplaybreaks
\usepackage{pgf,tikz}
\usepackage{stmaryrd}
\usepackage{bm}

{\theoremstyle{plain}
\newtheorem{theorem}{Theorem}[section]
\newtheorem{proposition}[theorem]{Proposition}
\newtheorem{corollary}[theorem]{Corollary}
\newtheorem{lemma}[theorem]{Lemma}}
{\theoremstyle{definition}
\newtheorem{example}[theorem]{Example}
\newtheorem{definition}[theorem]{Definition}
\newtheorem{construction}[theorem]{Construction}

\newtheorem{question}[theorem]{Question}}
{\theoremstyle{remark}
\newtheorem{remark}[theorem]{Remark}}

\newcommand{\mult}{\operatorname{mult}}
\newcommand{\aut}{\operatorname{Aut}}

\newcommand{\id}{\operatorname{id}}

\newcommand{\vdim}{\operatorname{vdim}}
\newcommand{\vt}{\operatorname{Vert}}
\newcommand{\eg}{\operatorname{Edge}}
\newcommand{\dv}{\operatorname{div}}

\newcommand{\trop}{\operatorname{trop}}

\newcommand{\el}{\mathcal{E}}

\newcommand{\cl}{\mathcal{C}}

\newcommand{\nl}{\mathcal{N}}

\newcommand{\ol}{\mathcal{O}}
\newcommand{\pl}{\mathcal{P}}
\newcommand{\dl}{\mathcal{D}}
\newcommand{\yl}{\mathcal{Y}}

\newcommand{\xl}{\mathcal{X}}
\newcommand{\gl}{\mathcal{G}}

\newcommand{\rb}{\mathbb{R}}
\newcommand{\fb}{\mathbb{F}}
\newcommand{~}{\quad}
\newcommand{\cb}{\mathbb{C}}

\newcommand{\zb}{\mathbb{Z}}
\newcommand{\pb}{\mathbb{P}}

\newcommand{\qb}{\mathbb{Q}}
\newcommand{\eb}{\mathbb{E}}
\newcommand{\ab}{\mathbb{A}}

\newcommand{\xk}{\mathfrak{X}}

\newcommand{\undl}{\underline}

\newcommand{\wt}{\widetilde}

\definecolor{cardinalred}{RGB}{140,21,21}
\definecolor{coolgray}{RGB}{77,79,83}
\definecolor{black}{RGB}{0,0,0}
\definecolor{beige}{RGB}{210,194,149}
\definecolor{darkbeige}{RGB}{179,153,93}
\definecolor{darkcardinal}{RGB}{94,48,50}
\definecolor{lightcardinal}{RGB}{141,60,30}
\definecolor{darkpurple}{RGB}{83,40,79}
\definecolor{darkcyan}{RGB}{0,124,146}
\definecolor{skyblue}{RGB}{0,152,219}
\definecolor{seablue}{RGB}{10,100,180}
\definecolor{darkblue}{RGB}{20,80,150}
\definecolor{treegreen}{RGB}{0,155,118}
\definecolor{darkorange}{RGB}{168,101,12}
\definecolor{beigegray}{RGB}{95,87,79}
\definecolor{boxgray}{RGB}{238,235,233}
\definecolor{footergray}{RGB}{199,209,197}

\begin{document}

\title[Refined floor diagrams and Caporaso-Harris formula]{Refined floor diagrams relative to a conic and Caporaso-Harris type formula}

\author{Yanqiao Ding}

\address{School of Mathematics and Statistics, Zhengzhou University, Zhengzhou, 450001, China}

\email{yqding@zzu.edu.cn}

\author{Jianxun Hu}

\address{School of Mathematics, Sun Yat-sen University, Guangzhou, 510275, China}

\email{stsjxhu@mail.sysu.edu.cn}



\date{\today}

\begin{abstract}
    We prove a $q$-refined correspondence theorem between higher genus relative Gromov-Witten invariants with a Lambda class $\lambda_{g-g'}$ insertion in the blow-up of $\pb^2$ at $k$ points on a conic and the refined counts of genus $g'$ floor diagrams relative to a conic, after the change of variables $q=e^{iu}$.
    We provide a Caporaso-Harris type recursive formula for the refined counts of higher genus floor diagrams. 
    As an application of the correspondence theorem, we propose a higher genus version of the BPS polynomials of del Pezzo surfaces of degree $\geq3$ and Hirzebruch surfaces, which generalize the higher genus Block-G\"ottsche polynomials. 
\end{abstract}

\maketitle

\tableofcontents

\section{Introduction}
The study of counting curves in algebraic surfaces has a long history and has attracted a lot of interest in recent decades. Various enumerative invariants have been defined to address the problem of counting curves on a surface from different perspectives. For example, Gromov-Witten invariants of a surface ``count" the number of complex curves in the surface; Welschinger invariants of a real rational surface provide lower bounds for the number of real rational curves in the surface \cite{wel2005a}; refined invariants of a complex surface count curves in a refined way \cite{gs14}. 
In his remarkable paper \cite{mikhalkin05}, Mikhalkin showed that Gromov-Witten invariants and Welschinger invariants of toric surfaces can be computed using tropical curves with complex and real multiplicities, respectively.
Block and G\"ottsche \cite{bg2016} discovered that complex and real multiplicities can be described using quantum integers (see Section \ref{sec:2.2}) and provided a refined multiplicity for tropical curves.
The Block-G\"ottsche multiplicity is a symmetric Laurent polynomial in the variable $q^{\frac{1}{2}}$ with integer coefficients. When $q=1$,
the refined multiplicity coincides with the complex multiplicity, and when $q=-1$, the refined multiplicity recovers the real multiplicity. 
Itenberg-Mikhalkin \cite{im13} showed that the refined count of tropical curves of fixed degree and genus passing through a configuration of points is an invariant when the configuration is generic.
Such invariants are called the Block-G\"ottsche polynomials of toric surfaces.
Since then, considerable effort has been devoted to exploring the refined counts of curves in the sense of Block and G\"ottsche \cite{bs19,fs15,mandel21,gk16,bg16,ss18,shustin20,gs19,blomme22}, as well as to providing enumerative explanations for the refined counts \cite{mik17,bou2019,nps18}. Refined enumerative geometry has now become one of the most important topics in enumerative geometry.

Mikhalkin \cite{mik17} made a breakthrough in this direction. He provided an interesting perspective for understanding the Block-G\"ottsche multiplicity, 
which was related to a generating function for the count of real rational curves on a toric surface with a fixed area of the amoeba and passing through a configuration of real points on the toric boundary of the surface. 
Blomme extended Mikhalkin's correspondence theorem to the case that real rational curves pass through a configuration of real points 
and pairs of complex conjugated points in toric divisors \cite{blomme20}, and to higher dimensional cases \cite{blomme24-jag}. 
Recently, Itenberg and Shustin extended Mikhalkin's result to the case of curves of genus $1$ and $2$ \cite{is24},
and they introduced new invariants of a class of toric surfaces that arise from an appropriate enumeration of real curves of genus 1 and 2.
Itenberg and Shustin claimed that in their forthcoming papers \cite{is24-2,is24-3}, they will introduce tropical counterparts of the above refined invariants and compare these tropical counterparts with the Block-G\" ottsche refined tropical invariants.

In the definition of Gromov-Witten invariants, we also consider the Lambda class insertions which come from the Chern classes of the Hodge bundle over the moduli spaces of curves.
In his ground-breaking paper \cite{bou2019}, Bousseau showed, after the change of variables $q=e^{iu}$, a generating series of higher genus logarithmic Gromov-Witten invariants of a toric surface with a Lambda class insertion is equal to the Block-G\" ottsche refined tropical invariants. Bousseau's correspondence provided a geometric interpretation of the
Block-G\"ottsche tropical refined invariants and made their deformation invariance evident. 
Let $C$ be a smooth cubic in $\pb^2$. Bousseau \cite{bousseau-2023} proved a correspondence theorem between the higher genus maximal contact Gromov-Witten invariants of $(\pb^2,C)$ with a Lambda class insertion and the Betti numbers of the moduli spaces of one-dimensional Gieseker semistable sheaves on $\pb^2$, after the change of variables $q=e^{iu}$.
In particular, Bousseau \cite{bousseau-2023} obtained a tropical correspondence theorem for a generating series of higher genus 1-marked relative Gromov-Witten invariants with maximal tangency and a Lambda class insertion.
Kennedy-Hunt, Shafi, and Urundolil Kumaran \cite{ksu24} showed that, after the change of variables $q=e^{iu}$, a generating series of higher genus descendant logarithmic Gromov-Witten invariants with a Lambda class insertion in toric surfaces agrees with a $q$-refined count of rational tropical curves satisfying higher valency conditions, as defined by Blechman and Shustin in \cite{bs19}. 
Gr\"afnitz, Ruddat and Zaslow \cite{grz-2024} showed that the proper Landau–Ginzburg superpotential for a toric Fano surface $X$ with smooth anticanonical divisor is equal to the open mirror map for outer Aganagic–Vafa branes in the canonical bundle $K_X$.
Note that all the results discussed above are in the toric case.

The main goal of the present paper is to explore the correspondence between enumerative invariants and refined counts 
in the {\it non-toric} case.
One approach to study such a correspondence is to remain within the tropical setting. 
Nishinou \cite{nishinou20} obtained an algebraic-tropical correspondence theorem for abelian surfaces, 
analogous to Mikhalkin's correspondence for toric surfaces. 
Blomme studied the correspondence between the enumerative invariants of the abelian surfaces and the refined counts of tropical curves in \cite{blomme23,blomme24}, 
and computed these enumerative invariants using the tropical approach in \cite{blomme22-3}.
Let $(X,D)$ be a log Calabi-Yau pair consisting of a smooth projective surface and a singular reduced effective normal crossing anticanonical divisor $D$.
Bousseau \cite{bou-2024} showed that the $q$-refined 2-dimensional Kontsevich-Soibelman scattering diagrams compute a generating series of higher genus logarithmic Gromov–Witten invariants with maximal tangency
condition and insertion of the top Lambda class of $(X,D)$.
When $X$ is a smooth del Pezzo surface of degree $\geq3$ and $D$ is a smooth anticanonical divisor, 
Gr\"afnitz \cite{Graefnitz22} established a correspondence between genus zero logarithmic Gromov-Witten invariants of $X$ 
intersecting $D$ at a single point with maximal tangency and a weighted count of tropical curves arising as tropicalizations of stable log maps. 
This tropical correspondence is generalized to the case of the higher genus $2$-marked logarithmic Gromov-Witten invariants with a top Lambda class insertion of the log Calabi-Yau pair $(Y,E)$ in \cite{grafnitz-2025}, where $Y$ is a smooth del Pezzo surface and $E$ is a smooth anticanonical divisor. 
Recently, Gr\"afnitz, Ruddat, Zaslow and Zhou \cite{grzz-2025} proposed a $q$-refinement of the open mirror map which is defined by quantum periods of mirror curves for outer Aganagic-Vafa branes on the local Calabi-Yau $K_Y$.
Moreover, they showed that the $q$-refinement of the open mirror map determines a generating series of 2-marked higher genus logarithmic Gromov-Witten invariants with a Lambda class insertion of the log Calabi-Yau surface $(Y,E)$.

Another approach to consider the correspondence between enumerative invariants and the Block-G\" ottsche refined counts 
is to employ algebro-geometric methods. 
In the toric case, the enumeration of tropical curves can be described as weighted counts of floor diagrams introduced by Brugall\'e and Mikhalkin \cite{bm2007,bm2008}, 
after establishing a correspondence between tropical curves and floor diagrams.
In the refined case, Bousseau showed, after the change of variables $q=e^{iu}$, 
a generating series of higher genus relative Gromov-Witten invariants of a Hirzebruch surface or $\pb^2$ with a Lambda class insertion can be computed using the refined counts of floor diagrams \cite{bou2021}.
In order to compute Gromov-Witten invariants and Welschinger invariants of non-toric del Pezzo surfaces,
Brugall\'e proposed floor diagrams relative to a conic in \cite{bru2015}. 
By utilizing the degeneration formula in Gromov-Witten theory, Brugall\'e showed that higher genus Gromov-Witten invariants 
and Welschinger invariants of del Pezzo surfaces can be computed using floor diagrams relative to a conic. 
Recently, Arg\"uz and Bousseau \cite{ab25} generalized the genus zero Block-G\"ottsche polynomials of toric del Pezzo surfaces to arbitrary surface $S$ by using BPS polynomials of the 3-fold $S\times\pb^1$, that are Laurent polynomials in $q$ with integer coefficients.
By establishing a $q$-refined correspondence for refined counts of genus zero floor diagrams relative to a conic, they also showed that the evaluation of BPS polynomials at $q=-1$ yields Welschinger invariants of del Pezzo surfaces of degree $\geq3$.

In this paper, we 
provide a correspondence between a generating series of higher genus relative Gromov-Witten invariants with a Lambda class 
$\lambda_{g-g'}$ insertion in the blow-up of $\pb^2$ 
at $k$ points on a conic and the refined counts of {\it higher genus} floor diagrams relative to a conic, after the change of variables $q=e^{iu}$.

\subsection*{Main results}

The first main result is the following theorem.

\begin{theorem}[Theorem \ref{thm:1}]
\label{thm:1-c}
    For the blow-up of $\pb^2$ at $k$ points on a conic, $q$-refined counts $N_q^{\mu_1,\mu_2}(X_k',d,g')$ (see equation $(\ref{eq:2.2})$) of genus $g'$ floor diagrams relative to a conic are, after the change of variables $q=e^{iu}$, generating series of higher genus relative Gromov-Witten invariants with insertion of a Lambda class $\lambda_{g-g'}$.
\end{theorem}

Let $E$ be a smooth conic in $\pb^2$.
We use $X_k$ to denote the blow-up of $\pb^2$ at $k$ 
points in the general position
and we use $X_{k}'$ to denote the blow-up of $\pb^2$
at $k$ general points on a smooth conic $E$.
Since there is no confusion in the context,
we still use $E$ to denote the strict transform of $E$ in $X_{k}'$. 
The normal bundle of $E$ in $X_{k}'$ is denoted by $N_{E|X_{k}'}$.
We denote by $\nl$ the $\pb^1$-bundle $\pb(\ol_E\oplus N_{E|X_{k}'})$ over $E$.

Arg\"uz and Bousseau used genus zero floor diagrams relative to a conic to compute a generating series of higher genus relative Gromov-Witten invariants with a Lambda class insertion of the highest degree in $X_k'$ in \cite[Theorem 4.17]{ab25}.
The computation in \cite[Section 4]{ab25} is based on the degeneration of $X_k'$ into a union of $\pb^2$, a chain of $\fb_4$ and the blow-up of $\fb_4$ at $k$ points in a conic.
We use Brugall\'e's strategy to handle the refined case here.
When the surface $X_k'$ degenerates into a union $Y$ of $X_k'$ and a chain of $n$ copies of $\nl$, we put exactly one point in each copy of $\nl$.
From the degeneration formula, the higher genus relative Gromov-Witten invariants with $n$ point insertions in $X_k'$ 
are determined by the relative Gromov-Witten invariants of $X_k'$ with no point insertion and 
the relative Gromov-Witten invariants of $\nl$ with one point insertion. 
The relative Gromov-Witten invariants of $\nl$ with one point insertion were computed by Bousseau in \cite[Theorem 4.4]{bou2021}, and the relative Gromov-Witten invariants of $X_k'$ with no point insertion were computed by Arg\"uz and Bousseau in \cite[Theorem 4.17]{ab25}. 
After a combinatorial computation of the product of contributions of all floors and edges to the relative Gromov-Witten invariants of $X_k'$ with $n$ point insertions, 
we obtain the correspondence (see Theorem \ref{thm:1}) between a generating series of higher genus relative Gromov-Witten invariants with a Lambda class $\lambda_{g-g'}$ insertion in $X_k'$ and the refined counts of genus $g'$ floor diagrams relative to a conic, after the change of variables $q=e^{iu}$. 
When $g'=0$, our correspondence (Theorem \ref{thm:1}) recovers the correspondence result of Arg\"uz and Bousseau in \cite[Theorem 4.17]{ab25}.

As an application of the correspondence, we introduce a higher genus version of BPS polynomials that are Laurent polynomials in $q$ with integer coefficients, of del Pezzo surfaces of degree $\geq3$ and Hirzebruch surfaces. 

\begin{theorem}[Theorem \ref{thm:main1}]
\label{thm:2-c}
    Let $X$ be a del Pezzo surface of degree $\geq3$ or a Hirzebruch surface. Then, for every $g'\in\zb_{>0}$, the power series $P_{g',d}^X(u)$ in $u$ with rational coefficients is a Laurent polynomial $n_{g',d}^X(q)$, called genus $g'$ BPS polynomial of $X$, in $q$ with integer coefficients under change of variables $q=e^{iu}$, where $P_{g',d}^X(u)$ is defined in $(\ref{eq:g-bps-poly})$.
\end{theorem}

Based on Theorem \ref{thm:2-c} and \cite[Lemma 3.3]{ab25}, it is interesting to ask whether, under change of variables $q=e^{iu}$, the power series $P_{g',d}^{X}(u)$ is a Laurent polynomial in $q$ with integer coefficients, not only for del Pezzo surfaces of degree $\geq3$ and Hirzebruch surfaces but also for other projective surfaces.

\begin{question}
    Let $X$ be a projective surface other than a del Pezzo surface of degree $\geq3$ and a Hirzebruch surface. 
    Let $g'\in\zb_{\geq0}$, and $d\in H_2(X;\zb)$ be a class such that $c_1(X)\cdot d+g'-1\geq0$. 
    Let $P_{g',d}^X(u)$ be the power series in $u$ with rational coefficients given in $(\ref{eq:g-bps-poly})$.
    Is $P_{g',d}^X(u)$ a Laurent polynomial in $q$ with integer coefficients under the change of variables $q=e^{iu}$ ?
\end{question}

Let $X$ be a toric del Pezzo surface or a Hirzebruch surface, we show that the Laurent polynomial $n_{g,d}^X(q)$ is equal to the genus $g$ Block-G\"ottsche polynomial $N_{g,d}^{X,\trop}(q)$ of $X$ (see Section \ref{sec:trop-ref-inv} for the definition of $N_{g,d}^{X,\trop}(q)$).

\begin{theorem}[Theorem \ref{thm:case1}, Theorem \ref{conj1}]\label{thm:3-c}
    Let $X$ be a toric del Pezzo surface or a Hirzebruch surface, then the genus $g$ BPS polynomial $n_{g,d}^X(q)$ of $X$ is equal to the genus $g$ Block-G\"ottsche polynomial $N_{g,d}^{X,\trop}(q)$ of $X$.    
\end{theorem}

In Section \ref{subsec:higher-rel-bps}, we introduce the higher genus version of relative BPS polynomials $n_{g',d,(\mu_1,\mu_2)}^{X_k'|E}(q)$ (see Definition \ref{def:rel-bps-poly}) of $X_k'$, which are related to genus $g'$ BPS polynomials $n_{g',d}^{X_k}(q)$ of $X_k$ by Theorem \ref{thm:2} and Theorem \ref{thm:3}. 
As another application of the correspondence theorem (Theorem \ref{thm:1}), we prove that the genus $g'$ relative BPS polynomials $n_{g',d,(\mu_1,\mu_2)}^{X_k'|E}(q)$ are equal to the refined counts $N_q^{\mu_1,\mu_2}(X_k',d,g')$ of genus $g'$ floor diagrams relative to a conic up to a factor.

\begin{corollary}[Corollary \ref{cor:rel-corre}]
Let $g'\in\zb_{\geq0}$, and $d\in H_2(X_k';\zb)$ be a homology class such that $d\cdot[L]>0$, $d\cdot[E]\geq0$ and $n:=d\cdot[L]-1+g'+l(\mu_2)\geq0$. Suppose that $\vec\mu=(\vec\mu_1,\vec\mu_2)$ is an ordered partition of $d\cdot[E]$. Then,
we have the equality
$$
\begin{aligned}
\prod_{i=1}^{l(\mu_1)}\frac{[\mu_1^{(i)}]_q}{\mu_1^{(i)}}\cdot\prod_{i=1}^{l(\mu_2)}\frac{[\mu_2^{(i)}]_q}{\mu_2^{(i)}}\cdot n_{g',d,(\mu_1,\mu_2)}^{X_k'|E}(q)=N_{q}^{\mu_1,\mu_2}(X_k',d,g').
\end{aligned}
$$
\end{corollary}

Floor diagrams provide us with a combinatorial depiction of curve degeneration. 
One can directly compute the refined counts of floor diagrams relative to a conic from their definition. 
An alternative approach to computing relative enumerative invariants is the Caporaso-Harris type recursive formula,
which has been first applied by Caporaso and Harris in the case of $\pb^2$ with a line \cite{ch1998}.
In the unrefined case, a Caporaso-Harris type recursive formula for relative Gromov-Witten invariants 
in the case of $X_6'$ together with the strict transform of a conic was provided by Vakil in \cite{v2000}.
Shoval and Shustin \cite{ss2013} generalized this to the case of the blow-up of $k$ generic points in $E$ 
and one point in $X\setminus E$. 
In the tropical setting, Gathmann and Markwig proposed a tropical version of Caporaso-Harris formula in \cite{gm2007a,gm2007b}. 
Itenberg, Kharlamov and Shustin applied the ideas of Gathmann and Markwig to 
obtain a Caporaso-Harris formula for tropical Welschinger invariants in \cite{iks2009}. 
Block and G\" ottsche extended the argument of Gathmann and Markwig to the refined case, 
obtaining a Caporaso-Harris type recursive formula for the
tropical refined invariants in \cite{bg2016}.

The second main result of this paper is a Caporaso-Harris type recursive formula for refined counts of floor diagrams relative to a conic.
\begin{theorem}[Theorem \ref{thm:CH}]\label{thm:CH-c}
    The refined counts of floor diagrams relative to a conic satisfy a Caporaso-Harris type recursive formula.
\end{theorem}

Theorem \ref{thm:1} enables us to prove Theorem \ref{thm:CH-c} using the degeneration formula in Gromov-Witten theory, 
rather than analyzing the degeneration properties of floor diagrams. 
The Caporaso-Harris formula in the case of $\pb^2$ with a line was recovered using the degeneration formula in \cite{li2004,ip1998}.
We use the degeneration formula to get a Caporaso-Harris type recursive formula for relative
Gromov-Witten invariants with a Lambda class insertion first (see Lemma \ref{lem:7.1}). 
Then we obtain Theorem \ref{thm:CH-c} applying Theorem \ref{thm:1}.

\subsection*{Relation with previous results}
The contribution of this paper is two-fold. 
First, we obtain a $q$-refined correspondence for refined counts of higher genus floor diagrams and higher genus relative Gromov-Witten invariants in the non-toric case. Moreover, this theorem provides a higher genus version of BPS polynomials, generalizing the higher genus Block-G\"ottsche polynomials to the non-toric case. Second,
we provide a Caporaso-Harris type recursive formula for the refined counts of higher genus floor diagrams.

By using the degeneration formula, we reduce many computations of relative Gromov-Witten invariants of $X_k'$ 
to those on the Hirzebruch surface $\nl$.
Consequently, we use \cite{ab25,bou2021} extensively. 
In \cite{ab25}, Arg\"uz and Bousseau obtained a correspondence between refined counts of genus zero floor diagrams and higher genus relative Gromov-Witten invariants with a Lambda class of highest degree. Our correspondence theorem is a higher genus version of the correspondence established by Arg\"uz and Bousseau in \cite{ab25}. To prove the correspondence (Theorem \ref{thm:1}) we use the degeneration of $X_k'$ used by Brugall\'e, which is different from the strategy in \cite{ab25}.


\subsection*{Organization of the paper}
In Section \ref{sec:2}, we recall the definition of Brugall\'e's floor diagrams relative to a conic and the definition of tropical refined invariants. 
In the next section, we prove our first main result, that is the correspondence between the refined counts of higher genus floor diagrams and the higher genus relative Gromov-Witten invariants with a Lambda class $\lambda_{g-g'}$ insertion.
In Section \ref{sec:5}, we introduce the higher genus (relative) BPS polynomials of the del Pezzo surfaces of degree $\geq3$ and the Hirzebruch surfaces, then we show that the higher genus BPS polynomials are equal to the higher genus Block-G\"ottsche polynomials of toric del Pezzo surfaces and Hirzebruch surfaces.
Our second main result, the Caporaso-Harris type formula for refined counts, is given in Section \ref{sec:6}.
We conclude this paper with concrete computations of the refined counts of floor diagrams via two methods in the Appendix \ref{sec:appendix}.

\subsection*{Notations}
In this paper, we use the following notation.

The homology class in $H_2(X;\zb)$ realized by a divisor
$D$ in a surface $X$ is denoted by $[D]$.
Let $Z^+=\pb(\{0\}\oplus N_{E|X_{k}'})$ and
$Z^-=\pb(\ol_E\oplus\{0\})$ denote
the two distinguished sections of $\nl$.
Note that in the Hirzebruch surface $\nl$,
$[Z^-]^2=-[Z^+]^2=4-k$.
The exceptional divisors of the blow-ups $X_k$ or $X_{k}'$ 
at the $k$ points are denoted by $E_1,\ldots,E_k$.
We use $L$ to denote the strict transform of a general line in $X_{k}'$ or $X_k$.

Let $\mu=(\mu^{(1)},\ldots,\mu^{(n)})$, $\mu'=(\mu^{(1)},\ldots,\mu^{(m)})$ and $\lambda=(\lambda^{(1)},\ldots,\lambda^{(k)})$ be three ordered partitions, where $m\leq n$.
We put 
\begin{align*}
|\mu|=\sum_{i=1}^n\mu^{(i)}&,~ l(\mu)=n,~
\mu\setminus\mu'=(\mu^{(m+1)},\ldots,\mu^{(n)}),\\
\mu\cup\lambda&=(\mu,\lambda)=(\mu^{(1)},\ldots,\mu^{(n)},\lambda^{(1)},\ldots,\lambda^{(k)}).
\end{align*}
We denote by $i(\mu)$ the number of entries in $\mu$ that are equal to $i$.
We use $\lambda\subset\mu$ to denote that the entries in $\lambda$ are also the entries in $\mu$.
Let $a_1,\ldots,a_k, n\in\zb_{\geq0}$ be integers with $\sum_{i=1}^ka_i\leq n$.
Let $\mu_1,\ldots,\mu_k,\mu$ be a sequence of partitions such that $\cup_{i=1}^k\mu_i\subset\mu$.
We put
$$
\begin{aligned}
\binom{n}{a_1,\ldots,a_k}:=\frac{n!}{a_1!\cdots a_k!(n-a_1-\cdots -a_k)!},~
\binom{\mu}{\mu_1,\ldots,\mu_k}:=\prod_{i}\binom{i(\mu)}{i(\mu_1),\ldots,i(\mu_k)}.
\end{aligned}
$$


Let $\Gamma$ be a finite graph. We denote by $\vt(\Gamma)$ and $\eg(\Gamma)$ the sets of vertices and edges of $\Gamma$, respectively. A vertex $v\in\vt(\Gamma)$ is called a \textit{leaf} of $\Gamma$ if $v$ is adjacent to only one edge.
Vertices that are not leaves are called \textit{inner vertices}.
An \textit{end} of $\Gamma$ is an edge that is adjacent to a leaf. Edges of $\Gamma$ which are not ends are called \textit{inner edges} of $\Gamma$.
Given an oriented graph $\Gamma$, it is said \textit{acyclic} if there is no non-trivial oriented cycle in $\Gamma$.
Let $e$ be an edge of the oriented graph $\Gamma$ with two endpoints $v$ and $v'$. If the orientation on $e$ is pointing from $v$ to $v'$, the edge $e$ is called an \textit{outgoing} edge of $v$ and an \textit{incoming} edge of $v'$.
A graph $\Gamma$ is said to be \textit{weighted}, if it is equipped with an integer valued function $w:\eg(\Gamma)\to\zb_{>0}$. For an oriented weighted graph $\Gamma$, we denote by $\vt^\infty(\Gamma)\subset \vt(\Gamma)$ the set consisting of leaves whose adjacent edges are all outgoing edges.
Let $\eg^\infty(\Gamma)$ denote the set of edges adjacent to the vertices in $\vt^\infty(\Gamma)$.

\section{Floor diagrams and tropical refined invariants}
\label{sec:2}

\subsection{Floor diagrams relative to a conic}
We review the floor diagrams relative to a conic from \cite[Section 3]{bru2015}.
Note that we use partitions rather than vectors to represent tangency orders.
For more details on (refined) floor diagrams, we refer the readers to \cite{bm2007,bm2008,bru2015,bou2021,bg2016}.

Let $\Gamma$ be an oriented weighted graph.
For a vertex $v\in\vt(\Gamma)$, the \textit{divergence} $\dv(v)$ at $v$ is the sum of the weights of all incoming edges of $v$ minus the sum of the weights of all outgoing edges of $v$.

\begin{definition}[{\cite[Definition 3.1]{bru2015}}]
\label{def:2.1}
A \textit{floor diagram $\dl$ of genus $g$ and degree $d_\dl$} is a connected weighted oriented graph $\dl$ satisfying the following conditions.
\begin{enumerate}
    \item The oriented graph $\dl$ is acyclic.
    \item The first Betti number $b_1(\dl)$ of $\dl$ is equal to $g$.
    \item For any vertex $v\in\vt(\dl)\setminus\vt^\infty(\dl)$, the divergence $\dv(v)=2$ or $4$. And $\dv(v)\leq-1$ for any $v\in\vt^\infty(\dl)$.
    \item If $\dv(v)=2$, the vertex $v\in\vt(\dl)\setminus\vt^\infty(\dl)$ is a \textit{sink}, that is, a vertex whose all adjacent edges are oriented toward itself.
    \item The sum of divergences of vertices in $\vt^\infty(\dl)$ is $-2d_\dl$.
\end{enumerate}
A vertex $v\in\vt(\dl)\setminus\vt^\infty(\dl)$ is called \textit{a floor of degree $\frac{\dv(v)}{2}$}.
\end{definition}

\begin{remark}
\label{rmk:2.1}
In \cite[Section 3.1]{bru2015}, Brugall\'e pointed out that the objects defined in \cite[Definition 3.1]{bru2015} should be called \textit{floor diagrams in $\cb P^2$ relative to a conic}. 
\end{remark}

Note that the orientation of a floor diagram $\Gamma$ induces a partial ordering on the graph $\Gamma$. Let $m$ be a map between two partially ordered sets. If $m(i)>m(j)$ implies $i>j$, the map $m$ is said to be \textit{increasing}.

\begin{definition}[{\cite[Definition 3.3]{bru2015}}]
\label{def:2.2}
Let $k$, $g$ be two non-negative integers, and $d\in H_2(X_{k}';\zb)$. 
Fix an ordered partition $\vec\mu=(\vec\mu_1,\vec\mu_2)$ of $d\cdot[E]$, 
and suppose that $\vec\mu_1=(\mu_1^{(1)},\ldots,\mu_1^{(l(\mu_1))})$ and $\vec\mu_2=(\mu_2^{(1)},\ldots,\mu_2^{(l(\mu_2))})$.
Choose $k+1$ disjoint sets $A_0,A_1,\ldots,A_k$ such that $|A_i|=d\cdot[E_i]$ for $i=1,\ldots,k$, and
$$
A_0=\{1,\ldots,d\cdot[L]-1+g+l(\mu_1)+l(\mu_2)\}.
$$
A map $m:\cup_{i=0}^kA_i\to\dl$ is called a \textit{$d$-marking of type $(\vec\mu_1,\vec\mu_2)$} of a floor diagram $\dl$ of genus $g$ and degree $d\cdot[L]$ if the following conditions hold.
\begin{enumerate}
    \item The map $m$ is injective and increasing, and no floor of degree $1$ of $\dl$ is contained in $m(\cup_{i=0}^kA_i)$.
    \item For any vertex $v\in\vt^\infty(\dl)$ adjacent to the edge $e\in\eg^\infty(\dl)$, exactly one element of the set $\{v, e\}$ is in the image of $m$.
    \item $m(\cup_{i=1}^kA_i)\subset\vt^\infty(\dl)$ and $m(\{1,\ldots,l(\mu_1)\})=m(A_0)\cap\vt^\infty(\dl)$.
    \item For any $i\in\{1,\ldots,k\}$, a floor of $\dl$ is adjacent to at most one edge which is adjacent to a vertex in $m(A_i)$.
    \item For any $i\in\{1,\ldots,l(\mu_1)\}$, the edge adjacent to $m(i)$ is of weight $\mu_1^{(i)}$.
    \item Exactly $l(\mu_2)$ edges in $\eg^\infty(\dl)$ are in the image of $m|_{A_0}$, and the weights of these $l(\mu_2)$ edges are $\mu_2^{(1)},\ldots,\mu_2^{(l(\mu_2))}$, respectively.
\end{enumerate}
\end{definition}

A floor diagram $\dl$ equipped with a $d$-marking $m$ is called a \textit{$d$-marked floor diagram}, and $\dl$ is also said to be marked by $m$.
Suppose that $m: A_0\bigcup(\cup_{i=1}^kA_i)\to\dl$ and $\bar m: A_0\bigcup(\cup_{i=1}^k\bar A_i)\to\dl$ are two $d$-markings of the floor diagram $\dl$. The two markings $m$ and $\bar m$ are called \textit{equivalent}, if there is an isomorphism of weighted oriented graphs $\varphi:\dl\to\dl$ and a bijection $\psi:A_0\bigcup(\cup_{i=1}^kA_i)\to A_0\bigcup(\cup_{i=1}^k\bar A_i)$ such that $\psi|_{A_0}=\id$, $\bar m\circ\psi=\varphi\circ m$, and $\psi|_{A_i}:A_i\to\bar A_i$ is a bijection for $i\in\{1,\ldots,k\}$.  In the following of this paper, we always consider marked floor diagrams up to equivalence.

\begin{remark}
    We only consider $d$-marked floor diagrams of positive degree, so the class $d\neq l[E_i]$ for any $i=1,\ldots,k$ and any $l>0$. 
\end{remark}

\subsection{$q$-derivative and quantum numbers}\label{sec:2.2}
In this section, we recall the definition of the $q$-derivative from \cite{kc-2002} and quantum numbers. 

\begin{definition}
    Let $f(x)$ be an arbitrary function. The $q$-derivative of the function $f(x)$ is the expression:
    $$
    D_qf(x)=\frac{f(q^{\frac{1}{2}}x)-f(q^{-\frac{1}{2}}x)}{q^{\frac{1}{2}}x-q^{-\frac{1}{2}}x}.
    $$
\end{definition}
Note that
$$
\lim_{q\to1}D_qf(x)=\frac{df(x)}{dx},
$$
if $f(x)$ is differentiable.

\begin{example}
    Let $f(x)=x^n$, where $n$ is a positive integer. By definition,
    $$
    \begin{aligned}
        D_qf(x)&=\frac{q^{\frac{n}{2}}-q^{-\frac{n}{2}}}{q^{\frac{1}{2}}-q^{-\frac{1}{2}}}x^{n-1}.
    \end{aligned}
    $$
\end{example}

The coefficient $\frac{q^{\frac{n}{2}}-q^{-\frac{n}{2}}}{q^{\frac{1}{2}}-q^{-\frac{1}{2}}}$ is known as the quantum integer.
For any $n\in\zb_{\geq0}$, the \textit{quantum integer $[n]_q$} is defined as
$$
[n]_q:=\frac{q^{\frac{n}{2}}-q^{-\frac{n}{2}}}{q^{\frac{1}{2}}-q^{-\frac{1}{2}}}
=q^{-\frac{n-1}{2}}+q^{-\frac{n-3}{2}}+\cdots+q^{\frac{n-1}{2}}.
$$
The quantum integer $[n]_q$ is a Laurent polynomial in a formal variable $q^\frac{1}{2}$. 
In the case $q=1$, the $q$-integer $[n]_q=n$. Assume $q=-1$, 
then the $q$-integer $[n]_q^2=\left\{\begin{aligned}
    &0, \text{ if $n$ is even,}\\
    &1,\text{ if $n$ is odd.}
\end{aligned}\right.$


\subsection{Tropical refined invariants}\label{sec:trop-ref-inv}
We recall the definition of refined tropical invariants in this subsection from \cite[Section 2.3]{ab25}. We refer the readers to \cite{bg2016,im13} for more details.

Let $X$ be a projective toric surface. Let $\rho_1,\ldots,\rho_j$ be the rays of the fan of $X$ in $\rb^2$ with integral primitive directions $m_1,\ldots,m_j\in\zb^2$. 
Denote by $D_1,\ldots,D_j$ the toric divisors in $X$ corresponding to $\rho_1,\ldots,\rho_j$, respectively.
Let $d\in H_2(X;\zb)$ be a curve class such that $d\cdot[D_i]\geq0$ for any $i\in\{1,\ldots,j\}$.
On the toric surface $X$, the balancing condition $\sum_{i=1}^j(d\cdot[D_i])m_i=0$ is true. There are at least two toric divisors $D_i$ such that $d\cdot[D_i]\geq1$, so we get $c_1(X)\cdot d\geq2$.
Fix an integer $g\in\zb_{\geq0}$. Let $\pl$ be a set of $c_1(X)\cdot d+g-1$ points in $\rb^2$. Let $\cl_{d,\pl}$ denote the set of parameterized genus $g$ tropical curves $h:\Gamma\to\rb^2$ with $d\cdot[D_i]$ unbounded weighted 1 edges of direction $m_i$ for all $1\leq i\leq j$ and passing through $\pl$.
If $\pl$ is generic, the set $\cl_{d,\pl}$ is finite (see \cite[Proposition 4.13]{mikhalkin05}). 
Moreover, the domain graph $\Gamma$ of every parameterized tropical curve $h:\Gamma\to\rb^2$ in $\cl_{d,\pl}$ is 3-valent.
Let $e_1,e_2,e_3$ be the three edges of $\Gamma$ incident to a vertex $v$.
The multiplicity of the vertex $v\in\Gamma$ is defined as
$$
\mult_v:=|\det(w_{e_1}u_{e_1},w_{e_2}u_{e_2})|.
$$
Here, $w_{e_1},w_{e_2},w_{e_3}$ are weights of $e_1,e_2,e_3$, respectively, and $u_{e_1},u_{e_2},u_{e_3}$ are the integral primitive direction vectors pointing outward from $v$ of $e_1,e_2,e_3$, respectively.
The balancing condition $\sum_{i=1}^{3}w_{e_i}u_{e_i}=0$ implies that $\mult_v$ does not depend on the choice of $e_1,e_2$.
Following \cite{bg2016}, the Block-G\"ottsche polynomial is defined as
\begin{equation}\label{eq:trop-ref-inv}
N_{g,d}^{X,\trop}(q):=\sum_{h\in\cl_{d,\pl}}\prod_{v\in\Gamma}[\mult_v]_q\in\zb[q^{\pm\frac{1}{2}}].
\end{equation}
When the configuration $\pl$ is generic, from \cite[Theorem 1]{im13},  the polynomial $N_{g,d}^{X,\trop}(q)$ does not depend on $\pl$.

\subsection{Refined counts of floor diagrams}
In this subsection, we first recall the definitions of the complex multiplicities of a marked floor diagram, as described in \cite[Section 3]{bru2015}. 
Then, we recall the Block-G\"ottsche refined multiplicity of a floor diagram relative to a conic in \cite[Section 4]{ab25}. 

\begin{definition}[{\cite[Definition 3.5]{bru2015}}]
\label{def:2.3}
The \textit{complex multiplicity} of a marked floor diagram $(\dl,m)$ of type $(\vec\mu_1,\vec\mu_2)$ is defined as
$$
\mult^\cb(\dl,m)=\prod_{i=1}^{l(\mu_2)}\mu_2^{(i)}\prod_{e\in\eg^\circ(\dl)}w(e)^2,
$$
where $\eg^\circ(\dl)=\eg(\dl)\setminus\eg^\infty(\dl)$, $\mu_2=(\mu_2^{(1)},\ldots,\mu_2^{(l(\mu_2))})$.
\end{definition}
The complex multiplicity of a marked floor diagram only depends on its type and the underlying floor diagram. The count with complex multiplicity of $d$-marked floor diagrams of genus $g$ and type $(\vec\mu_1,\vec\mu_2)$ is given as follows.
\begin{equation}
\label{eq:2.1}
N_{\cb}^{\vec\mu_1,\vec\mu_2}(X_{k}',d,g)=\sum_{[\dl,m]}\mult^\cb(\dl,m),
\end{equation}
where the sum is taken over the equivalence classes of $d$-marked floor diagrams of genus $g$ and type $(\vec\mu_1,\vec\mu_2)$.

We recall the refined version of the multiplicity of a $d$-marked floor diagram in the sense of Block and G\"ottsche \cite{bg2016}.
We use the following notations.
\begin{itemize}
    \item $\eg^{\vec\mu_1}(\dl)$ is the set of edges adjacent to $m(\{1,\ldots,l(\mu_1)\})$.
    \item $\eg^{\vec\mu_2}(\dl)$ is the set of $l(\mu_2)$ edges in $\eg^\infty(\dl)$ given in the item $(6)$ in Definition \ref{def:2.2}.
\end{itemize}

\begin{definition}
\label{def:2.4}
The \textit{$q$-refined multiplicity} of a marked floor diagram $(\dl,m)$ of type $(\vec\mu_1,\vec\mu_2)$ is defined as
\begin{equation}
    \begin{aligned}
        \mult_q(\dl,m)=&\prod_{e\in\eg^{\vec\mu_2}}\frac{[w(e)]_q}{w(e)}\cdot\prod_{e\in\eg^{\vec\mu_1}}\frac{[w(e)]_q}{w(e)}\cdot\prod_{e\in\eg^{\vec\mu_2}}w(e)
        \cdot\prod_{e\in\eg^\circ(\dl)}[w(e)]_q^2.
    \end{aligned}
\end{equation}
\end{definition}
Note that the multiplicity $\mult_q(\dl,m)$ we use here is different from \cite[Definition 4.14]{ab25} by a factor $\prod_{e\in\eg^{\vec\mu_2}}\frac{[w(e)]_q}{w(e)}\cdot\prod_{e\in\eg^{\vec\mu_1}}\frac{[w(e)]_q}{w(e)}$.
The $q$-refined multiplicity of a marked floor diagram only depends on its type and the underlying floor diagram. The count with refined multiplicity of $d$-marked floor diagrams of genus $g$ and type $(\vec\mu_1,\vec\mu_2)$ is defined as
\begin{equation}
\label{eq:2.2}
N_{q}^{\vec\mu_1,\vec\mu_2}(X_{k}',d,g)=\sum_{[\dl,m]}\mult_q(\dl,m),
\end{equation}
where the sum is taken over the equivalence classes of $d$-marked floor diagrams of genus $g$ and type $(\vec\mu_1,\vec\mu_2)$.

\begin{remark}
    From section \ref{sec:2.2} we see that when $q=1$ the refined count $N_{q}^{\vec\mu_1,\vec\mu_2}(X_{k}',d,g)$ of $d$-marked floor diagrams is equal to the analogue complex count $N_{\cb}^{\vec\mu_1,\vec\mu_2}(X_{k}',d,g)$.
    When $q=-1$ and $g=0$, the refined count $N_{q}^{\vec\mu_1,\vec\mu_2}(X_{k}',d,0)$ is related to the relative Welschinger invariants (see \cite[Section 5]{ab25}).
\end{remark}

\section{Floor diagrams relative to a conic from degeneration}\label{sec:4}
In this section, we prove a correspondence (Theorem \ref{thm:1}) by using the degeneration formula \cite{li2002} (see \cite{lr2001} for the symplectic setting). 
The degeneration argument that we use here was employed by Brugall\'e in \cite{bru2015} 
to compute higher genus Gromov-Witten invariants with only point insertions as well as Welschinger invariants. 
It is different from the one employed by Arg\"uz and Bousseau \cite[Section 4]{ab25}.

\subsection{Relative Gromov-Witten invariants}
\label{sec:rel-gw-bps}
Let $X$ be a smooth projective surface, and $Z$ be a smooth divisor of $X$.
Let $d\in H_2(X;\zb)$ be a curve class with $d\cdot [Z]\geq0$,
and $\mu=(\mu^{(i)})_{1\leq i\leq l(\mu)}$ be an unordered partition of $d\cdot [Z]$.
We choose an ordering of the entries of $\mu$,
and denote by $\vec\mu=(\vec\mu^{(i)})_{1\leq i\leq l(\vec\mu)}$ the corresponding ordered partition.

We recall the construction of moduli stack of relative stable maps
in detail from \cite{li2002} (see \cite{lr2001} for the symplectic setting).
Let $\Delta[r]$ be the surface with normal crossing singularity
obtained by gluing a length $r$
chain of $\pb^1$-bundles $\Delta_1,\ldots,\Delta_r$,
where $\Delta_i=\pb(\ol_Z\oplus N_{Z|X})$ for $i=1,\ldots,r$.
There are two distinguished sections $Z^-_i=\pb(\ol_Z\oplus {0})$
and $Z^+_i=\pb({0}\oplus N_{Z|X})$ of the $\pb^1$-bundle $\Delta_i$.
The $\pb^1$-bundles $\Delta_i$ and $\Delta_{i+1}$ are transversally glued together
along $Z^-_i\simeq Z^+_{i+1}$.
Denote by $X[r]$ the expanded degeneration of $X$ along $Z$
which is obtained from $X$ by gluing with $\Delta[r]$ transversally along
$Z\simeq Z^+_1$.

A genus $g$ degree $d$ \textit{relative map} to $X|Z$
with $n$ absolute marked points and $l(\vec\mu)$ relative marked points weighted by $\vec\mu$
is actually a map targeting
an expanded degeneration $X[r]|Z^-_r$ for some integer $r$:
$$
f:(C,x_1,\ldots,x_n,y_1,\ldots,y_{l(\mu)})\to X[r]|Z^-_r
$$
such that
\begin{itemize}
    \item $C$ is a connected nodal curve of arithmetic genus $g$;
    \item $x_1,\ldots,x_n,y_1,\ldots,y_{l(\vec\mu)}$ are distinct smooth points of $C$;
    \item $f^{-1}(Z^-_r)=\sum_{i=1}^{l(\vec\mu)}\vec\mu^{(i)}y_i$ and $\deg(f)=d$.
\end{itemize}
A relative map $f:C\to X[r]|Z^-_r$ is \textit{pre-deformable}
if the following holds:
\begin{enumerate}
    \item $f^{-1}(Z_i)$ is a discrete set for every $i=0,1,\ldots,r$, where $Z_0=Z^+_1$, and $Z_j=Z^-_j$ for $j\in\{1,2,\ldots,r\}$;
    \item Every point $p\in f^{-1}(Z_i)$ is a node of $C$
    and is the intersection of two irreducible components $C^-$, $C^+$ of $C$ such that $f(C^-)\subset\Delta_i$ and $f(C^+)\subset\Delta_{i+1}$, where $\Delta_0=X$.
    Moreover, the contact order of $f(C^-)$ along $Z_i$ at $p$ is
    equal to the contact order of $f(C^+)$ along $Z_i$ at $p$.
\end{enumerate}
A relative map $f:C\to X[r]|Z^-_r$ is \textit{stable}
if it is pre-deformable and stable when $f:C\to X[r]|Z^-_r$
is considered as a map whose domain is $C$ with $n+l(\vec\mu)$ marked points $x_1,\ldots,x_n,y_1,\ldots,y_{l(\vec\mu)}$.

Let $\overline M_{g,n}(X|Z,d,\vec\mu)$
denote the moduli space of relative stable maps of genus $g$ and degree $d$ to $X|Z$ with $n$ absolute marked points
and $l(\vec\mu)$ relative marked points weighted by $\vec\mu$.
The moduli space $\overline M_{g,n}(X|Z,d,\vec\mu)$
is a proper and separated Deligne-Mumford stack \cite{li2002}.
The virtual dimension of $\overline M_{g,n}(X|Z,d,\vec\mu)$ is
$$
\vdim \overline M_{g,n}(X|Z,d,\vec\mu)=\int_d c_1(X)+g-1+n-d\cdot [Z]+l(\vec\mu),
$$
where $c_1(X)$ is the first Chern class of $X$. Relative Gromov-Witten invariants of $X|Z$
are defined as follows:
\begin{equation}
\label{eq:relative-GW}
\langle\gamma;p_1,\ldots,p_n|\mu,\delta\rangle_{g,d}^{X|Z}:=
\frac{1}{|\aut(\mu,\delta^\circ)|}
\int_{[\overline M_{g,n}(X|Z,d,\vec\mu)]^{vir}}\gamma\cdot\prod_{i=1}^n ev_i^*(p_i)\cdot
\prod_{j=1}^{l(\mu)} (ev_j^Z)^*(\delta_j),
\end{equation}
where,
\begin{itemize}
    \item $[\overline M_{g,n}(X|Z,d,\vec\mu)]^{vir}$
    is the virtual fundamental class of the moduli stack of relative stable maps;
    \item $ev_i:\overline M_{g,n}(X|Z,d,\vec\mu)\to X$ 
    (resp. $ev_i^Z:\overline M_{g,n}(X|Z,d,\vec\mu)\to Z$) is the evaluation map defined by
    $ev_i([f:(C,x_1,\ldots,x_n,y_1,\ldots,y_{l(\mu)})\to X])=f(x_i)$
    (resp. $ev_j^Z([f:(C,x_1,\ldots,x_n,y_1,\ldots,y_{l(\mu)})\to X])=f(y_j)$).
    \item $p_1,\ldots,p_n\in H^{2*}(X;\qb)$, $\delta=(\delta_j)_{1\leq j\leq l(\mu)}$
    with $\delta_j\in H^{2*}(Z;\qb)$, and $\gamma$ is a cohomology class in $H^{2*}(\overline M_{g,n}(X|Z,d,\vec\mu);\qb)$.
    \item $\aut(\mu,\delta^\circ)$ is the group of permutation symmetries of the set of pairs $(\mu^{(j)},\delta_j)$ for all $j\in\{1,\ldots,l(\mu)\}$ such that $\delta_j=1\in H^0(Z;\qb)$.
\end{itemize}
Since we only consider cohomology classes with even degrees, the wedge product in the right-hand side of equation (\ref{eq:relative-GW}) commutes. Hence, the relative Gromov-Witten invariant in equation (\ref{eq:relative-GW})
does not depend on the order chosen to define $\vec\mu$,
and only depends on the unordered partition $\mu$.
The factor $\frac{1}{|\aut(\mu,\delta^\circ)|}$ before the
integration in (\ref{eq:relative-GW}) neutralizes the
effects of the ordering of the relative marked points
with given contact order and trivial cohomology class insertion.

Let $d\in H_2(X_k';\zb)$ be a curve class, and $E\subset X_k'$ be the strict transform of a smooth conic.
Let $\vec\mu=(\vec\mu_1,\vec\mu_2)=(\mu_1^{(1)},\ldots,\mu_1^{(l(\mu_1))},\mu_2^{(1)},\ldots,\mu_2^{(l(\mu_2))})$ 
be an ordered partition of $d\cdot[E]$.
We put
\begin{equation}
\label{equ:relative-GW1}
N_{g,d,n}^{X_{k}'|E}(m,\mu_1,\mu_2):=
\left<(-1)^m\lambda_m;p_1,\ldots,p_n|\mu,\delta\right>_{g,d}^{X_{k}'|E},
\end{equation}
where,
\begin{itemize}
    \item $\lambda_m=c_m(\eb)\in H^{2m}(\overline M_{g,n}(X_k'|E,d,\vec\mu);\qb)$,
    where $m\leq g$, $\eb$ is the Hodge bundle over $\overline M_{g,n}(X_k'|E,d,\vec\mu)$,
    and $c_m(\eb)$ is the $m$-th Chern class of $\eb$;
    \item $p_i\in H^4(X_k';\zb)$ is the cohomology class
    Poincar\'e dual to a point in $X_k'$ for any $1\leq i\leq n$;
    \item $\delta=(\delta_i)_{1\leq i\leq l(\mu)}$,
    where $\delta_i\in H^2(E;\zb)$ is the cohomology class Poincar\'e
    dual to a point on $E$ for any $1\leq i\leq l(\mu_1)$,
    and $\delta_j=1\in H^0(E;\zb)$ for all $l(\mu_1)+1\leq j\leq l(\mu)$;
    \item $n=d\cdot[L]+g-m-1+l(\mu_2)$.
\end{itemize}
When $m=g$, we use $N_{g,d,n}^{X_{k}'|E}(\mu_1,\mu_2)$
as the abbreviation of $N_{g,d,n}^{X_{k}'|E}(g,\mu_1,\mu_2)$.

\begin{remark}
\label{prop:vanishing-X_k1}
Assume that $n:=d\cdot[L]+g-m-1+l(\mu_2)=0$ and $d\cdot[E]\geq0$.
Since $d\cdot[L]+(g-m)+l(\mu_2)=1$,
we have $(d\cdot[L],g-m,l(\mu_2))=(1,0,0)$, $(0,1,0)$ or $(0,0,1)$.
Then relative Gromov-Witten invariants $N_{g,d,0}^{X_{k}'|E}(m,\mu_1,\mu_2)$ are all zero,
except for the above three cases.
\end{remark}

\begin{proposition}[{\cite[Lemma 4.12, Lemma 4.13]{ab25}}]\label{prop:vanishing-X_k2}
In the case $(d\cdot[L],g-m,l(\mu_2))=(1,0,0)$, we have
$$
\begin{aligned}
&\sum_{g\geq0}N_{g,[L],0}^{X_{k}'|E}((1,1),\emptyset)u^{2g}=u^{-1}2\sin(\frac{u}{2})=u^{-1}(-i)(q^\frac{1}{2}-q^{-\frac{1}{2}}),\\
&\sum_{g\geq0}N_{g,[L],0}^{X_{k}'|E}((2),\emptyset)u^{2g-1}=u^{-1}\cos(\frac{u}{2})=u^{-1}\frac{[2]_q}{2}, \\
&\sum_{g\geq0}N_{g,[L]-[E_i],0}^{X_{k}'|E}((1),\emptyset)u^{2g-1}=u^{-1}, \\
&\sum_{g\geq0}N_{g,[L]-[E_i]-[E_j],0}^{X_{k}'|E}(\emptyset,\emptyset)u^{2g-2}=\left(2u\sin(\frac{u}{2})\right)^{-1}=u^{-1}((-i)(q^\frac{1}{2}-q^{-\frac{1}{2}}))^{-1}.
\end{aligned}
$$
\end{proposition}

We consider the remaining two cases in Proposition \ref{prop:excep1}.

\begin{proposition}\label{prop:excep1}
When $(d\cdot[L],g-m,l(\mu_2))=(0,1,0)$ or $(0,0,1)$.
The relative Gromov-Witten invariants $N_{g,l[E_i],0}^{X_{k}'|E}(m,\mu_1,\mu_2)$ are all zero,
except for the following:
$$
\sum_{g\geq0} N_{g,l[E_i],0}^{X_{k}'|E}(\emptyset,(l))u^{2g-1}=
\frac{(-1)^{l-1}}{l}\frac{1}{2\sin(\frac{lu}{2})},
$$
where $i=1,\ldots,k$, and $l\geq1$.
\end{proposition}

\begin{proof}
We prove this proposition case by case.

\textbf{Case $(1)$:} $(d\cdot[L],g-m,l(\mu_2))=(0,1,0)$.

The curve class $d$ has to be $l[E_i]$
for some $i=1,\ldots,k$, where $l\geq1$,
and $m=g-1$. Moreover,
$\mu=\mu_1$ is a partition of $l=d\cdot[E]$, so $l(\mu_1)\geq1$.
We fix the position of $l(\mu_1)$ contact points in $E$.
Since the exceptional curve $E_i$ is rigid in $X_{k}'$,
every stable map of class $l[E_i]$ factors through $E_i$.
We can choose the position of the $l(\mu_1)$ contact points in $E\setminus E_i$.
but a curve of class $l[E_i]$ is contained in the exceptional divisor $E_i$.
Hence, the set of curves matching the constraints is empty,
and the corresponding invariant is zero.

\textbf{Case $(2)$:} $(d\cdot[L],g-m,l(\mu_2))=(0,0,1)$.

The curve class $d$ has to be $l[E_i]$
for some $i=1,\ldots,k$, where $l\geq1$.
Suppose that $\mu_2=(l_2)$ with $l_2\leq l$.
Since the exceptional curve $E_i$ is rigid in $X_{k}'$,
every stable map of class $l[E_i]$ factors through $E_i$.
The composition of a relative stable map to $E_i$ with the inclusion $E_i\subset X_{k}'$
defines a closed embedding
$$
\overline M_{g,0}(E_i|(E\cap E_i),l[E_i],\vec\mu)
\subset\overline M_{g,0}( X_k'|E,l[E_i],\vec\mu).
$$
Let $\pi:\cl\to\overline M_{g,0}(E_i|(E\cap E_i),l[E_i],\vec\mu)$
be the universal curve, and $f:\cl\to\pb^1$ the universal map.
The difference between the perfect obstruction theories for
curves mapping to the surface $X_{k}'$ and
for curves mapping to the curve $E_i\simeq\pb^1$ with normal bundle $\ol(-1)$ in $X_{k}'$
is the Euler class $e(R^1\pi_*f^*\ol(-1))$.
Moreover, $N_{g,l[E_i],0}^{X_{k}'|E}(\mu_1,(l_2))$ is non-zero only if $\mu_1=\emptyset$.
Hence, we obtain
\begin{equation}
\label{eq:non-vanishing1}
N_{g,l[E_i],0}^{X_{k}'|E}(\emptyset,(l_2))=
\int_{[\overline M_{g,0}(E_i|(E\cap E_i),l[E_i],\vec\mu)]^{vir}}
(-1)^g\lambda_g\cdot e(R^1\pi_*f^*\ol(-1)).
\end{equation}
The virtual dimension of the moduli stack
$\overline M_{g,0}(E_i|(E\cap E_i),l[E_i],\vec\mu)$
is $2g-2+l+l(\mu)$, while the complex dimension of the
integrand in equation (\ref{eq:non-vanishing1}) is $2g-1+l$.
The invariant $N_{g,l[E_i],0}^{X_{k}'|E}(\mu_1,(l_2))$ is non-zero
only if $\mu=\mu_2=(l)$.
It follows from \cite[Theorem 5.1]{bp-2005} that
$N_{g,l[E_i],0}^{X_{k}'|E}(\emptyset,(l))$ is the coefficient of $u^{2g-1}$ in
$$
\frac{(-1)^{l-1}}{l}\frac{1}{2\sin(\frac{lu}{2})}.
$$
\end{proof}

\subsection{Degeneration of the surface $X_k'$}
\label{sec:4.1}

Recall that $\nl$ is the $\pb^1$-bundle $\pb(\ol_E\oplus N_{E|X_{k}'})$, 
and $Z^-=\pb(\ol_E\oplus\{0\})$, $Z^+=\pb(\{0\}\oplus N_{E|X_{k}'})$ are the two distinguished sections of the $\pb^1$-bundle $\nl$. 
We choose a homology class $d\in H_2(X_k';\zb)$ such that $d\cdot[L]>0$ and $d\cdot[E]\geq0$. 
Let $\vec\mu=(\vec\mu_1,\vec\mu_2)$ be an ordered partition of $d\cdot [E]$, 
and $g$, $g'$ be two non-negative integers with $g'\leq g$.

Let $n=d\cdot[L]-1+g'+l(\mu_2)$. We assume that $n>0$.
By successive degeneration of $X_{k}'$ to the normal cone of $E$,
\textit{i.e.} the blow-up of $X_{k}'\times\ab^1$ along $E\times\{0\}$, we construct a flat morphism 
$\pi:\xl\to\ab^1$ satisfying the following conditions:
\begin{itemize}
    \item $\xl_t:=\pi^{-1}(t)=X_{k}'$ for $t\neq0$;
    \item the special fiber $\xl_0:=\pi^{-1}(0)$ is a chain of $X_{k}'$ and $n$ copies of $\pb^1$-bundle $\nl=\pb(\ol_E\oplus N_{E|X_{k}'})$:
    $$
    \pi^{-1}(0)=X_{k}'\sideset{_E}{_{Z_n^+}}{\mathop{\cup}} \nl_{n}\sideset{_{Z_n^-}}{_{Z^+_{n-1}}}{\mathop{\cup}}\nl_{n-1}\sideset{_{Z^-_{n-1}}}{_{Z^+_{n-2}}}{\mathop{\cup}}\cdots\sideset{_{Z^-_2}}{_{Z^+_1}}{\mathop{\cup}}\nl_1,
    $$
    where $Z^-_i=\pb(\ol_E\oplus\{0\})$ and $Z^+_i=\pb(\{0\}\oplus N_{E|X_{k}'})$ are the two distinguished sections in $\nl_i=\pb(\ol_E\oplus N_{E|X_{k}'})$, $i=1,\ldots,n$.
    Moreover, in the special fiber $\pi^{-1}(0)$, $X_{k}'$ intersects $\nl_n$ transversely along $E\simeq Z^+_n$ and is disjoint with $\nl_i$ for $1\leq i<n$, $\nl_j$ intersects $\nl_{j-1}$ transversely along $Z^-_j\simeq Z_{j-1}^+$ and $\nl_j\cap \nl_l=\emptyset$ when $|j-l|>1$. 
\end{itemize}

\subsection{Decorated weighted graphs from degeneration of curves}

To compute the relative invariant $N_{g,d,n}^{X_{k}'|E}(g-g',\mu_1,\mu_2)$, we precisely state 
the degeneration formula in relative Gromov-Witten theory \cite{li2002,lr2001} 
that is applied to the degeneration $\pi:\xl\to\ab^1$. Here, 
we distribute the $n$ point classes appearing in (\ref{equ:relative-GW1}) 
by placing one on each of the $n$ components $\nl_1,\ldots,\nl_n$.
Weighted decorated graphs, which are dual graphs of stable maps, 
are used to index the terms in the degeneration formula and 
to describe the degeneration types of curves in the special fiber.

\begin{definition}
\label{def:4.1}
We denote by $\gl_{g,g',d,n}^{\vec\mu_1,\vec\mu_2}$ the set of connected decorated weighted graphs $\Gamma$ given below.
\begin{enumerate}
    \item $\Gamma$ is a finite connected weighted graph. $l(\vec\mu_1)$ (resp. $l(\vec\mu_2)$) leaves of $\Gamma$ are decorated by the entries in the partition $\vec\mu_1$ (resp. $\vec\mu_2$), respectively. Denote by $\vt^{\vec\mu_1}(\Gamma)$ and $\vt^{\vec\mu_2}(\Gamma)$ the sets of leaves decorated by $\vec\mu_1$ and $\vec\mu_2$, respectively.
    \item Every vertex $v\in\vt^\circ(\Gamma)=\vt(\Gamma)\setminus(\vt^{\vec\mu_1}(\Gamma)\cup\vt^{\vec\mu_2}(\Gamma))$ is decorated by a genus $g_v\in\zb_{\geq0}$ such that $g_\Gamma+\sum\limits_{v\in\vt^\circ(\Gamma)}g_v=g$, where $g_\Gamma$ is the first Betti number of the graph $\Gamma$.
    \item Every vertex $v\in\vt^\circ(\Gamma)$ has an index decoration $j_v\in\{1,\ldots,n+1\}$ and a cohomology class decoration. Moreover, an inner vertex $v$ with index $1\leq j_v\leq n$ is decorated by $d_v=h_v[Z^+_{j_v}]+l_v[F_{j_v}]\in H_2(N_{j_v};\zb)$. Every vertex $v$ with $j_v=n+1$ is decorated by $d_v=a_v[L]-\sum_{i=1}^kb_v^i[E_i]\in H_2(X_k';\zb)$.
    \item For every $j\in\{1,\ldots,n\}$, there is a distinguished vertex among the vertices with index $j_v=j$. We denote the distinguished vertex with index $j$ by $v_j$.
    \item Every inner edge $e\in\eg(\Gamma)$ is decorated by an index $j_e\in\{1,\ldots,n\}$, and every end is decorated by an index $j_e=0$ or $n$.
    Moreover, the $l(\vec\mu_1)+l(\vec\mu_2)$ ends adjacent to the leaves in $\vt^{\vec\mu_1}(\Gamma)\cup\vt^{\vec\mu_2}(\Gamma)$ are exactly the ends decorated by $j_e=0$. These ends are weighted by the entry decorations of their leaves. If $e$ is an edge with $0\leq j_e\leq n$, the two endpoints of $e$ are indexed by $j_e,j_{e}+1$.
    \item Let $v$ be a vertex with index $j_v=n+1$ and $d_v=a_v[L]-\sum_{i=1}^kb_v^i[E_i]$. Assume that there are $r$ edges $e_1,\ldots,e_r$ adjacent to $v$.
    Then the tuple of weights of edges $e_1,\ldots,e_r$ is a partition of $2a_v-\sum_{i=1}^kb_v^i$. 
    Let $v$ be an inner vertex with $1\leq j_v\leq n$. Assume that there are $s$ edges $e'_1,\ldots,e'_s$ incident to $v$ with $j_{e'_i}=j_v$ and $t$ edges $e''_1,\ldots,e''_t$ incident to $v$ with $j_{e''_i}=j_v-1$. Then the tuple of weights of edges $e'_1,\ldots,e'_s$ is a partition of $(k-4)h_v+l_v$, and the tuple of weights of edges $e''_1,\ldots,e''_t$ is a partition of $l_v$.
    \item Every half-edge, \textit{i.e.} a pair $(v,e)$ with $v\in\vt(\Gamma)$ and $e\in\eg(\Gamma)$ adjacent to $v$, is decorated by a cohomology class $c_{(v,e)}\in H^*(Z_{j_e};\zb)$. If $e$ is an end adjacent to a leaf in $\vt^{\vec\mu_2}(\Gamma)$ (resp. $\vt^{\vec\mu_1}(\Gamma)$), the cohomology class $c_{(v,e)}=pt\in H^2(Z_{j_e};\zb)$ (resp. $c_{(v,e)}=1\in H^0(Z_{j_e};\zb)$), where $pt$ is the cohomology class Poincar\'e dual to a point in $Z_{j_e}$. In the case that $e$ is an end with $j_e=n$, and $v$ is indexed by $j_v=n+1$ and $d_v\neq b_v^i[E_i]$ (resp. $d_v=b_v^i[E_i]$), the cohomology class $c_{(v,e)}=pt\in H^2(Z_{j_e};\zb)$ (resp. $c_{(v,e)}=1\in H^0(Z_{j_e};\zb)$). In the remaining cases, for exactly one vertex $v$ incident to $e$ we have $c_{(v,e)}=1$, and for the other vertex $v'$ incident to $e$, we have $c_{(v',e)}=pt$.
    \item Every vertex $v\in\vt^\circ(\Gamma)$ is decorated by an index $0\leq m_v\leq g_v$ such that $\sum_{v}m_v=g-g'$.
\end{enumerate}
\end{definition}

We refer the readers to the proof of Proposition \ref{prop:4.1}, \cite[Lemma 5.5]{bou2021}
and \cite[Proposition 5.1]{bru2015} for the geometric interpretation of the graphs in $\gl^{\vec\mu_1,\vec\mu_2}_{g,g',d,n}$.

\begin{definition}
\label{def:4.2}
Let $\Gamma\in\gl^{\vec\mu_1,\vec\mu_2}_{g,g',d,n}$ and $v\in\vt^\circ(\Gamma)$. We put
$$
\mu_v:=(w_e)_{e\in E_{v,+}}, \nu_v:=(w_e)_{e\in E_{v,-}}, \delta_v^2:=(c_{v,e})_{e\in E_{v,+}}, \delta_v^1:=(c_{v,e})_{e\in E_{v,-}},
$$
where $E_{v,+}$ (resp. $E_{v,-}$) is the set of edges adjacent to $v$ such that $j_e=j_v$ (resp. $j_e=j_v-1$).
We define some relative Gromov-Witten invariants of $X_{k}'$ and $\nl$ as follows.
\begin{equation}
\label{eq:4.1}
N_{\Gamma,v}:=\left\{
\begin{array}{ll}
\langle(-1)^{m_v}\lambda_{m_v}|\nu_v,\delta_v^1\rangle_{g_v,d_v}^{X_k'/E}& \text{ if } j_v=n+1, \\
\langle\mu_v,\delta_v^2|(-1)^{m_v}\lambda_{m_v}|\nu_v,\delta_v^1\rangle_{g_v,d_v}^{\nl_{j_v}/Z^+_{j_v}\cup Z^-_{j_v}} & \text{ if } 1\leq j_v\leq n \text{ and } v\neq v_{j_v},\\
\langle\mu_v,\delta_v^2|(-1)^{m_v}\lambda_{m_v};p_{j_v}|\nu_v,\delta_v^1\rangle_{g_v,d_v}^{\nl_{j_v}/Z^+_{j_v}\cup Z^-_{j_v}} & \text{ if } 1\leq j_v\leq n \text{ and } v= v_{j_v}.
\end{array}
\right.
\end{equation}
\end{definition}

\begin{proposition}
\label{prop:4.1}
The relative Gromov-Witten invariants $N_{g,d,n}^{X_{k}'|E}(g-g',\mu_1,\mu_2)$ of $X_k'$ are given as follows.
\begin{equation}
\label{eq:4.2}
N_{g,d,n}^{X_{k}'|E}(g-g',\mu_1,\mu_2)=\sum_{\Gamma\in\gl^{\vec\mu_1,\vec\mu_2}_{g,g',d,n}}\frac{\prod_{e\in\eg^\circ(\Gamma)}w_e}{|\aut(\Gamma)|}\prod_{v\in\vt^\circ(\Gamma)}N_{\Gamma,v},
\end{equation}
where $\aut(\Gamma)$ is the group of permutation symmetries of $\Gamma$ as decorated weighted graph, and $\eg^\circ(\Gamma)$ is the set of edges of $\Gamma$ except the ends adjacent to vertices in $\vt^{\vec\mu_1}(\Gamma)\cup\vt^{\vec\mu_2}(\Gamma)$.
\end{proposition}

\begin{proof}
Equation (\ref{eq:4.2}) is the precise form of the degeneration formula for relative Gromov-Witten invariants \cite{li2002} (see also \cite{lr2001}) applied to the degeneration $\pi:\xl\to\ab^1$ in Section \ref{sec:4.1}, where we distribute the $n$ point classes by placing one on each of the $n$ components $\nl_1,\ldots,\nl_n$.
Let $\Gamma$ be a decorated weighted graph in $\gl^{\vec\mu_1,\vec\mu_2}_{g,g',d,n}$.
The graph $\Gamma\setminus(\vt^{\vec\mu_1}(\Gamma)\cup\vt^{\vec\mu_2}(\Gamma))$ is the dual graph of a stable map in the special fiber $\pi^{-1}(0)$.
The proof of this proposition is the same as the proof of \cite[Lemma 5.5]{bou2021}, so we only explain the meaning of leaves in $\vt^{\vec\mu_1}(\Gamma)\cup\vt^{\vec\mu_2}(\Gamma)$ here and refer the readers to the proof of \cite[Lemma 5.5]{bou2021} for the rest of the proof. See \cite[Proposition 5.1]{bru2015} for the case without Lambda class.

Let $e$ be an end with $j_e=0$ and $v\in\vt^{\vec\mu_1}(\Gamma)\cup\vt^{\vec\mu_2}(\Gamma)$ 
be the leave adjacent to $e$. Then $v$ and $e$ correspond to a relative marking on $Z^-_1$. 
An end $e'$ with $j_{e'}=n$ corresponds to a node of the source curve 
which is mapped to the divisor $E\simeq Z_{n}^+$ by a relative stable map.
A vertex $v'$ with $j_{v'}=n+1$ corresponds to a curve of genus $g_{v'}$ and class $d_{v'}$ contained in $X_k'$. 
The rest is the same as the proof of \cite[Lemma 5.5]{bou2021}, so we omit it.
\end{proof}

The following proposition is essentially \cite[Corollary 5.4]{bru2015}. 
We just modify \cite[Corollary 5.4]{bru2015} and its proof to fit our setting.
\begin{proposition}[{\cite[Corollary 5.4]{bru2015}}]
\label{prop:4.2}
Suppose that $d\neq l[E_i]$ with $l\geq2$. Let $v\in\Gamma$ be a leaf with $j_v=n+1$, $d_v=b_v^i[E_i]$, $g_v=m_v$ and $c_{(v,e)}=1$ for the edge $e$ adjacent to $v$, where $i\in\{1,\ldots,k\}$. If the relative Gromov-Witten invariant $N_{\Gamma,v}$ has non-trivial contribution to equation $(\ref{eq:4.2})$, the homology class $d_v=[E_i]$ for some $i$.
\end{proposition}

\begin{proof}
Let $N_{\Gamma,v}$ be a non-zero factor appearing in the right-hand of equation (\ref{eq:4.2}), 
and suppose that the homology class $d_v=b_v^i[E_i]$ for some $i\in\{1,\ldots,k\}$.
It follows from Proposition \ref{prop:excep1} that $g_v=m_v$ and $c_{(v,e)}=1$ for the edge $e$ adjacent to $v$.
Now we show that $b_v^i=1$. Let 
$$
f_t:(C_t,x_1(t),\ldots,x_n(t),y_1(t),\ldots,y_{l(\mu_1)}(t),\ldots,y_{l(\mu)}(t))\to \xl_t[r]|Z_r^-
$$
be a relative stable map defining a point in the moduli space $\overline M_{g,n}(\xl_t|E,d,\vec\mu)$. 
When $t\to0$, the relative stable map $f_t$ degenerates to $f_0$ 
that consists of several relative stable maps $f_0^1,\ldots,f_0^m$ to expanded spaces of $X_k'$, or $\nl_n$, $\dots$, or $\nl_1$. 
Suppose that $f_0$ contains $a_i$ relative stable maps 
$f_0^{s_j}:C_0^{s_j}\to X_k'[r_{j}]|Z_{r_{j}}^-$ in $\overline{M}_{g,0}(X_k'|E,l_j[E_i],(l_j))$, where $j=1,\ldots,a_i$. 
From $l_j\geq1$, we get $a_i\leq\sum_{j=1}^{a_i}l_j$.
Since the exceptional curve $E_i$ is rigid in $X_k'$, every stable map of class $l[E_i]$ factors through $E_i$.
Let $C_0^\nl$ denote the union of all irreducible components of $C_0:=\lim\limits_{t\to0}C_t$ that are mapped to the expanded space of $\nl_n,\ldots,\nl_1$ by relative stable maps in $f_0$.
We denote by $C_0^{X_k'}$ the union of irreducible components of $C_0$ 
that are mapped to expanded space of $X_k'$ by relative stable maps in $f_0$, 
and the homology classes they represent are not multiples of any $[E_i]$.

Suppose that $d=a[L]-\sum_{j=1}^kb_j[E_j]$ and $(\pi^-)_*\circ(f_0)_*([C_0^\nl])=c[Z_n^+]$,
where $\pi^-:\nl_n[t_n]\cup\cdots\cup\nl_1[t_1]\to Z_{n}^+$ is the projection from the expanded spaces to the section $Z_n^+$. 
Then we have
$$
(\pi^+)_*\circ(f_0)_*([C_0^{X_k'}])=(a-2c)[L]-\sum_{j\neq i}(b_j-c)[E_j]-(b_i-c+\sum_{m=1}^{a_i}l_m)[E_i],
$$
where $\pi^+:X_k'[r]\to X_k'$ is the projection from the expanded space of $X_k'$ to $X_k'$.
In the degeneration $\xl_t\to\xl_0$, the exceptional curve $E_i$ degenerates to the union $\el$ of $E_i$ 
and the fiber in $\nl_n[t_n]\cup\cdots\cup\nl_1[t_1]$ passing through $E_i\cap E$. 
The intersection point of two different irreducible components in $C_0$
are mapped to a transversal intersection point in $f_0(C_0)$.
We denote by $N(f_0)$ the set of such transversal intersection points in $f_0(C_0)$.
We consider the sum of multiplicity of intersections over intersection points of $f_0(C_0^\nl\cup C_0^{X_k'})$ 
with $\el$ and not in the set $N(f_0)$.
This number is equal to
$$
b_i-c_1+\sum_{m=1}^{a_i}l_m+(c_1-a_i)=b_i+\sum_{m=1}^{a_i}l_m-a_i.
$$
Note that the $-a_i$ in the above formula corresponds to the $a_i$ intersection points of the irreducible components in $C_0$ mapped to $\el$ with the irreducible components in $C_0^\nl\cup C_0^{X_k'}$.
On the other hand, all these intersections are deformed from $\xl_t$, hence we have
$$
d\cdot[E_i]=b_i\geq b_i+\sum_{m=1}^{a_i}l_m-a_i.
$$
Therefore, we obtain that $a_i=\sum_{m=1}^{a_i}l_m$, and the conclusion is proved.
\end{proof}

\subsection{Floor diagrams relative to a conic from decorated weighted graphs}
\label{subsec:4.3}
In this section, we apply the vanishing results about relative Gromov-Witten invariants 
to the degeneration formula (\ref{eq:4.2}) and obtain floor diagrams relative to a conic from decorated weighted graphs.

Recall that we always assume $d\cdot[L]>0$, $d\cdot[E]\geq0$ and $n=d\cdot[L]-1+g'+l(\mu_2)>0$.
We describe the vertices in a decorated weighted graph $\Gamma\in\gl^{\vec\mu_1,\vec\mu_2}_{g,g',d,n}$ corresponding to non-vanishing relative invariants as follows.
\begin{definition}
\label{def:4.3}
Let $\Gamma\in\gl^{\vec\mu_1,\vec\mu_2}_{g,g',d,n}$ be a decorated weighted graph, and $v\in\vt^\circ(\Gamma)$ be a vertex in $\Gamma$. The vertex $v$ is called \textit{effective} if $v$ is characterized by one of the following cases.
\begin{enumerate}
    \item $j_v=n+1$, $d_v=[L]$, $\nu_v=(1,1)$ or $(2)$, $m_v=g_v=0$. For any edge $e$ adjacent to $v$, the half-edge $(v,e)$ has $c_{(v,e)}=pt$.
    \item $j_v=n+1$, $d_v=[L]-[E_i]$ for some $i\in\{1,\ldots,k\}$, $\nu_v=(1)$, $m_v=g_v=0$. For the edge $e$ adjacent to $v$, the half-edge $(v,e)$ has $c_{(v,e)}=pt$.
    \item $j_v=n+1$, $d_v=[E_i]$ for some $i\in\{1,\ldots,k\}$, $\nu_v=(1)$, $m_v=g_v$. For the edge $e$ adjacent to $v$, the half-edge $(v,e)$ has $c_{(v,e)}=1$.
    \item $1\leq j_v\leq n$, $v\neq v_{j_v}$, $d_v=l_v[F_{j_v}]$, $\mu_v=\nu_v=(l_v)$, one of the edges incident to $v$ has $c_{(v,e)}=1$ and the other has $c_{(v,e')}=pt$, and $m_v=g_v=0$.
    \item $1\leq j_v\leq n$, $v=v_{j_v}$, $d_v=l_v[F_{j_v}]$, $\mu_v=\nu_v=(l_v)$, both edges incident to $v$ have $c_{(v,e)}=1$, and $m_v=g_v=0$.
    \item $1\leq j_v\leq n$, $v=v_{j_v}$, $d_v=[Z_{j_v}^+]+l_v[F_{j_v}]$, both edges incident to $v$ have $c_{(v,e)}=pt$, and $m_v=g_v$.
\end{enumerate}
\end{definition}
Denote by $V_a(\Gamma)$ the set of effective vertices of $\Gamma$ characterized by case $(a)$, where $a=1,\ldots,6$. In particular, let $V_2^i\subset V_2$ be the subset of $V_2$ consisting of vertices decorated by the class $d_v=[L]-[E_i]$ for some $i\in\{1,\ldots,k\}$, and let $V_3^i\subset V_3$ be the subset of $V_3$ consisting of vertices decorated by the class $d_v=[E_i]$ for some $i\in\{1,\ldots,k\}$. 
Denote by $V_1^i(\Gamma)$ the subset of vertices in $V_1(\Gamma)$ that are adjacent to weighted $i$ edges, $i=1,2$.
Let $\overline\gl^{\vec\mu_1,\vec\mu_2}_{g,g',d,n}$ be the subset of decorated weighted graphs $\Gamma\in\gl^{\vec\mu_1,\vec\mu_2}_{g,g',d,n}$ such that the vertices in $\vt^\circ(\Gamma)$ are all effective vertices. 

\begin{lemma}
\label{lem:4.01}
Let $\Gamma$ be a decorated weighted graph in $\overline\gl^{\vec\mu_1,\vec\mu_2}_{g,g',d,n}$. 
Then, we have
$$
d\cdot[E_i]=|V_2^i|+|V_6|-|V_3^i|.
$$
\end{lemma}

\begin{proof}
The equality follows from the following relation between homology classes:
$$
d=|V_1|[L]+\sum_{i=1}^k|V_2^i|([L]-[E_i])+\sum_{v\in V_3}d_v+|V_6|[E].
$$
\end{proof}

\begin{lemma}
\label{lem:4.02}
Let $\Gamma$ be a decorated weighted graph in $\overline\gl^{\vec\mu_1,\vec\mu_2}_{g,g',d,n}$. 
Let $c$ be a chain of edges in $\Gamma$ such that the endpoints of $c$ are not in $V_4\cup V_5$
and the edges in $c$ are connected by bivalent effective vertices in $V_4$ or $V_5$.
Then, we have:
\begin{enumerate}
    \item[$(1)$] If the endpoints of $c$ are in one of the following two cases:
    \begin{itemize}
        \item one is in $V_6$ and the other is in $V_1\cup V_2\cup V_6$;
        \item one is in $\vt^{\vec\mu_2}(\Gamma)$
        and the other is in $V_1\cup V_2\cup V_6$;
    \end{itemize}
    the chain $c$ contains exactly one vertex in $V_5$.
    \item[$(2)$] If the endpoints of $c$ do not satisfy any condition in the case $(1)$, 
    the chain $c$ does not contain a vertex in $V_5$.
    \item[$(3)$] If the chain $c$ has one endpoint in $V_3$, then the other endpoint of $c$ is in $V_6$, and $c$ is a chain consisting of weighted 1 edges.
    Moreover, different vertices in $V_3^i$ are connected to different vertices in $V_6$ by such a chain $c$.
\end{enumerate}
\end{lemma}

\begin{proof}
The proof of Lemma \ref{lem:4.02} follows from the degeneration formula $(\ref{eq:4.2})$ and the non-vanishing results about relative Gromov-Witten invariants: Remark \ref{prop:vanishing-X_k1}, Proposition \ref{prop:vanishing-X_k2}, Proposition \ref{prop:excep1}, \cite[Lemma 5.1]{bou2021} and \cite[Lemma 5.2]{bou2021}. The detailed analysis is the same as the analysis in \cite[Section 5.2]{bru2015} or the proof of \cite[Lemma 5.7]{bou2021}, so we omit the proof of items $(1)$, $(2)$ and refer the readers to \cite[Section 5.2]{bru2015} or the proof of \cite[Lemma 5.7]{bou2021}.

Let $c$ be a chain having an endpoint in $V_3$. If any vertex in $V_6$ is not an endpoint of $c$, the other endpoint of $c$ must be in $\vt^{\vec\mu_2}(\Gamma)$. 
By the statement $(2)$, $c$ does not contain a vertex in $V_5$. 
Note that we assume $d\cdot[L]>0$. Hence, there must be other vertices and edges disjoint with $c$ in $\Gamma$. This contradicts the fact that $\Gamma$ is connected. The weight of the edges in $c$ is obtained by Proposition \ref{prop:4.2}.
Suppose that two different vertices in $V_3^i$ are connected to a same vertices $v$ in $V_6$.
Then there are two ones in the partition $\mu_v$ (see Definition \ref{def:4.2}) corresponding to a same intersection point in $Z^+_v$.
This contradicts to the degeneration formula \cite{li2002}. Hence,
different vertices in $V_3^i$ are connected to different vertices in $V_6$ by a chain.
\end{proof}

From the degeneration formula, non-vanishing results about relative Gromov-Witten invariants (Remark \ref{prop:vanishing-X_k1}, Proposition \ref{prop:vanishing-X_k2}, Proposition \ref{prop:excep1}, \cite[Lemma 5.1]{bou2021} and \cite[Lemma 5.2]{bou2021}), and Proposition \ref{prop:4.2}, one knows that only graphs in $\overline\gl^{\vec\mu_1,\vec\mu_2}_{g,g',d,n}$ have non-trivial contribution to the formula (\ref{eq:4.2}).

Floor diagrams relative to a conic were constructed by Brugall\'e in \cite[Section 5.2]{bru2015}. 
Now, we adapt Brugall\'e's construction to our setting. 
Let $\Gamma$ be a decorated weighted graph in $\overline\gl^{\vec\mu_1,\vec\mu_2}_{g,g',d,n}$.
For every $1\leq i\leq k$, we denote by $A_i(\Gamma)$ the set composed of couples $(v,i)$ such that either $v\in V_2^i$, or $v\in V_6$ and for any vertex $v'\in V_3^i$, the two vertices $v,v'$ are not endpoints of any chains of edges in $\Gamma$ connected by bivalent vertices in $V_4$.

\begin{construction}[{\cite[Section 5.2]{bru2015}}]
\label{const:1}
Let $\Gamma$ be a graph in $\overline\gl^{\vec\mu_1,\vec\mu_2}_{g,g',d,n}$. We use Brugall\'e's construction to give an oriented weighted graph $\dl_\Gamma$ as follows:
\begin{itemize}
    \item Vertices in $\vt^\infty(\dl_\Gamma)$ are ono-to-one correspondence with elements in the set $\left(\cup_{i=1}^kA_i(\Gamma)\right)\cup\vt^{\vec\mu_1}(\Gamma)\cup\vt^{\vec\mu_2}(\Gamma)$. Vertices corresponding to $(v,i)\in\cup_{i=1}^kA_i(\Gamma)$ are denoted by $v_{(v,i)}$, and the vertices corresponding to $v\in\vt^{\vec\mu_1}(\Gamma)\cup\vt^{\vec\mu_2}(\Gamma)$ are still denoted by $v$.
    \item $\vt(\dl_\Gamma)\setminus\vt^\infty(\dl_\Gamma)=V_1\cup V_2\cup V_6$.
    \item Edges in $\eg^\infty(\dl_\Gamma)$ are one-to-one correspondence with vertices in $\vt^\infty(\dl_\Gamma)$. Moreover, we have the followings.
    \begin{itemize}
        \item For every $i\in\{1,\ldots,k\}$, the edge $e_{(v,i)}$ corresponding to a couple $(v,i)\in A_i(\Gamma)$ is adjacent to $v_{(v,i)}$ and the vertex $v$. The edge $e_{(v,i)}$ is oriented from $v_{(v,i)}$ to $v$.
        The weight of $e_{(v,i)}$ is equal to $1$.
        \item Every edge $e_v$ corresponding to a vertex $v\in\vt^{\vec\mu_1}(\Gamma)$ (resp. $v\in\vt^{\vec\mu_2}(\Gamma)$) is adjacent to vertices $v$ and $v'$, where $v'$ is the unique vertex in $V_1\cup V_2\cup V_6$ which is connected to $v$ by a chain of edges in $\Gamma$ connected by bivalent effective vertices in $V_4$ (resp. $V_4\cup V_5$). The edge $e_v$ is oriented from $v$ to $v'$, and the weight of $e_v$ is equal to the weight of the leaf of $\Gamma$ adjacent to $v$.
    \end{itemize}
    \item Let $B$ be the set of chains of edges in $\Gamma$ connected by bivalent effective vertices in $V_4\cup V_5$ and with endpoints $v,v'\in\vt(\dl_\Gamma)\setminus\vt^{\infty}(\dl_\Gamma)$.
    Edges in $\eg(\dl_\Gamma)\setminus\eg^\infty(\dl_\Gamma)$ are one to one correspondence with chains in $B$. Let $e$ be an edge corresponding to a chain $c\in B$ with endpoints $v$ and $v'$, then $e$ is adjacent to $v$ and $v'$. If $j_v>j_{v'}$, the edge $e$ is oriented from $v'$ to $v$. Otherwise, $e$ is oriented from $v$ to $v'$. The weight of $e$ is the weight of the edges in the chain $c$.
\end{itemize}
\end{construction}
\noindent We denote by $\eg^{\vec\mu}(\dl)=\eg^{\vec\mu_1}(\dl)\cup\eg^{\vec\mu_2}(\dl)$ the set of edges in $\eg^\infty(\dl)$ which are adjacent to vertices in $\vt^{\vec\mu_1}\cup\vt^{\vec\mu_2}$.

\begin{lemma}[{\cite[Lemma 5.10]{bru2015}}]
\label{lem:4.1}
For every $\Gamma\in\overline\gl^{\vec\mu_1,\vec\mu_2}_{g,g',d,n}$, the oriented weighted graph $\dl_\Gamma$ constructed in Construction $\ref{const:1}$ is a floor diagram of degree $d\cdot[L]$ and genus $g'$.
\end{lemma}

\begin{proof}
We only need to compute the first Betti number $b_1(\dl_\Gamma)$ and $\sum_{v\in\vt^\infty(\dl_\Gamma)}\dv(v)$, since the other properties of a floor diagram follow immediately from Construction $\ref{const:1}$. From Lemma \ref{lem:4.01}, we have
$$
\sum_{v\in\vt^\infty(\dl_\Gamma)}\dv(v)=-|\mu_1|-|\mu_2|-|V_2|-k|V_6|-\sum_{v\in V_3}d_v\cdot(\sum_{i=1}^k[E_i])=-2d\cdot[L].
$$
From the Definition \ref{def:4.1}, $g_\Gamma+\sum_{v\in\vt^\circ(\Gamma)}g_v=g$.
From Construction $\ref{const:1}$, effective vertices in $V_4(\Gamma)\cup V_5(\Gamma)$ are not included in $\dl_\Gamma$,
and $-|\vt(\dl_\Gamma)|+|\eg(\dl_\Gamma)|=-|\vt(\Gamma)|+|\eg(\Gamma)|$.
Hence, we obtain that
$$
g_{\dl_\Gamma}=-|\vt(\dl_\Gamma)|+|\eg(\dl_\Gamma)|+1=-|\vt(\Gamma)|+|\eg(\Gamma)|+1=g_\Gamma=g'.
$$
\end{proof}

The floor diagram $\dl_\Gamma$ of genus $g'$ and degree $d\cdot[L]$ admits a marking of type $(\vec\mu_1,\vec\mu_2)$. Recall that $d\in H_2(X_k';\zb)$ is a homology class, and $\vec\mu=(\vec\mu_1,\vec\mu_2)$ is a partition of $d\cdot[E]$.
We put $A_0=\{1,\ldots,n+l(\mu_1)\}$, where $n=d\cdot[L]-1+g'+l(\mu_2)$. Let $A_1,\ldots,A_k$ be some disjoint sets such that $|A_i|=d\cdot[E_i]$ for any $i\in\{1,\ldots,k\}$.

\begin{definition}
\label{def:4.4}
We define a map $m_{\Gamma}:A_0\cup\left(\cup_{i=1}^kA_i\right)\to\dl_\Gamma$ as follows:
\begin{itemize}
    \item For each $i\in A_0$ and $i\geq l(\mu_1)+1$, if the distinguished vertex $v_{i-l(\mu_1)}$ is contained in $V_6$ we set $m_{\Gamma}(i)=v_{i-l(\mu_1)}$, otherwise, we set $m_{\Gamma}(i)$ to be the edge in $\dl$ which corresponds to the unique chain $c$ of $\Gamma$ containing the distinguished vertex $v_{i-l(\mu_1)}$.
    \item The restriction of $m_\Gamma$ on $A_i$ is a bijection to the set 
    $\{e_{(v,i)}|(v,i)\in A_i(\Gamma)\}$ for every $i\in\{1,\ldots,k\}$.
    \item For $1\leq j\leq l(\mu_1)$, we set $m_\Gamma(j)=\bar v$, where $\bar v\in\vt^{\vec\mu_1}(\Gamma)$ is the unique leaf in $\Gamma$ which is labeled by $\mu_1^j$.
\end{itemize}
\end{definition}

\begin{lemma}[{\cite[Lemma 5.11]{bru2015}}]
\label{lem:4.2}
For every floor diagram $\dl_\Gamma$ of degree $d\cdot[L]$ and genus $g'$, the map
$m_{\Gamma}:A_0\cup\left(\cup_{i=1}^kA_i\right)\to\dl_\Gamma$ defined in Definition $\ref{def:4.4}$ is a $d$-marking of $\dl_\Gamma$ of type $(\vec\mu_1,\vec\mu_2)$.
\end{lemma}

\begin{proof}
It is a direct consequence of Lemma \ref{lem:4.01}, Lemma \ref{lem:4.02}, Definition \ref{def:4.1} and Construction \ref{const:1}.
\end{proof}

\begin{definition}\label{def:additional}
A edge in a $d$-marked floor diagram $(\dl,m)$ is called an \textit{additional end} if it is in the image $m(\cup_{i=1}^kA_i)$, 
otherwise, it is called an \textit{original edge}. For any $v\in V_6$, we denote by $i_v$ the number of $i\in\{1,\ldots,k\}$ such that $(v,i)\notin A_i(\Gamma)$.
\end{definition}

\subsection{Relative Gromov-Witten invariants and refined counts}
In this section, we first express the relative Gromov-Witten invariants $N_{g,d,n}^{X_{k}'|E}(g-g',\mu_1,\mu_2)$ in terms of floor diagrams whose vertices are weighted by relative Gromov-Witten invariants of Hirzebruch surfaces $\nl$ and $X_k'$.
Then we show that the $q$-refined counts of floor diagrams are, after the change of variables $q=e^{iu}$, generating series of higher genus relative Gromov-Witten invariants with insertion of a Lambda class.

We put
\begin{equation}\label{eq:def-rel}
N_{g,[Z^+]+l[F]}^{\mu,\nu}:=\langle\mu,\delta^1|(-1)^{g}\lambda_{g};p|\nu,\delta^2\rangle_{g,[Z^+]+l[F]}^{\nl|(Z^+\cup Z^-)},
\end{equation}
where $p\in H^4(\nl;\zb)$ is Poincar\'e dual to a point in $\nl$, $\delta^1_i,\delta^2_j$ are Poincar\'e dual to points in $Z^+$ and $Z^-$, respectively.

\begin{lemma}
\label{lem:4.3}
Let $g'\in\zb_{\geq0}$, and $d\in H_2(X_k';\zb)$ be a homology class such that $d\cdot[L]>0$, $d\cdot[E]\geq0$. Suppose that $\vec\mu=(\vec\mu_1,\vec\mu_2)$ is an ordered partition of $d\cdot[E]$, 
and let $n:=d\cdot[L]-1+g'+l(\mu_2)$. 
Then relative Gromov-Witten invariants $N_{g,d,n}^{X_{k}'|E}(g-g',\mu_1,\mu_2)$ of $X_k'$ are given by:
\begin{equation}
\label{eq:4.21}
\begin{aligned}
&\sum_{g\geq g'}N_{g,d,n}^{X_{k}'|E}(g-g',\mu_1,\mu_2)u^{2g-2+l(\vec\mu_1)+l(\vec\mu_2)}\\
&=\sum_{\dl}\prod_{i=1}^{l(\vec\mu_2)}\mu_2^i\left(\prod_{e\in\eg^\circ(\dl)}w(e)^2\right)\left(\prod_{v\in V_6(\dl)}\sum_{g_v\geq0}N_{g_v,d_v}^{\mu_v,\nu_v}u^{2g_v-2+l(\mu_v)+l(\nu_v)}\right)\\
&\left(\frac{1}{2\sin{\frac{u}{2}}}\right)^{k|V_6|+|V_2|-\sum_{i=1}^kd\cdot[E_i]}\left(\prod_{v\in V_2(\dl)}u^{-1}\right)\\
&\left(\prod_{v\in V_1^2(\dl)}u^{-1}\cos\left(\frac{u}{2}\right)\right)\left(\prod_{v\in V_1^1(\dl)}u^{-1}2\sin\left(\frac{u}{2}\right)\right),
\end{aligned}
\end{equation}
where the sum over $\dl$ is taken over the equivalence classes of $d$-marked floor diagrams of genus $g'$ and type $(\vec\mu_1,\vec\mu_2)$,
$\mu_v$ is the partition having the weights of outgoing edges of $\dl$ incident to $v$ and $i_v$ ones as its entries, 
$\nu_v$ is the partition whose entries are the weights of original incoming edges in $\eg(\dl)$ incident to $v$.
\end{lemma}

\begin{proof}
When $n=0$, the relative invariants $N_{g,d,n}^{X_{k}'|E}(g-g',\mu_1,\mu_2)$ are given in Remark \ref{prop:vanishing-X_k1} and Proposition \ref{prop:vanishing-X_k2},
so we only consider the case that $n>0$ in the following.
From Proposition \ref{prop:4.1}, the relative invariants $N_{g,d,n}^{X_{k}'/E}(g-g',\mu_1,\mu_2)$ is expressed by equation $(\ref{eq:4.2})$ in terms of invariants $N_{\Gamma,v}$ (see Definition \ref{def:4.2} for $N_{\Gamma,v}$).
From the non-vanishing results about relative Gromov-Witten invariants (Remark \ref{prop:vanishing-X_k1}, Proposition \ref{prop:vanishing-X_k2}, Proposition \ref{prop:excep1}, \cite[Lemma 5.1]{bou2021} and \cite[Lemma 5.2]{bou2021}), along with Proposition \ref{prop:4.2} and Lemma \ref{lem:4.02}, 
it follows that only graphs in $\overline\gl^{\vec\mu_1,\vec\mu_2}_{g,g',d,n}$ make a non-trivial contribution to equation $(\ref{eq:4.2})$.
Hence, equation $(\ref{eq:4.2})$ can be rewritten as follows.
\begin{equation}
\label{eq:4.3}
\begin{aligned}
&N_{g,d,n}^{X_{k}'|E}(g-g',\mu_1,\mu_2)\\
&=
\sum_{\Gamma\in\overline\gl^{\vec\mu_1,\vec\mu_2}_{g,g',d,n}}\frac{\prod_{e\in\eg(\Gamma)\setminus\eg^{\vec\mu}(\Gamma)}w_e}{|\aut(\Gamma)|}\prod_{v\in V_6}N_{g_v,d_v}^{\mu_v,\nu_v}
\prod_{v\in V_3}N_{g_v,d_v,0}^{X_k'|E}(\emptyset,(1))\\
&\left(\prod_{v\in V_2(\Gamma)}N_{g_v,d_v,0}^{X_k'|E}((1),\emptyset)\right)\left(\prod_{v\in V_1^2(\Gamma)}N_{g_v,d_v,0}^{X_k'|E}((2),\emptyset)\right)\left(\prod_{v\in V_1^1(\Gamma)}N_{g_v,d_v,0}^{X_k'|E}((1,1),\emptyset)\right)\\
&\left(\prod_{v\in V_4(\Gamma)}\frac{1}{d_v\cdot[Z_{j_v}^+]}\right)\left(\prod_{v\in V_5(\Gamma)}1\right).
\end{aligned}
\end{equation}

We compute the contribution of chains in $\Gamma$ that correspond to the edges in a floor diagram $\dl$ from Lemma \ref{lem:4.02}. 
We give only a detailed argument for the chains satisfying the conditions in Lemma \ref{lem:4.02}$(1)$. 
The contributions of edges corresponding to chains in Lemma \ref{lem:4.02}$(2)$ can be calculated similarly, 
so we omit the details but give the contributions of the remaining edges.
\begin{itemize}
    \item If the endpoints of the chain $c$ consist of one vertex in $V_6$ and one in $V_1\cup V_2\cup V_6$, 
    there is exactly one vertex in $V_5$ which is contained in chain $c$. 
    Suppose that $e_1,\ldots,e_m$ is such a chain. 
    Then the weight of $e_1,\ldots,e_m$ is the same, say $w(e_i)=w$. 
    The chain contains one vertex in $V_5$ and $m-2$ vertex in $V_4$, 
    so the contribution of the chain to the left-hand side of equation (\ref{eq:4.3}) is $\frac{w^m}{w^{m-2}}=w^2$. 
    \item If the endpoints of the chain $c$ consist of one vertex in $V_1\cup V_2\cup V_6$ and one in $\vt^{\vec\mu_2}(\Gamma)$, 
    it follows from Lemma \ref{lem:4.02} that there is exactly one vertex in $V_5$ in chain $c$. 
    Suppose that $\bar e_1,\ldots,\bar e_m$ is such a chain. 
    The chain contains one vertex in $V_5$ and $m-2$ vertex in $V_4$. 
    Since we consider edges adjacent to a vertex in $\vt^{\vec\mu_2}(\dl)$, 
    the contribution of the chain to the left-hand side of equation (\ref{eq:4.3}) is $\frac{w^{m-1}}{w^{m-2}}=w$.
    \item If the two endpoints of the chain $c$ consist of one vertex in $V_1\cup V_2\cup V_6$ and one in $\vt^{\vec\mu_1}(\Gamma)$, the contribution of $c$ to the left-hand side of equation (\ref{eq:4.3}) is $1$. 
    \item If the two endpoints of the chain $c$ consist of one vertex in $V_3$ and one in $V_6$, the contribution of $c$ to the left-hand side of equation (\ref{eq:4.3}) is $1$.
\end{itemize}

When we neglect the contribution of chains of $\Gamma$ with one endpoint in $V_3$ and the other endpoint in $V_6$ and add some additional weighted $1$ leaves, the contributions of edges are not changed.
Finally, one obtains formula (\ref{eq:4.21}) from (\ref{eq:4.3}), Remark \ref{prop:vanishing-X_k1}, Proposition \ref{prop:vanishing-X_k2}, Proposition \ref{prop:excep1} and Lemma \ref{lem:4.01}.
\end{proof}

\begin{theorem}
\label{thm:1}
Let $g'\in\zb_{\geq0}$, and $d\in H_2(X_k';\zb)$ be a homology class such that $d\cdot[L]>0$, $d\cdot[E]\geq0$. Suppose that $\vec\mu=(\vec\mu_1,\vec\mu_2)$ is an ordered partition of $d\cdot[E]$, 
and let $n:=d\cdot[L]-1+g'+l(\mu_2)$. Then,
we have the equality
\begin{equation}\label{eq:thm1}
\begin{aligned}
\sum_{g\geq g'}N_{g,d,n}^{X_{k}'|E}&(g-g',\mu_1,\mu_2)u^{2g-2+l(\vec\mu_1)+l(\vec\mu_2)}\\
&=u^{-d\cdot[L]}N_{q}^{\mu_1,\mu_2}(X_k',d,g')\left((-i)(q^\frac{1}{2}-q^{-\frac{1}{2}})\right)^{n+g'-1+l(\mu_1)}
\end{aligned}
\end{equation}
of power series in $u$ with rational coefficients, where
$$
q=e^{iu}=\sum_{m\geq0}\frac{(iu)^m}{m!}.
$$
\end{theorem}

\begin{proof}
When $n=0$,  $(\ref{eq:thm1})$ follows from Remark \ref{prop:vanishing-X_k1} and Proposition \ref{prop:vanishing-X_k2}.
So we only consider the case that $n=d\cdot[L]-1+g'+l(\mu_2)>0$ in the following.
From Lemma \ref{lem:4.3}, Proposition \ref{prop:4.1} and \cite[Theorem 4.4]{bou2021}, we obtain the following.
\begin{equation}
\label{eq:4.4}
\begin{aligned}
&\sum_{g\geq g'}N_{g,d,n}^{X_{k}'|E}(g-g',\mu_1,\mu_2)u^{2g-2+l(\vec\mu_1)+l(\vec\mu_2)}\\
&=\sum_{\dl}u^{-|V_2|-|V_1|-2|V_6|}\left(\prod_{e\in\eg^{\vec\mu_2}(\dl)}w(e)\right)\left(\prod_{e\in\eg^\circ(\dl)}w(e)^2\right)\\
&\left(\prod_{v\in V_6(\dl)}\prod_{j=1}^{l(\mu_v)}\frac{1}{\mu_v^{(j)}}2\sin(\frac{\mu_v^{(j)}u}{2})\prod_{l=1}^{l(\nu_v)}\frac{1}{\nu_v^{(l)}}2\sin(\frac{\nu_v^{(l)}u}{2})\right)\left(\frac{1}{2\sin{\frac{u}{2}}}\right)^{k|V_6|+|V_2|-\sum_{i=1}^kd\cdot[E_i]}\\
&\left(\cos\left(\frac{u}{2}\right)\right)^{|V_1^2|}\left(2\sin\left(\frac{u}{2}\right)\right)^{|V_1^1|}.
\end{aligned}
\end{equation}
Let $e\in\dl$ be an edge pointed from $v$ to $v'$,
where $v\in V_6$. Suppose that the factor corresponding to $v$ in equation (\ref{eq:4.4}) is
$$
\prod_{j=1}^{l(\mu_v)}\frac{1}{\mu_v^{(j)}}2\sin(\frac{\mu_v^{(j)}u}{2})\prod_{l=1}^{l(\nu_v)}\frac{1}{\nu_v^{(l)}}2\sin(\frac{\nu_v^{(l)}u}{2}).
$$
From our construction of $\dl$, the weight $w(e)$ of $e$ is an entry in $\mu_v=(\mu_v^{(1)},\ldots,\mu_v^{(l(\mu_v))})$. 
For every $v\in V_6$, there are $i_v$ ones (see Definition \ref{def:additional}) in the partition $\mu_v$.
We consider the contribution of the $i_v$ ones now. 
From the definition of $i_v$, Proposition \ref{prop:4.2} and Lemma \ref{lem:4.02}(3), the sum 
$$
\sum_{v\in V_6}i_v=k|V_6|+|V_2|-\sum_{i=1}^kd\cdot[E_i].
$$
Hence, the contributions of the $i_v$ ones in $\mu_v$ for all $v\in V_6$ are cancelled by the factor $\left(\frac{1}{2\sin{\frac{u}{2}}}\right)^{k|V_6|+|V_2|-\sum_{i=1}^kd\cdot[E_i]}$.
Note that $d\cdot[L]=|V_2|+|V_1|+2|V_6|$.
The product of $u^{-d\cdot[L]}$ and the contributions of all edges and their endpoints is the total contribution of a $d$-marked floor diagram $\dl$ to the sum in the right-hand side of equation (\ref{eq:4.4}). 
Under the change of variables $q=e^{iu}$, we get $2\sin(\frac{w(e)u}{2})=(-i)(q^\frac{1}{2}-q^{-\frac{1}{2}})[w(e)]_q$ and $\cos(\frac{u}{2})=\frac{[2]_q}{2}$. 
Therefore, the right-hand side of equation (\ref{eq:4.4}) is
\begin{equation}
\label{eq:4.6}
u^{-d\cdot[L]}\sum_{\dl}\mult_q(\dl,m)((-i)(q^\frac{1}{2}-q^{-\frac{1}{2}}))^{N(\dl)+|V_1^1|},
\end{equation}
where $\mult_q(\dl,m)$ is given in Definition \ref{def:2.4}
and $N(\dl)$ is the number of original half-edges $(v,e)$ of $\dl$ with $v\in V_6(\dl)$.

Suppose that $V_6(\dl)\neq\emptyset$. 
The connectedness of $\dl$ implies that every vertex in $V_1\cup V_2$ is connected to a vertex in $V_6$.
In the marked floor diagram $(\dl,m)$, the number
$$
N(\dl)=2|E_6|+|V_2|+|V_1^2|+2|V_{6-1-6}|+|V_{\mu_1-1-6}|+|V_{\mu_1-6}|+|V_{\mu_2-1-6}|+|V_{\mu_2-6}|.
$$
Here, $V_{\mu_1-6}$ is the set of vertices in $\vt^{\vec\mu_1}(\dl)$ that are connected to a vertex in $V_6(\dl)$ by an edge,
$V_{\mu_i-1-6}\subset V_1^1$ the subset consisting of vertices in $V_1^1$ 
that are connected to a vertex in $\vt^{\vec\mu_i}(\dl)$ and a vertex in $V_6$ by two edges,
$E_6$ is the set of edges in $\dl$ with two endpoints in $V_6(\dl)$,    
and $V_{6-1-6}\subset V_1^1$ is the subset consisting of vertices in $V_1^1$ that are connected to two vertices in $V_6(\dl)$  by two edges.
From Definition \ref{def:4.4}, we obtain that $n$ is equal to $|V_6|+|E|$, where $E$ is the set of edges of $\dl$ that
correspond to chains of the type listed in Lemma \ref{lem:4.02}(1). Therefore,
$$
\begin{aligned}
n=|V_6|+|E_6|+|V_2|+|V_1^2|+2|V_{\mu_2-1-6}|+|V_{\mu_2-6}|+2|V_{6-1-6}|+|V_{\mu_1-1-6}|.
\end{aligned}
$$
Forget all inner vertices $v$ in $\dl\setminus V_6$ and merge the two edges adjacent to $v$ into a single edge.
The first Betti number of $\dl$ is invariant under this procedure. 
Then we get that the first Betti number $g'$ of $\dl$ is equal to $-|V_6|+|E_6|+|V_{6-1-6}|+1$.
Hence, the number
$$
N(\dl)=n+g'-1-|V_{6-1-6}|+|V_{\mu_1-6}|-|V_{\mu_2-1-6}|.
$$
Note that $l(\mu_1)=|V_{\mu_1-6}|+|V_{\mu_1-1-6}|$ and $|V_1^1|=|V_{\mu_1-1-6}|+|V_{\mu_2-1-6}|+|V_{6-1-6}|$.
We obtain that $N(\dl)=n+g'-1+l(\mu_1)-|V_1^1|$.
Then the sum in $(\ref{eq:4.6})$ is equal to
\begin{equation}
\label{eq:4.7}
u^{-d\cdot[L]}\left(\sum_{\dl}\mult_q(\dl,m)\right)((-i)(q^\frac{1}{2}-q^{-\frac{1}{2}}))^{n+g'-1+l(\mu_1)},
\end{equation}
after the change of variables $q=e^{iu}=\sum_{j\geq0}\frac{(iu)^m}{m!}$ and $iu=\ln q$. 
The equality in Theorem \ref{thm:1} follows from the formula (\ref{eq:4.7}) and Definition \ref{def:2.4}.

Suppose that $V_6(\dl)=\emptyset$. From $n\neq0$ and Definition \ref{def:4.4}, the set $E\neq\emptyset$.
It follows from Lemma \ref{lem:4.02} that $g'=0$, and $N(\dl)=n+g'-1+l(\mu_1)-|V_1^1|=0$. 
It is straightforward to verify that Theorem \ref{thm:1} holds.
\end{proof}

\section{Gromov-Witten invariants and higher genus BPS polynomials}\label{sec:5}

Arg\"uz and Bousseau \cite{ab25} generalized the genus zero Block-G\"ottsche polynomials to arbitrary surface by the BPS polynomials.
In this section, we use the correspondence theorem (Theorem \ref{thm:1}) to study the BPS polynomials. 
We propose a higher genus version of BPS polynomials of del Pezzo surfaces of degree $\geq3$ and Hirzebruch surfaces.
We show that the higher genus BPS polynomials are equal to the higher genus Block-G\"ottsche polynomials of toric del Pezzo surfaces and Hirzebruch surfaces.

\subsection{Gromov-Witten invariants and BPS polynomials of 3-folds}
\subsubsection{Gromov-Witten invariants}
Let $X$ be a smooth projective variety, and $d\in H_2(X;\zb)$ be a curve class.
Denote by $\overline M_{g,n}(X,d)$
the moduli space of stable maps of genus $g$ and degree $d$ to $X$ with $n$ marked points.
The moduli space $\overline M_{g,n}(X,d)$
is a proper and separated Deligne-Mumford stack \cite{beh1997,li-tian-1998,fp97}
(see \cite{fo1999,lt1998,ruan1999,pardon-2016} for the symplectic case)
with virtual dimension
$$
\vdim \overline M_{g,n}(X,d)=\int_d c_1(X)+(\dim_\mathbb{C}X-3)(1-g)+n,
$$
where $c_1(X)$ is the first Chern class of $X$.
Gromov-Witten invariants of $X$ are the virtual counts:
\begin{equation}
\label{eq:GW}
\left<\gamma;A_1,\ldots,A_n\right>_{g,d}^{X}:=
\int_{[\overline M_{g,n}(X,d)]^{vir}}\gamma\cdot\prod_{i=1}^n ev_i^*(A_i).
\end{equation}
Here,
\begin{itemize}
    \item $[\overline M_{g,n}(X,d)]^{vir}$
    is the virtual fundamental class of the moduli stack of stable maps;
    \item $ev_i:\overline M_{g,n}(X,d)\to X$ is the evaluation map defined by
    $ev_i([f:(C,x_1,\ldots,x_n)\to X])=f(x_i)$;
    \item $A_1,\ldots,A_n\in H^{*}(X;\qb)$, 
    and $\gamma$ is a cohomology class in $H^{2*}(\overline M_{g,n}(X,d);\qb)$.
\end{itemize}

Let $\pi:\cl\to\overline M_{g,n}(X,d)$ be the universal curve,
and $\omega_\pi$ be the relative dualizing sheaf.
The Hodge bundle $\eb_\pi:=\pi_*\omega_\pi$ is a rank $g$ vector bundle on
$\overline M_{g,n}(X,d)$. Lambda classes are the Chern classes
of the Hodge bundle $\eb_\pi$: 
$
\lambda_{m,\pi}:=c_m(\eb_\pi)\in H^{2m}(\overline M_{g,n}(X,d);\qb), m=0,1,\ldots,g.
$
When $p_1,\ldots,p_n\in H^{*}(X;\qb)$ are classes Poincar\'e dual to points in $X$, we put
\begin{equation}
\label{equ:GW1}
N_{g,d,n}^{X}(m):=
\left<(-1)^m\lambda_m;p_1,\ldots,p_n\right>_{g,d}^{X}.
\end{equation}
When $m=g$, we use $N_{g,d,n}^{X}$ to denote $N_{g,d,n}^{X}(g)$.

\subsubsection{BPS polynomials of 3-folds}
We recall the definition of BPS polynomials of 3-folds from \cite[Section 3.1]{ab25}.

Let $X$ be a smooth projective 3-fold, and let $A=(A_1,\ldots,A_n)$ be a collection of classes $A_i\in H^*(X;\zb)$.
When the class $d\in H_2(X;\zb)$ satisfies $c_1(X)\cdot d>0$, following \cite{pand-02,pand-99},
the \textit{Gopakumar-Vafa BPS invariants} $n^X_{g,d,A}$ are defined by the formula
$$
\sum_{g\geq0}\langle A_1,\ldots,A_n\rangle_{g,d}^Xu^{2g-2+c_1(X)\cdot d}=\sum_{g\geq0}n_{g,d,A}^X\left(2\sin(\frac{u}{2})\right)^{2g-2+c_1(X)\cdot d}.
$$
From \cite[Theorem 1.5]{zinger-99}, the BPS invariants $n^X_{g,d,A}$ are integers.

\begin{definition}[{\cite[Definition 3.1]{ab25}}]
    Let $X$ be a smooth projective 3-fold, and let $A=(A_1,\ldots,A_n)$ be a collection of classes $A_i\in H^*(X;\zb)$.
    When the class $d\in H_2(X;\zb)$ satisfies $c_1(X)\cdot d>0$, the \textit{BPS polynomial} $n_{d,A}^X(q)$ of $X$ is the Laurent polynomial defined as
    $$
    n_{d,A}^X(q):=\sum_{g\geq0}n_{g,d,A}^X\left(2\sin(\frac{u}{2})\right)^{2g}=\sum_{g\geq0}n_{g,d,A}^X(-1)^g(q-2+q^{-1})^{g}\in\zb[q^\pm],
    $$
    where $q=e^{iu}$.
\end{definition}

\subsection{BPS polynomials of surfaces}
We review the definition of BPS polynomials of surfaces following \cite[Section 3.2]{ab25}, then we introduce the definition of higher genus BPS polynomials.

Let $X$ be a smooth projective surface and let $Y=X\times\pb^1$. We denote the natural projections by $\pi_1:Y\to X$ and $\pi_2:Y\to\pb^1$.
Fix a class $d\in H_2(X;\zb)$ such that $c_1(X)\cdot d-1\geq0$. Let $P_d=(p_i)_{0\leq i\leq c_1(X)\cdot d-1}$ be a tuple of $c_1(X)\cdot d$ classes that are defined as
$$
p_0:=\pi_2^*(q_0)\in H^2(Y;\zb), p_i:=\pi_1^*(q_i)\in H^4(Y;\zb), ~\forall 1\leq i\leq c_1(X)\cdot d-1,
$$
where $q_0\in H^2(\pb^1;\zb)$ is a class Poincar\'e dual to a point in $\pb^1$, and $q_i\in H^4(X;\zb)$ is a class Poincar\'e dual to a point in $X$, $i\in\{1,\ldots,c_1(X)\cdot d-1\}$.

\begin{definition}[{\cite[Definition 3.2]{ab25}}]
    Let $X$ be a smooth projective surface and $d\in H_2(X;\zb)$ be a class such that $c_1(X)\cdot d-1\geq0$. The \textit{BPS polynomial} $n_{d}^X(q)$ of $X$ is the BPS polynomial $n_{(d,0),P_d}^Y(q)$ of the 3-fold $Y=X\times\pb^1$ of class $(d,0)\in H_2(Y;\zb)\cong H_2(X;\zb)\times\zb$ and with the insertion of $P_d$:
    $$
    n_{d}^X(q):=n_{(d,0),P_d}^Y(q)\in\zb[q^\pm].
    $$
\end{definition}

\begin{lemma}[{\cite[Lemma 3.3]{ab25}}]\label{lem:0-bps-poly}
    Let $X$ be a smooth projective surface and $d\in H_2(X;\zb)$ be a class such that $c_1(X)\cdot d-1\geq0$. Then, with the change of variables $q=e^{iu}$, we have
    $$
    n_{d}^X(q)=\left(2\sin(\frac{u}{2})\right)^{2-c_1(X)\cdot d}\sum_{g\geq0}N_{g,d,c_1(X)\cdot d-1}^Xu^{2g-2+c_1(X)\cdot d}.
    $$
\end{lemma}

We introduce the higher genus BPS polynomials as follows.
\begin{definition}\label{def:g-bps-poly}
    Let $g'\in\zb_{\geq0}$, and suppose that $X$ and $d$ satisfy one of the following conditions:
    \begin{itemize}
        \item $X$ is a del Pezzo surface of degree $\geq 3$ and $d\in H_2(X;\zb)$ is a class such that $c_1(X)\cdot d+g'-1\geq0$ and $d\cdot [L]>0$.
        \item $X$ is a Hirzebruch surface and $d\in H_2(X;\zb)$ is a class such that $d\cdot[D_i]\geq0$ for every toric divisor $D_i$ of $X$.
    \end{itemize}
    Define $ P_{g',d}^X(u)$ to be the power series $P_{g',d}^X(u)$ in $u$ with rational coefficients as follows
    \begin{equation}\label{eq:g-bps-poly}
    P_{g',d}^X(u)=\left(2\sin(\frac{u}{2})\right)^{2-2g'-c_1(X)\cdot d}\sum_{g\geq g'}N_{g,d,c_1(X)\cdot d+g'-1}^X(g-g')u^{2g-2+c_1(X)\cdot d}
    \end{equation}
 
  The \textit{genus $g'$ class $d$ BPS polynomial $n_{g',d}^X(q)$} of $X$ is defined as
    $$n_{g',d}^X(q):=P_{g',d}^X(u),$$
    under the change of variables $q=e^{iu}$.
\end{definition}

\begin{remark}
    It follows from Theorem \ref{thm:main1} that Definition \ref{def:g-bps-poly} is well defined.
\end{remark}

\subsection{Projective plane and Hirzebruch surfaces}

\begin{theorem}\label{thm:case1}
    Let $X$ be the projective plane $\pb^2$ or the Hirzebruch surface $\fb_k$. Given $g'\in\zb_{\geq0}$, let $d\in H_2(X;\zb)$ be a class such that $d\cdot[D_i]\geq0$ for every toric divisor $D_i$ of $X$.
    Then, under the change of variables $q=e^{iu}$, we have
    $$
    P_{g',d}^X(u)=N_{g',d}^{X,\trop}(q).
    $$
    In particular, $N_{g',d}^{X,\trop}$ is a Laurent polynomial in $q$ with integer coefficients.
\end{theorem}

\begin{proof}
    We first consider the case $X=\fb_k$. Let $S_k$ and $S_{-k}$ be the two distinguished sections of $\fb_k$ with $[S_k]^2=k=-[S_{-k}]^2$, and let $F$ be a fiber in $\fb_k$. Let $n=c_1(X)+g'-1$.
    In order to compare the Gromov-Witten invariants $N_{g,d,n}^X(g-g')$ with the relative Gromov-Witten invariants,
    we degenerate $X$ to the normal cone of the smooth divisor $S_{-k}\cup S_{k}$, and the central fiber is $\fb_k^{(1)}\sideset{_{S_k}}{_{S_{-k}}}{\mathop{\cup}}\fb_k^{(2)}\sideset{_{S_k}}{_{S_{-k}}}{\mathop{\cup}}\fb_k^{(3)}$.
    Then we apply the degeneration formula \cite{li2002} to this degeneration and keep the $n$ point insertions in $\fb_k^{(2)}$. Suppose that the non-zero terms in the degeneration formula are indexed by graphs $\Gamma$.
    Let $v\in\Gamma$ be a vertex corresponding to a curve in $\fb_k^{(1)}$ of class $h_v[S_k]+l_v[F]$, genus $g_v$ and contact orders $\mu$, where $\mu$ is a partition of $h_vk+l_v$.
    The expected dimension of the moduli space of relative stable maps corresponding to the vertex $v$ is $2h_v+l_v+g_v-1+l(\mu)$. The maximal degree of an insersion is $g_v+l(\mu)$. By a dimension counting, the contribution of $\Gamma$ is $0$ except for the case that $2h_v+l_v+g_v-1+l(\mu)\leq g_v+l(\mu)$. Hence, $l_v=1$ and $h_v=0$. By \cite[Lemma 3.3]{bou2021}, the contribution of the vertex $v$ is $1$ if $g_v=0$; otherwise, the contribution of $v$ is $0$. 
    If $v\in\Gamma$ is a vertex corresponding to a curve in $\fb_k^{(3)}$,
    by a similar argument as the case $\fb_k^{(1)}$ we obtain that $d_v=[F]$. Then, the contribution of the vertex $v$ is $1$ if $g_v=0$; otherwise, the contribution of $v$ is $0$. 
    Since the graph $\Gamma$ is connected, there is only one vertex $v'$ of $\Gamma$ corresponding to curves in $\fb_k^{(2)}$. The graph $\Gamma$ has to be genus $0$ and $g_{v'}=g$. By \cite[Proposition 3.1]{bou2021} the Lambda class at the vertex $v'$ is $\lambda_{g-g'}$.
    Hence, we get
    \begin{equation}\label{eq:thm-case1}
    N_{g,d,n}^X(g-g')=\langle(1^{d\cdot[S_{-k}]}),\delta^1|(-1)^{g-g'}\lambda_{g-g'};p_1,\ldots,p_n|(1^{d\cdot[S_{k}]}),\delta^2\rangle_{g,d}^{X|(S_{-k}\cup S_k)},
    \end{equation}
    where $\delta^1_i=1\in H^0(S_{-k};\zb)$ and $\delta^2_j=1\in H^0(S_{k};\zb)$.
    We denote by $N_{g,rel}^{\Delta,n}$ the relative invariant on the right-hand side of equation (\ref{eq:thm-case1}).
    Under the change of variables $q=e^{iu}$, we have
    $$
    P_{g',d}^X(u)=\left((-i)(q^{\frac{1}{2}}-q^{-\frac{1}{2}})\right)^{2-2g'-c_1(X)\cdot d}\sum_{g\geq g'}N_{g,rel}^{\Delta,n}u^{2g-2+c_1(X)\cdot d}.
    $$
    Theorem \ref{thm:case1} follows from \cite[Theorem 7.1]{bou2021}, \cite[Theorem 5.12]{bou2021} and \cite[Theorem 1]{bou2019}. 

    When $X=\pb^2$, we degenerate $\pb^2$ to the normal cone of the line $L$, and the central fiber is $\pb^2\cup\fb_1$. 
    Then we apply the degeneration formula \cite{li2002} to this degeneration and keep the $n$ point insertions in $\pb^2$. The rest of the proof follows the same line as in the case $X=\fb_k$, so we omit it.
\end{proof}

Let $X$ be a toric del Pezzo surface.
It follows from \cite[Theorem A]{ab25} that the BPS polynomial $n_{0,d}^X(q)$ is equal to the genus zero Block-G\"ottsche polynomial $N_{0,d}^{X,\trop}(q)$. 
From Theorem \ref{thm:case1} and the following theorem, the higher genus BPS polynomial of a toric del Pezzo surface is equal to the higher genus Block-G\"ottsche polynomial.

\begin{theorem}\label{conj1}
    Let $g'\in\zb_{\geq1}$, and $d\in H_2(X_k;\zb)$, $k=2,3$, be a class such that $d\cdot[D_i]\geq0$ for every toric divisor $D_i$ of $X_k$. Then, we have
    $$
    n_{g',d}^{X_k}(q)=N_{g',d}^{X_k,\trop}(q).
    $$
\end{theorem}

\begin{proof}
The proof of this theorem follows closely \cite[Section 3.3]{ab25}. We first recall the construction from \cite[Section 3.3]{ab25}, then modify it to our case.

Let $X$ be the toric del Pezzo surface $X_2$ or $X_3$.
Let $\rho_1,\ldots,\rho_j$ be the rays of the fan $\Sigma_X$ of $X$. We assume that these rays are labeled clockwise.
We use $m_i$ to denote the primitive integral point on $\rho_i$.
For any $i\in\{1,\ldots,j-1\}$, let $E_i$ be the edge in $\rb^2$ that connects $m_i$ to $m_{i+1}$. Let $E_j$ be the edge connecting $m_j$ to $m_1$.
Denote by $\pl_X$ the polyhedral decomposition of $\rb^2$ obtained from the fan $\Sigma_X$ by adding edges $E_1,\ldots,E_j$.
A toric degeneration $\pi:\xk\to\cb$ is determined by the polyhedral decomposition $\pl_X$ as in \cite[Section 3]{ns-2006}.
The central fiber $\xk_0:=\pi^{-1}(0)$ has dual intersection complex $\pl_X$, and the fiber $\xk_t:=\pi^{-1}(t)$, $t\neq0$, is isomorphic to the surface $X$.
Moreover, the decomposition of the central fiber $\xk_0$ into irreducible components is given by
$$
\xk_0=X\bigcup\left(\cup_{i=1}^j\pb_i\right).
$$
Here, $\pb_i$ are toric surfaces that are $\pb^1$-bundles over the divisors $D_i$ in $X$ corresponding to the rays $\rho_i$ of $\Sigma_X$.

Equip $\xk$ with the divisorial log structure defined by $\xk_0$, and $\cb$ with the divisorial log structure defined by $\{0\}\subset\cb$. Then the morphism $\pi:\xk\to\cb$ lifts to a log smooth morphism naturally.
Note that the tropicalization of $\xk$ is the cone over the compact polygon $P\subset\rb^2$ with vertices the points $m_1,\ldots,m_j$.
We choose $n:=c_1(X)\cdot d+g'-1$ generic sections $\undl p(t)=(p_1(t),\ldots,p_n(t))$ of $\pi:\xk\to\cb$ such that $p_i(0)$ are lying entirely in the irreducible component of the central fiber isomorphic to $X$. We specialize the $n$ point constraints as $\undl p(t)$.
Since the log structure on a generic fiber $\xk_t$ is the trivial log structure, by applying the decomposition formula in log Gromov-Witten theory of \cite[Theorem 5.4]{acgs-2020} to the log smooth degeneration $\pi:\xk\to\cb$, one obtains that
\begin{equation}\label{eq:decomp}
N_{g,d,n}^X(g-g')=\sum_{h:\Gamma\to P}\frac{n_h}{|\aut(h)|}N_{h,\log}^{\xk_0},
\end{equation}
where we sum over the rigid decorated tropical curves $h:\Gamma\to P$ of total genus $g$ total class $d$ and with $n$ legs.
Here, $n_h$ is the smallest positive integer such that $n_h\cdot h(\Gamma)$ has integral vertices, $|\aut(h)|$ is the order of the group of automorphisms of $h$, and $N_{h,\log}^{\xk_0}$ is the log Gromov-Witten invariant of the central fiber $\xk_0$ endowed with the restricted log structure form $\xk$.
Let $\overline{M}_h(\xk_0)$ be the moduli space of $h$-marked stable log maps to $\xk_0$ passing through the $n$ point constraints imposed at the marked points corresponding to the $n$ legs of $\Gamma$ (see \cite[Definition 2.31]{acgs-2020}).
The invariant $N_{h,\log}^{\xk_0}$ (see \cite{gs-2013} for the definition of log Gromov-Witten invariant) is defined as
$$
N_{h,\log}^{\xk_0}=\int_{[\overline{M}_h(\xk_0)]^{vir}}(-1)^{g-g'}\lambda_{g-g'}.
$$
From the vanishing property of $\lambda_g$ introduced in \cite[Lemma 8]{bou2019}, we know that $N_{h,\log}^{\xk_0}=0$ unless the genus $g_\Gamma$ of the graph $\Gamma$ satisfies $g_\Gamma\leq g'$.

We refine the polyhedral decomposition of $P$ such that it contains $h(\Gamma)$. Let $\wt\xk_0$ be the central fiber of the corresponding log modification of $\pi:\xk\to\cb$.
Denote by $\wt X^w$ the irreducible components of $\wt\xk_0$ labeled by the vertices $w$ of the refined polyhedral decomposition.
Equip every irreducible component $\wt X^w$ with the divisorial log structure defined by the intersection $\partial\wt X^w$ of $\wt X^w$ with the singular locus of $\xk_0$.
Note that the irreducible component $\wt X^0$ corresponding to the origin is a toric blow-up of the irreducible component $X$ of $\xk_0$.

For every vertex $v\in\Gamma$, let $h_v:\Gamma_v\to P$ be the rigid decorated tropical curve with one vertex obtained as the star of $v$ in $h:\Gamma\to P$.
Since the image $h(v)$ is a vertex of the refined polyhedral decomposition, it corresponds to an irreducible component $\wt X^{h(v)}$ of $\wt\xk_0$. 
We denote by $\overline{M}_{h_v}(\wt X^{h(v)})$ the moduli space of $h_v$-marked stable log maps to $\wt X^{h(v)}$ passing through the point constraints imposed at the marked points corresponding to the legs of $\Gamma$ adjacent to $v$.
Let $e\subset\Gamma$ be an edge of $\Gamma$.
Since $h(e)$ is contained in an edge of the refined polyhedral decomposition of $P$, it corresponds to an irreducible component $\wt D^{h(e)}$ of the singular locus of $\wt\xk_0$.
Let $1_{h(e)}\in H^0(\wt D^{h(e)})$ and $pt_{h(e)}\in H^2(\wt D^{h(e)})$ be the unit in cohomology and the class of a point in $\wt D^{h(e)}$, respectively. 
A {\it splitting data} $\sigma$ associates every half-edge $(e,v)$ of $\Gamma$ a cohomology class $\sigma_{e,v}\in H^*(\wt D^{h(e)})$ such that either $\sigma_{e,v}=1_{h(e)}$ and $\sigma_{e,v'}=pt_{h(e)}$, or $\sigma_{e,v}=pt_{h(e)}$ and $\sigma_{e,v'}=1_{h(e)}$, where $v$ and $v'$ are the two endpoints of the edge $e$.
Let $v\in\Gamma$ be a vertex with genus decoration $g_v$ and class decoration $d_v$, and $\sigma$ be a splitting data. Fix a non-negative integer $g_v'\leq g_v$. Let $N_{h,\sigma,\log}^{v,g_v'}$ be the log Gromov-Witten invariant defined as
\begin{equation}\label{eq:log-gw1}
N_{h,\sigma,\log}^{v,g_v'}:=\int_{[\overline M_{h_v}(\wt X^{h(v)})]^{vir}}(-1)^{g_v'}\lambda_{g_v'}\prod_{v\in e}ev_e^*(\sigma_{e,v}),
\end{equation}
where $ev_e:\overline M_{h_v}(\wt X^{h(v)})\to \wt D^{h(e)}$ are the evaluation maps at the corresponding marked points, and the product is taken over the edges of $\Gamma$ adjacent to $v$, that are considered as legs of $\Gamma_v$.

As in the proof of \cite[Proposition 13]{bou2019}, one obtains the gluing formula as follows (see page 18 of \cite{ab25} for the case $g'=0$):
\begin{equation}\label{eq:glue}
N_{h,\log}^{\xk_0}=\sum_{\sigma,g_v'}\prod_ew_e\prod_{v}N_{h,\sigma,\log}^{v,g_v'},
\end{equation}
where the sum is taken over the splitting data $\sigma$ and integers $g_v'$ such that $0\leq g_v'\leq g_v$ and $\sum g_v'=g-g'$, and $w_e$ are weights of the edges of $\Gamma$.
We only consider the splitting data and $g_v'$ so that the log Gromov-Witten invariants $N_{h,\sigma,\log}^{v,g_v'}\neq0$.
To conclude the proof of Theorem \ref{conj1}, we need the following lemma.

\begin{lemma}\label{lem:non-vanish}
Let $h:\Gamma\to P$ be a rigid tropical curve which contributes to the sum in the right-hand side of equation $(\ref{eq:decomp})$. 
Let $V_0(\Gamma)$ be the set of vertices $v$ of $\Gamma$ such that $h(v)=0\in P$ \textit{i.e.}, $\wt X^{h(v)}=\wt X^0$, and $V(\Gamma)$ be the set of vertices of $\Gamma$. Denote by $\nu:\wt\xk_0\to X$ the composition morphism of $\wt\xk_0\to\xk_0$ and the natural projection $\xk_0\to X$.
If $N_{h,\sigma,\log}^{v,g_v'}\neq0$, then we have the following results.
\begin{enumerate}[$(1)$]
    \item The genus $g_\Gamma$ of the graph $\Gamma$ is zero.
    \item The set $V_0(\Gamma)$ consists of only one vertex $v_0$, and we have $g_{v_0}-g'=g_{v_0}'$.
    \item For every vertex $v\notin V_0(\Gamma)$, we have $\nu_*d_v=0$ and $g_v=g_v'$.
    \item For each edge $e$ adjacent to the vertex $v_0\in V_0(\Gamma)$, the divisor $\wt D^{h(e)}$ is not an exceptional divisor of the blow-up $\wt X^0\to X$. Moreover, the weight $w_e$ of $e$ is equal to one and $\sigma_{e,v_0}=1\in H^0(\wt D^{h(e)})$.
\end{enumerate}
\end{lemma}

\begin{proof}[Proof of Lemma $\ref{lem:non-vanish}$]
    The dimension of the virtual class $[\overline M_{h_v}(\wt X^{h(v)})]^{vir}$ is
    $$
    c_1(\wt X^{h(v)})\cdot d_v+g_v-1+\sum_{v\in e}(1-w_e)-m_v,
    $$
    where $m_v$ is the number of legs of $\Gamma$ adjacent to $v$.
    The dimension of the classes on the right-hand side of equation (\ref{eq:log-gw1}) is $g_v'+\sum_{v\in e}\deg_\cb\sigma_{e,v}$, where $\deg_\cb\sigma_{e,v}$ is the complex degree of the cohomology class $\sigma_{e,v}$.
    If $N_{h,\sigma,\log}^{v,g_v'}\neq0$, we must have
    \begin{equation}\label{eq:dim-eq1}
    c_1(\wt X^{h(v)})\cdot d_v+g_v-1+\sum_{v\in e}(1-w_e)-m_v=g_v'+\sum_{v\in e}\deg_\cb\sigma_{e,v}.
    \end{equation}
    Sum the equalities given by (\ref{eq:dim-eq1}) over the vertices in $V_0(\Gamma)$, then we have
    $$
        \sum_{v\in V_0(\Gamma)}c_1(\wt X^{0})\cdot d_v+\sum_{v\in V_0(\Gamma)}g_v=|V_0(\Gamma)|+\sum_{v\in V_0(\Gamma)}(m_v+g_v')+\sum_{\substack{v\in V_0(\Gamma) \\ v\in e}}(\deg_\cb\sigma_{e,v}+w_e-1).
    $$
    Given $v\in V_0(\Gamma)$, let $E(v)$ be the set of edges $e$ adjacent to $v$ such that $\wt D^{h(e)}$ is an exceptional divisor of the toric blow-up $\wt X^0\to X$.
    In the toric blow-up $\wt X^0\to X$, we get the relation $c_1(\wt X^0)\cdot d_v=c_1(X)\cdot\nu_*d_v-\sum_{e\in E(v)}w_e$.
    Hence, we have
    $$
    \sum_{v\in V_0(\Gamma)}c_1(X)\cdot \nu_*d_v=|V_0(\Gamma)|+\sum_{v\in V_0(\Gamma)}(m_v+g_v'-g_v)+\sum_{\substack{v\in V_0(\Gamma) \\ v\in e}}(\deg_\cb\sigma_{e,v}+w_e-1)+\sum_{\substack{v\in V_0(\Gamma)\\e\in E(v)}}w_e.
    $$
    Since $d=\sum_v\nu_*d_v$ and $\sum_{v\in V_0(\Gamma)}m_v=c_1(X)\cdot d+g'-1$, the above equation gives us
    \begin{equation}\label{eq:dim-count1}
    \begin{aligned}
        1=&\sum_{v\notin V_0(\Gamma)}c_1(X)\cdot \nu_*d_v+|V_0(\Gamma)|+g'+\sum_{v\in V_0(\Gamma)}(g_v'-g_v)+\sum_{\substack{v\in V_0(\Gamma)\\e\in E(v)}}w_e\\
        &+\sum_{\substack{v\in V_0(\Gamma) \\ v\in e}}(\deg_\cb\sigma_{e,v}+w_e-1).
    \end{aligned}
    \end{equation}
    Recall that $g_v\geq g_v'$, so we have 
    $$
    \begin{aligned}
    \sum_{v\in V_0(\Gamma)}(g_v'-g_v)&\geq\sum_{v\in V(\Gamma)}(g_v'-g_v)\\
    &=g-g'-(g-g_\Gamma)=g_\Gamma-g'.
    \end{aligned}
    $$
    Hence, we obtain the following inequality from equality (\ref{eq:dim-count1}):
    \begin{equation}\label{eq:dim-count2}
    \begin{aligned}
        1\geq&\sum_{v\notin V_0(\Gamma)}c_1(X)\cdot \nu_*d_v+|V_0(\Gamma)|+g_\Gamma+\sum_{\substack{v\in V_0(\Gamma)\\e\in E(v)}}w_e+\sum_{\substack{v\in V_0(\Gamma) \\ v\in e}}(\deg_\cb\sigma_{e,v}+w_e-1).
    \end{aligned}
    \end{equation}
    Note that the classes $\nu_*d_v$ are effective and $X$ is a del Pezzo surface, so we have $c_1(X)\cdot\nu_*d_v\geq0$. Moreover, $c_1(X)\cdot\nu_*d_v=0$ if and only if $\nu_*d_v=0$.
    Since $c_1(X)\cdot d\geq2$ (see Section \ref{sec:trop-ref-inv}) and the $c_1(X)\cdot d+g'-1$ point constraints are in $\wt X^0$, we always have $V_0(\Gamma)\neq\emptyset$.
    All the other terms in the right-hand side of inequality (\ref{eq:dim-count2}) are non-negative, so we get: $\nu_*d_v=0$ for every $v\notin V_0(\Gamma)$, $|V_0(\Gamma)|=1$, $g_{\Gamma}=0$, $\deg_\cb\sigma_{e,v}=0$, $w_e=1$, and $E(v)=\emptyset$ for all $v\in V_0(\Gamma)$ and $v\in e$. Then we obtain from the equality (\ref{eq:dim-count1}) that $g_{v_0}=g_{v_0}'+g'$. Note that $\sum_{v\in V(\Gamma)}g_v'=\sum_{v\in V(\Gamma)}g_v-g'$, so we have $g_v'=g_v$ for $v\notin V_0(\Gamma)$.
\end{proof}

Note that $h:\Gamma\to P$ is a rigid tropical curve. From Lemma \ref{lem:non-vanish}, the graph $\Gamma$ is a graph of genus $0$, and we have the following.
\begin{itemize}
    \item There exists only one vertex $v_0$ of $\Gamma$ such that $h(v_0)=0$.
    \item For every ray $\rho_i$ of the fan $\Sigma_X$, there are exactly $d\cdot D_i$ vertices $v_{i,k}$, $1\leq k\leq d\cdot D_i$, in $\Gamma$ that are mapped to the primitive integral point $m_i$ by $h$, where $D_i$ is the toric divisor corresponding to the ray $\rho_i$. 
    \item The vertices $v_{i,k}$, $1\leq i\leq j, 1\leq k\leq d\cdot D_i$, are leaves of $\Gamma$, and $v_0$ is an inner vertex of $\Gamma$. Moreover, $\Gamma$ has no vertices other than the vertices $v_0$ and $v_{i,k}$.
    \item The weight of every end of $\Gamma$ is $1$, and $\Gamma$ has no edges other than ends.
    \item The class decoration $d_{v_0}$ of the inner vertex $v_0$ is $d$, and the the class decorations $d_{v_{i,k}}$ of the leaves are the class of a $\pb^1$-fiber of $\pb_i\to D_i$.
    \item $g_{v_0}'=g_{v_0}-g'$ and $g_{v_{i,k}}'=g_{v_{i,k}}$.
\end{itemize}
From Lemma \ref{lem:non-vanish}, there is only one splitting data $\sigma$ of $\Gamma$: $\sigma_{e_{i,k},v_0}=1_{m_i}$, and $\sigma_{e_{i,k},v_{i,k}}=pt_{m_i}$. 
From \cite[Lemma 15]{bou2019}, one obtains $N_{h,\sigma,\log}^{v_{i,k},g_{v_{i,k}}}=0$ if $g_{v_{i,k}}>0$, and $N_{h,\sigma,\log}^{v_{i,k},g_{v_{i,k}}}=1$ if $g_{v_{i,k}}=0$.
Since the genus of the graph $\Gamma$ is zero, we have $g_{v_0}=g$ and $g_{v_0}'=g-g'$.
Therefore, the genus decoration of $\Gamma$ is unique.
Let $\overline M_{g,n}^{\log}(X/D,d)$ be the moduli space of genus $g$ stable log maps to $X$ of class $d$ with $n$ marked points and $d\cdot D_i$ marked points having contact order one along $D_i$ for all $1\leq i\leq j$.
We denote the genus $g$ log Gromov-Witten invariant of $(X,D)$ of class $d$ by
$$
N_{g,d,n,\log}^{(X,D)}(g-g')=\int_{[\overline M_{g,n}^{\log}(X/D,d)]^{vir}}(-1)^{g-g'}\lambda_{g-g'}\prod_{i=1}^nev_i^*(pt),
$$
where $ev_i:\overline M_{g,n}^{\log}(X/D,d)\to X$ are the evaluation maps at the corresponding marked points.
From the gluing formula (\ref{eq:glue}) and the decomposition formula (\ref{eq:decomp}), we get that
$$
N_{g,d,n}^X(g-g')=N_{g,d,n,\log}^{(X,D)}(g-g').
$$
From \cite[Theorem 1]{bou2019}, we have
$$
N_{g',d}^{X,\trop}(q)=\left(2\sin\left(\frac{u}{2}\right)\right)^{2-d\cdot D-2g'}\sum_{g\geq g'}N_{g,d,n,\log}^{(X,D)}(g-g')u^{2g-2+d\cdot D}.
$$
Therefore, the result follows from the definition of the higher genus BPS polynomial (Definition \ref{def:g-bps-poly}).
\end{proof}

\subsection{del Pezzo surfaces of degrees $>3$}
In this subsection, we always assume that $k\leq5$.
It is well known that when $k\leq5$ the surface $X_k\cong X_k'$.

\begin{lemma}\label{lem:deg>3}
Let $g'\in\zb_{\geq0}$, and $d\in H_2(X_k;\zb)$ be a homology class such that $d\cdot[L]>0$, $d\cdot[E]\geq0$ and $n:=c_1(X_k)\cdot d+g'-1\geq0$. Then, for any $g\geq g'$
we have the equality
$$
N_{g,d,n}^{X_{k}}(g-g')=N_{g,d,n}^{X_k|E}(g-g',\emptyset,1^{d\cdot[E]}).
$$
\end{lemma}

\begin{proof}
    We first degenerate $X_k$ to the normal cone of the smooth conic $E$, and the central fiber is $X_k\cup\nl$,
    where $\nl=\pb(\ol_E\oplus N_{E|X_k'})$.
    Then we apply the degeneration formula to this degeneration and keep the $n$ point insertions in $X_k$. Suppose that the non-zero terms in the degeneration formula are indexed by graphs $\Gamma$.
    Let $Z^+=\pb(0\oplus N_{E|X_k'})$, and $F$ be a fiber in $\nl$.
    Let $v\in\Gamma$ be a vertex corresponding to a curve in $\nl$ of class $h_v[Z^+]+l_v[F]$, genus $g_v$ and contact orders $\mu$, where $\mu$ is a partition of $h_v(k-4)+l_v$.
    The expected dimension of the moduli space of relative stable maps corresponding to the vertex $v$ is $2h_v+l_v+g_v-1+l(\mu)$. The maximal degree of an insersion is $g_v+l(\mu)$. By a dimension counting, the contribution of $\Gamma$ is $0$ except for the case that $2h_v+l_v+g_v-1+l(\mu)\leq g_v+l(\mu)$. Hence, $l_v=1$ and $h_v=0$. By \cite[Lemma 3.3]{bou2021}, the contribution of the vertex $v$ is $1$ if $g_v=0$; otherwise, the contribution of $v$ is $0$. Since the graph $\Gamma$ is connected, there is only one vertex $v'$ of $\Gamma$ corresponding to curves in $X_k$. The graph $\Gamma$ has to be genus $0$ and $g_{v'}=g$. By \cite[Proposition 3.1]{bou2021} the Lambda class at the vertex $v'$ is $\lambda_{g-g'}$.
\end{proof}

\begin{theorem}\label{thm:2}
Let $g'\in\zb_{\geq0}$, and $d\in H_2(X_k;\zb)$ be a homology class such that $d\cdot[L]>0$, $d\cdot[E]\geq0$ and $n:=c_1(X_k)\cdot d+g'-1\geq0$. Then,
we have the equality
$$
\begin{aligned}
P_{g',d}^{X_k}(u)
=N_{q}^{\emptyset,(1^{d\cdot [E]})}(X_k,d,g'),
\end{aligned}
$$
where $q=e^{iu}=\sum_{m\geq0}\frac{(iu)^m}{m!}$.
\end{theorem}

\begin{proof}
From the definition of $P_{g',d}^{X_k}(u)$ (see equation (\ref{eq:g-bps-poly})), Lemma \ref{lem:deg>3} and Theorem \ref{thm:1}, we have
$$
\begin{aligned}
    P_{g',d}^{X_k}(u)&=\left((-i)(q^\frac{1}{2}-q^{-\frac{1}{2}})\right)^{2-2g'-c_1(X_k)\cdot d}\sum_{g\geq g'}N_{g,d,n}^{X_k|E}(g-g',\emptyset,1^{d\cdot[E]})u^{2g-2+c_1(X_k)\cdot d}\\
    &=N_{q}^{\emptyset,(1^{d\cdot [E]})}(X_k,d,g').
\end{aligned}
$$
\end{proof}

\subsection{del Pezzo surfaces of degree $3$}
Let $\pi:\yl\to\cb$ be the classical flat degeneration of $X_6$ with central fiber $\pi^{-1}(0)=X_6'\cup(\pb^1\times\pb^1)$. 
In the central fiber, $X_6'$ intersects $\pb^1\times\pb^1$ transversally. 
The intersection $X_6'\cap(\pb^1\times\pb^1)$ is the rational curve $E$ in $X_6'$, 
and is a hyperplane section in $\pb^1\times\pb^1$ (see the proof of \cite[Proposition 6.1]{bru2015} or \cite[Section 4.2]{iks2013a}).

\begin{lemma}\label{lem:deg=3}
Let $g'\in\zb_{\geq0}$, and $d\in H_2(X_6;\zb)$ be a homology class such that $d\cdot[L]>0$, $d\cdot[E]\geq0$ and $n:=c_1(X_6)\cdot d+g'-1\geq0$. Then, for any $g\geq g'$
we have the equality
$$
N_{g,d,n}^{X_{6}}(g-g')=\sum_{k\geq0}\binom{d\cdot[E]+2k}{k}N_{g,d-k[E],n}^{X_6'|E}(g-g',\emptyset,1^{d\cdot[E]+2k}),
$$
where the sum is taken over $k$ such that $(d-k[E])\cdot[L]>0$.
\end{lemma}

\begin{proof}
    We apply the degeneration formula \cite{li2002} to the degeneration $\pi:\yl\to\cb$ of $X_6$ and keep the $n$ point insertions in $X_6'$. Suppose that the non-zero terms in the degeneration formula are indexed by graphs $\Gamma$.
    Let $L_1=\pb^1\times\{0\}$ and $L_2=\{0\}\times\pb^1$ be the two divisors in $\pb^1\times\pb^1$, then the homology classes $[L_1]$ and $[L_2]$ generate the homology group $H_2(\pb^1\times\pb^1;\zb)$.
    Let $v\in\Gamma$ be a vertex corresponding to a curve in $\pb^1\times\pb^1$ of class $h_v[L_1]+l_v[L_2]$, genus $g_v$ and contact orders $\mu$, where $\mu$ is a partition of $[E]\cdot(h_v[L_1]+l_v[L_2])=h_v+l_v$.
    The expected dimension of the moduli space of relative stable maps in $\pb^1\times\pb^1$ corresponding to the vertex $v$ is $h_v+l_v+g_v-1+l(\mu)$. The maximal degree of an insersion is $g_v+l(\mu)$. By a dimension counting, the contribution of $\Gamma$ is $0$ except for the case that $h_v+l_v+g_v-1+l(\mu)\leq g_v+l(\mu)$. Hence, $(h_v,l_v)=(1,0)$ or $(0,1)$. By \cite[Lemma 3.3]{bou2021}, the contribution of the vertex $v$ is $1$ if $g_v=0$; otherwise, the contribution of $v$ is $0$. Moreover, the relative insertion of the vertex $v$ is a point insertion. We may assume that all vertices $v\in\Gamma$ corresponding to the curves in $\pb^1\times\pb^1$ have the genus $g_v=0$ and $(h_v,l_v)=(1,0)$ or $(0,1)$.

    Each vertex $v\in\Gamma$ corresponding to a curve in $\pb^1\times\pb^1$ is connected to a vertex $v'\in\Gamma$ corresponding to curves in $X_6'$ by one edge. Since the graph $\Gamma$ is connected, there is only one vertex $v'$ in $\Gamma$ corresponding to curves in $X_6'$. The graph $\Gamma$ has to be genus $0$ and $g_{v'}=g$. By \cite[Proposition 3.1]{bou2021} the Lambda class at the vertex $v'$ is $\lambda_{g-g'}$. The homology class $d_{v'}=d-k[E]$ for some $k\geq0$. If $d_{v'}\cdot[L]=0$, the class $d_{v'}$ must be of the form $a_i[E_i]$ with $a_i\geq1$. 
    Note that the exceptional curve $E_i$ is rigid. 
    If $n>0$, this is not possible. When $n=0$, from Remark \ref{prop:vanishing-X_k1} and Proposition \ref{prop:excep1}, $v'$ is adjacent to only one vertex in $\Gamma$ that corresponds to a curve in $\pb^1\times\pb^1$. This is a contradiction.
    Hence, the homology class $d_{v'}$ always satisfies $d_{v'}\cdot[L]>0$.
    We may assume that the sum of the homology classes $\sum\limits_vd_v$ of the vertices $v$ corresponding to curves in $\pb^1\times\pb^1$ is $k[L_1]+(k+d\cdot[E])[L_2]$. We have 
    $$
    \binom{d\cdot[E]+2k}{k}
    $$
    choices to choose the $k$ vertices $v\in\Gamma$ with $(h_v,l_v)=(1,0)$ and the $d\cdot[E]+k$ vertices with $(h_v,l_v)=(0,1)$ among the $d\cdot[E]+2k$ vertices in $\Gamma$ corresponding to curves in $\pb^1\times\pb^1$.
\end{proof}

\begin{theorem}\label{thm:3}
Let $g'\in\zb_{\geq0}$, and $d\in H_2(X_6;\zb)$ be a homology class such that $d\cdot[L]>0$, $d\cdot[E]\geq0$ and $n:=c_1(X_6)\cdot d+g'-1\geq0$. Then,
we have the equality
\begin{equation}\label{eq:thm:3}
\begin{aligned}
P_{g',d}^{X_6}(u)&=\sum_{k\geq0}\binom{d\cdot[E]+2k}{k}N_{q}^{\emptyset,(1^{d\cdot [E]+2k})}(X_6',d-k[E],g'),
\end{aligned}
\end{equation}
where $q=e^{iu}=\sum_{m\geq0}\frac{(iu)^m}{m!}$.
\end{theorem}

\begin{proof}
    From Lemma \ref{lem:deg=3},  we have
    \begin{align*}
    &\sum_{g\geq g'}N_{g,d,n}^{X_{6}}(g-g')u^{2g-2+d\cdot [E]}\\
    &=\sum_{k\geq0}\binom{d\cdot[E]+2k}{k}\left(\sum_{g\geq g'}N_{g,d-k[E],n}^{X_6'|E}(g-g',\emptyset,1^{d\cdot[E]+2k})u^{2g-2+d\cdot [E]}\right).
    \end{align*}
    By Theorem \ref{thm:1}, one has
    \begin{align*}
    &\sum_{g\geq g'}N_{g,d-k[E],n}^{X_6'|E}(g-g',\emptyset,1^{d\cdot[E]+2k})u^{2g-2+d\cdot [E]+2k-2k}\\
    &=u^{-(d-k[E])\cdot[L]-2k}N_{q}^{\emptyset,(1^{d\cdot [E]+2k})}(X_6',d-k[E],g')\left((-i)(q^\frac{1}{2}-q^{-\frac{1}{2}})\right)^{n+g'-1}.
    \end{align*}
    Therefore, we obtain equation $(\ref{eq:thm:3})$ from the definition of $P_{g',d}^{X_6}(u)$ (see equation (\ref{eq:g-bps-poly})).
\end{proof}

Summarizing results in Theorem \ref{thm:case1}, Theorem \ref{thm:2} and Theorem \ref{thm:3}, we have 

\begin{theorem}\label{thm:main1}
Let $X$ be del Pezzo surfaces of degree $\geq 3$ or Hirzebruch surfaces. Then
 $ P_{g',d}^X(u)$ is a Laurent polynomial in $q$ with integer coefficients under the change of variables $q=e^{iu}$.
\end{theorem}

\begin{remark}
    It follows from Lemma \ref{lem:0-bps-poly} and Theorem \ref{thm:main1} that the polynomial $n^{X}_{0,d}(q)$ is equal to the BPS polynomial $n_d^X(q)$.
\end{remark}

\subsection{Higher genus relative BPS polynomials}\label{subsec:higher-rel-bps}
We introduce a relative version of the BPS invariants.
\begin{definition}\label{def:rel-bps-inv}
Let $g'\in\zb_{\geq0}$, and $d\in H_2(X_k';\zb)$ be a homology class such that $d\cdot[L]>0$, $d\cdot[E]\geq0$ and $n:=c_1(X_k')\cdot d+g'-1\geq0$. Suppose that $\vec\mu=(\vec\mu_1,\vec\mu_2)$ is an ordered partition of $d\cdot[E]$. For every $g\geq g'$, we define the \textit{relative BPS invariants $n_{g,d}^{X_k'|E}(g',\mu_1,\mu_2)\in\qb$} by the formula
\begin{equation}
\begin{aligned}
    &\sum_{g\geq g'}N_{g,d,n}^{X_k'|E}(g-g',\mu_1,\mu_2)u^{2g-2+d\cdot[L]+l(\mu_1)+l(\mu_2)}\\
    &=\left(\prod_{i=1}^{l(\mu_1)}\frac{1}{\mu_1^{(i)}}2\sin(\frac{\mu_1^{(i)}u}{2})\right)\left(\prod_{i=1}^{l(\mu_2)}\frac{1}{\mu_2^{(i)}}2\sin(\frac{\mu_2^{(i)}u}{2})\right)\sum_{g\geq g'}n_{g,d}^{X_k'|E}(g',\mu_1,\mu_2)\left(2\sin(\frac{u}{2})\right)^{2g-2+d\cdot[L]}.
\end{aligned}
\end{equation}
\end{definition}

\begin{definition}\label{def:rel-bps-poly}
Let $g'\in\zb_{\geq0}$, and $d\in H_2(X_k';\zb)$ be a homology class such that $d\cdot[L]>0$, $d\cdot[E]\geq0$ and $n:=c_1(X_k')\cdot d+g'-1\geq0$. Suppose that $\vec\mu=(\vec\mu_1,\vec\mu_2)$ is an ordered partition of $d\cdot[E]$. We define the \textit{genus $g'$ relative BPS polynomial $n_{g',d,(\mu_1,\mu_2)}^{X_k'|E}(q)$} by the formula
\begin{equation}
\begin{aligned}
    n_{g',d,(\mu_1,\mu_2)}^{X_k'|E}(q):=&\sum_{g\geq g'}n_{g,d}^{X_k'|E}(g',\mu_1,\mu_2)\left(2\sin(\frac{u}{2})\right)^{2g-2g'}\\
    =&\sum_{g\geq g'}n_{g,d}^{X_k'|E}(g',\mu_1,\mu_2)(-1)^{g-g'}(q-2+q^{-1})^{g-g'}.
\end{aligned}
\end{equation}
\end{definition}

\begin{remark}
    When $g'=0$, the relative BPS polynomial $n_{g',d,(\mu_1,\mu_2)}^{X_k'|E}(q)$ is equal to the BPS polynomial defined in \cite[Definition 4.3]{ab25}.
\end{remark}

\begin{corollary}\label{cor:rel-corre}
Let $g'\in\zb_{\geq0}$, and $d\in H_2(X_k';\zb)$ be a homology class such that $d\cdot[L]>0$, $d\cdot[E]\geq0$ and $n:=d\cdot[L]-1+g'+l(\mu_2)\geq0$. Suppose that $\vec\mu=(\vec\mu_1,\vec\mu_2)$ is an ordered partition of $d\cdot[E]$. Then,
we have the equality
\begin{equation}
\begin{aligned}
\prod_{i=1}^{l(\mu_1)}\frac{[\mu_1^{(i)}]_q}{\mu_1^{(i)}}\cdot\prod_{i=1}^{l(\mu_2)}\frac{[\mu_2^{(i)}]_q}{\mu_2^{(i)}}\cdot n_{g',d,(\mu_1,\mu_2)}^{X_k'|E}(q)=N_{q}^{\mu_1,\mu_2}(X_k',d,g').
\end{aligned}
\end{equation}
\end{corollary}

\begin{proof}
From Definition \ref{def:rel-bps-inv} and Definition \ref{def:rel-bps-poly}, under the change of variables $q=e^{iu}$ we have
$$
\begin{aligned}
&\prod_{i=1}^{l(\mu_1)}\frac{[\mu_1^{(i)}]_q}{\mu_1^{(i)}}\cdot\prod_{i=1}^{l(\mu_2)}\frac{[\mu_2^{(i)}]_q}{\mu_2^{(i)}}\cdot n_{g',d,(\mu_1,\mu_2)}^{X_k'|E}(q)\left((-i)(q^{\frac{1}{2}}-q^{-\frac{1}{2}})\right)^{-2+d\cdot[L]+2g'+l(\mu_1)+l(\mu_2)}\\
&=\prod_{i=1}^{l(\mu_1)}\frac{[\mu_1^{(i)}]_q}{\mu_1^{(i)}}\cdot\prod_{i=1}^{l(\mu_2)}\frac{[\mu_2^{(i)}]_q}{\mu_2^{(i)}}\sum_{g\geq g'}n_{g,d}^{X_k'|E}(g',\mu_1,\mu_2)\left((-i)(q^{\frac{1}{2}}-q^{-\frac{1}{2}})\right)^{-2+d\cdot[L]+2g+l(\mu_1)+l(\mu_2)}\\
&=\sum_{g\geq g'}N_{g,d,n}^{X_k'|E}(g-g',\mu_1,\mu_2)u^{2g-2+d\cdot[L]+l(\mu_1)+l(\mu_2)}.
\end{aligned}
$$
From Theorem \ref{thm:1}, we obtain the required equality.
\end{proof}

\section{Caporaso-Harris formula for the refined counts}\label{sec:6}
We first employ the degeneration formula \cite{li2002} to derive a recursive formula 
of $N_{g,d,n}^{X_{k}'|E}(g-g',\mu_1, \mu_2)$ in the sense of Caporaso-Harris \cite{ch1998}. 
Subsequently, we utilize the correspondence theorem (Theorem \ref{thm:1})
to establish a Caporaso-Harris recursive formula for the refined counts of the floor diagrams relative to a conic. Combined with Corollary \ref{cor:rel-corre}, it is a Caporaso-Harris type recursive formula for the higher genus relative BPS polynomials.

\begin{lemma}
\label{lem:7.1}
Let $g'\in\zb_{\geq0}$, and $d\in H_2(X_k';\zb)$ be a homology class such that $d\cdot[L]>0$ and $d\cdot[E]\geq0$. Choose an ordered partition $\vec\mu=(\vec\mu_1,\vec\mu_2)$ of $d\cdot[E]$. Suppose that $n:=d\cdot[L]-1+g'+l(\mu_2)>0$ and $g\geq g'$.
Then we have the equality
\begin{equation}\label{eq:CH-1}
\begin{aligned}
N_{g,d,n}^{X_{k}'|E}&(g-g',\mu_1, \mu_2)=\sum_{w(\mu_2)\neq0} w\cdot N^{X_k'|E}_{g,d,n-1}(g-g',\mu_1\cup(w),\mu_2\setminus(w))\\
&+\sum\frac{1}{\sigma}\binom{d\cdot[L]+g'-2+l(\mu_2)}{d_1\cdot[L]+g_1'-1+l(\mu_2^1),\ldots,d_m\cdot[L]+g_m'-1+l(\mu_2^m)}\\
&\cdot\binom{\mu_1}{\mu_1^1,\ldots,\mu_1^m}\prod_{j=1}^m\binom{\mu_2^j}{\mu_2^j-\xi^j}\prod_{w\in\xi}w\prod_{j=1}^m N_{g_j,d_j,n_j}^{X_k'|E}(g_j-g_j',\mu^j_1,\mu_2^j)N_{g_0,[Z^+]+l[F]}^{\xi,\eta}.
\end{aligned}
\end{equation}
Here, $w(\mu_2)$ is the number of entries in $\mu_2$ equal to $w$, the second sum is taken over all collections of classes $d_1,\ldots,d_m$, collections of integers $g_1,\ldots,g_m$ and $g_1',\ldots,g_m'$, and collections of partitions $\mu^1=(\mu^1_1,\mu^1_2), \ldots, \mu^m=(\mu^m_1,\mu^m_2),\xi^1,\ldots,\xi^m$
that satisfy the following conditions:
$$
\begin{aligned}
    &\cup_{j=1}^m\mu_1^j\subset\mu_1, \\
    &\mu_2\cup(\cup_{j=1}^m\xi^j)=\cup_{j=1}^m\mu_2^j, |\xi^j|\neq0,\\
    &d_1+\cdots+d_m=d-[E],\\
    &g-g'=g_1-g_1'+\cdots+g_m-g_m'+g_0, 0\leq g_i'\leq g_i,\\
    &\sum_{j=1}^m(l(\xi^j)+g_j)-m+g_0=g.
\end{aligned}
$$
Here, $\eta=\mu_1\setminus\cup_{j=1}^m\mu_1^j$, $\xi=\cup_{j=1}^m\xi^j$, and $n_j=d_j\cdot[L]+g_j'-1+l(\mu_2^j)$.
Moreover, $d_i\neq l[E_a]$ with $l\geq2$. If $d_i=[E_a]$, we have $d_j\neq[E_a]$ for any $j\neq i$, where $i,j\in\{1,\ldots,m\}, a\in\{1,\ldots,k\}.$
The symbol $\sigma$ is defined as follows.
We define an equivalence relation $\sim$ in the set $\{1,\ldots,m\}$ as follows: $i\sim j$ if and only if $(g_{i},d_{i},g_{i}',\mu^{i})=(g_{j},d_{j},g_{j}',\mu^{j})$.
The number $\sigma$ denotes the product of factorials of the cardinalities of the equivalent classes. 
The non-zero initial data for formula $(\ref{eq:CH-1})$ are given in Remark $\ref{prop:vanishing-X_k1}$, Proposition $\ref{prop:vanishing-X_k2}$ and \cite[Theorem $4.4$]{bou2021}.
\end{lemma}

\begin{proof}
    As in the proof of Lemma \ref{lem:deg>3}, we degenerate $X_k'$ to the normal cone of the smooth conic $E$, and the central fiber is $X_k'\cup\nl$,
    where $\nl=\pb(\ol_E\oplus N_{E|X_k'})$.
    We apply the degeneration formula \cite{li2002} to this degeneration. In the degeneration procedure, we keep $n-1$ point insertions in $X_k'$ and $1$ point in $\nl$. 
    Suppose that the non-zero terms in the degeneration formula are indexed by graphs $\Gamma$.
    Let $Z^+=\pb(0\oplus N_{E|X_k'})$, and $F$ be a fiber in $\nl$.
    Let $v\in\Gamma$ be a vertex corresponding to a curve in $\nl$ of class $d_v=h_v[Z^+]+l_v[F]$, genus $g_v$ and contact orders $\xi$, $\eta$, where $\xi$ and $\eta$ are two partitions of $[Z^+]\cdot d_v=h_v(k-4)+l_v$ and $[Z^-]\cdot d_v=l_v$ respectively.

    If the curve in $\nl$ corresponding to $v$ does not pass through the absolute marked point in $\nl$, the expected dimension of the moduli space of relative stable maps in $\nl$ corresponding to $v$ is $2h_v+g_v-1+l(\xi)+l(\eta)$.
    The maximal degree of an insersion is $g_v+l(\xi)+l(\eta)$.
    The contribution of $\Gamma$ is $0$ except for the case that 
    $$
    2h_v+g_v-1+l(\xi)+l(\eta)\leq g_v+l(\xi)+l(\eta).
    $$
    Hence, we have $(h_v,l_v)=(0,l_v)$, where $l_v$ is one entry in the partition $\mu=(\mu_1,\mu_2)$.
    By \cite[Lemma $5.1$]{bou2021}, the contribution of $v$ is $\frac{1}{l_v}$ if $g_v=0$; otherwise, the contribution of $v$ is $0$.

    Now we consider the curve in $\nl$ corresponding to $v$ that passes through the absolute marked point in $\nl$.
    The expected dimension of the moduli space of relative stable maps corresponding to the vertex $v$ is $2h_v+g_v+l(\xi)+l(\eta)$.
    The maximum degree of insertion is $g_v+l(\xi)+l(\eta)+2$. The contribution of $\Gamma$ is $0$ except for the case $2h_v+g_v+l(\xi)+l(\eta)\leq g_v+l(\xi)+l(\eta)+2$.
    Hence, $(h_v,l_v)=(0,l_v)$ or $(h_v,l_v)=(1,l_v)$, where $l_v$ is an entry in partition $\mu=(\mu_1,\mu_2)$.
    We first consider the case $(h_v,l_v)=(0,l_v)$.
    By \cite[Lemma 5.2]{bou2021}, the contribution of the vertex $v$ is $1$ if $g_v=0$; otherwise, the contribution of $v$ is $0$.
    Moreover, $l(\xi)=l(\eta)=1$ and the two relative insertions are a trivial class.
    Therefore, $l_v$ is an entry in $\mu_2$.
    Since all classes $d_v$ in $\nl$ are fiber classes and the graph $\Gamma$ is connected, there is only one vertex $v'$ of $\Gamma$ corresponding to curves in $X_k$. The graph $\Gamma$ has to be genus $0$ and $g_{v'}=g$. By \cite[Proposition 3.1]{bou2021} the Lambda class at the vertex $v'$ is $\lambda_{g-g'}$.
    By the degeneration formula \cite{li2002}, this case corresponds to the first summarization on the right-hand side of equation (\ref{eq:CH-1}).

    In the case that $(h_v,l_v)=(1,l_v)$, from \cite[Lemma 5.2]{bou2021}, the contribution of the vertex $v$ is 
    $N_{g_v,d_v}^{\xi,\eta}$ (see equation (\ref{eq:def-rel})),
    if the lambda class at $v$ is $\lambda_{g_v}$; otherwise, the contribution of $v$ is $0$.
    Moreover, all $l(\xi)+l(\eta)$ cohomology classes $\delta^1$ and $\delta^2$ are Poincar\'e dual to a point. Hence, $\eta$ consists of some entries in $\mu_1$. 
    By \cite[Proposition 3.1]{bou2021} the genus $g(\Gamma)$ of $\Gamma$ satisfies $g(\Gamma)\leq g'$.
    Suppose that there are $m$ vertices $v_1',\ldots,v_m'$ of $\Gamma$ corresponding to the curves in $X_k'$. Then the data $d_{i}, g_{i},g_i',\mu^{i}=(\mu_{1}^{i},\mu_{2}^{i}),\lambda_{v_i'}$ are chosen such that the graph $\Gamma$ is connected with $g(\Gamma)\leq g'$ and the Lambda classes are split according to \cite[Proposition 3.1]{bou2021}. Obviously, $\mu_{1}=\eta\cup(\cup_{j=1}^m\mu_{1}^{j})$ and $\xi\cup\mu_2=\cup_{j=1}^m\mu_{2}^{j}$. The classes $\sum_{i=1}^md_{i}=d-[E]$, and $g(\Gamma)=\sum_{j=1}^ml(\xi^j)-m$,
    where $\xi^j$ is the partition consisting of the weights of the edges in $\Gamma$ connecting $v_j'$ to $v$. 
    Note that $\xi=\cup\xi^j$. By \cite[Proposition 3.1]{bou2021}, $g-g'=g_1-g_1'+\cdots+g_m-g_m'+g_0$ and $0\leq g_i'\leq g_i$.
    The degeneration of the moduli spaces implies that $g(\Gamma)+\sum_{j=1}^mg_j+g_0=g$.
    The factor 
    $$
    \binom{d\cdot[L]+g'-2+l(\mu_2)}{d_1\cdot[L]+g_1'-1+l(\mu_2^1),\ldots,d_m\cdot[L]+g_m'-1+l(\mu_2^m)}
    $$
    corresponds to the number of ways to distribute the $n-1$ absolute marked points in the $m$ vertices $v_1',\ldots,v_m'$.
    The factors
    $$
    \binom{\mu_1}{\mu_1^1,\ldots,\mu_1^m} \text{ and }\prod_{j=1}^m\binom{\mu_2^j}{\mu_2^j-\xi^j}
    $$
    correspond to the numbers of ways to choose $\mu_1^1,\ldots,\mu_1^m$ 
    and $\mu_2^1,\ldots,\mu_2^m,\xi^1,\ldots,\xi^m$, respectively.
    The factor $\prod_{m\in\xi}m$ comes from the product of the relative invariants of $\nl$ 
    with no absolute marked points and the weights of the edges in $\Gamma$ connecting 
    vertices corresponding to curves in $\nl$ to the vertices $v_1',\ldots,v_m'$. 
    Note that the factor $\sigma=|\aut(\Gamma)|$.
\end{proof}

\begin{remark} 
When there is no Lambda classes, 
Vakil derived a Caporaso-Harris type recursive formula (see \cite[Theorem 6.8]{v2000}) 
for the higher genus relative Gromov-Witten invariants of $X_k'$.
\end{remark}

\begin{theorem}\label{thm:CH}
Let $g'\in\zb_{\geq0}$, and $d\in H_2(X_k';\zb)$ be a homology class such that $d\cdot[L]>0$ and $d\cdot[E]\geq0$. Choose an ordered partition $\vec\mu=(\vec\mu_1,\vec\mu_2)$ of $d\cdot[E]$. Suppose that $n:=d\cdot[L]-1+g'+l(\mu_2)>0$.
Then we have the equality
\begin{equation}\label{eq:CH-2}
\begin{aligned}
N_{q}^{\mu_1,\mu_2}(X_k',d,g')&=\sum_{w(\mu_2)\neq0} w\cdot N^{\mu_1\cup(w),\mu_2\setminus(w)}_{q}(X_k',d,g')\\
&+\sum\frac{1}{\sigma}\binom{d\cdot[L]+g'-2+l(\mu_2)}{d_1\cdot[L]+g_1'-1+l(\mu_2^1),\ldots,d_m\cdot[L]+g_m'-1+l(\mu_2^m)}\\
&\cdot\binom{\mu_1}{\mu_1^1,\ldots,\mu_1^m}\prod_{j=1}^m\binom{\mu_2^j}{\mu_2^j-\xi^j}\prod_{j=1}^m N_{q}^{\mu_1^j,\mu_2^j}(X_k',d_j,g_j)\prod_{x\in\xi}[x]_q\prod_{y\in\eta}\frac{[y]_q}{y}.
\end{aligned}
\end{equation}
Here, the second sum is taken over all collections of classes $d_1,\ldots,d_m$, collections of integers $g_1,\ldots,g_m$ and $g_1',\ldots,g_m'$, and collections of partitions $\mu^1=(\mu^1_1,\mu^1_2), \ldots, \mu^m=(\mu^m_1,\mu^m_2),\xi^1,\ldots,\xi^m$ as in Lemma $\ref{lem:7.1}$. Moreover, $\eta=\mu_1\setminus\cup_{j=1}^m\mu_1^j$, $\xi=\cup_{j=1}^m\xi^j$, $n_j=d_j\cdot[L]+g_j'-1+l(\mu_2^j)$, and $\sigma$
is defined as in Lemma $\ref{lem:7.1}$.
The non-zero initial data for formula $(\ref{eq:CH-2})$ are given in the following:
$$
\begin{aligned}
&N_q^{(2),\emptyset}(X_k',[L],0)=\frac{[2]_q}{2},~N_q^{(1,1),\emptyset}(X_k',[L],0)=1,\\
&N_q^{\emptyset,(1)}(X_k',[E_i],0)=1,~N_q^{(1),\emptyset}(X_k',[L]-[E_i],0)=1.
\end{aligned}
$$
\end{theorem}

\begin{remark}
    There is no $E_i$-marked floor diagrams relative to a conic according to Definition \ref{def:2.2}, so the notation $N_q^{\emptyset,(1)}([E_i],0)$ in Theorem \ref{thm:CH} means the Laurent polynomial such that the generating series $\sum_{g\geq0}N_{[E_i],g,0}^{X_k'|E}(\emptyset,(1)u^{2g-1}$ satisfies Theorem \ref{thm:1}, after changing the variable $q=e^{iu}$.
\end{remark}

\begin{proof}
In this proof we use the notation
$$
\begin{aligned}
\Delta=&\frac{1}{\sigma}\binom{d\cdot[L]+g'-2+l(\mu_2)}{d_1\cdot[L]+g_1'-1+l(\mu_2^1),\ldots,d_m\cdot[L]+g_m'-1+l(\mu_2^m)}\\
&\cdot\binom{\mu_1}{\mu_1^1,\ldots,\mu_1^m}\prod_{j=1}^m\binom{\mu_2^j}{\mu_2^j-\xi^j}\prod_{w\in\xi}w.
\end{aligned}
$$
From Lemma \ref{lem:7.1}, one obtains
\begin{equation}\label{eq:7.2}
\begin{aligned}
&\sum_{g\geq g'}N_{g,d,n}^{X_k'/E}(g-g',\mu_1,\mu_2)u^{2g-2+l(\mu_1)+l(\mu_2)}\\
=&\sum_{w(\mu_2)\neq0}w\cdot\sum_{g\geq g'}N_{g,d,n-1}^{X_k'/E}(g-g',\mu_1\cup(w),\mu_2\setminus(w))u^{2g-2+l(\mu_1)+l(\mu_2)}\\
&+\sum_{g\geq g'}\left(\sum\Delta\prod_{j=1}^m N_{g_j,d_j,n_j}^{X_k'/E}(g_j-g_j',\mu^j_1,\mu_2^j)N_{g_0,d_0}^{\xi,\eta}\right)u^{2g-2+l(\mu_1)+l(\mu_2)},
\end{aligned}
\end{equation}
where $d_0=[Z^+]+l[F]$. The second sum in the right-hand side of equality (\ref{eq:7.2}) can be reformulated as
$$
\begin{aligned}
&\sum_{g\geq g'}\sum\Delta\prod_{j=1}^m\left(N_{g_j,d_j,n_j}^{X_k'/E}(g_j-g_j',\mu^j_1,\mu_2^j)u^{2g_j-2+l(\mu_1^j)+l(\mu_2^j)}\right)\left(N_{g_0,d_0}^{\xi,\eta}u^{2g_0-2+l(\xi)+l(\eta)}\right)\\
=&\sum\Delta\prod_{j=1}^m(\sum_{g_j\geq g_j'}N_{g_j,d_j,n_j}^{X_k'/E}(g_j-g_j',\mu^j_1,\mu_2^j)u^{2g_j-2+l(\mu_1^j)+l(\mu_2^j)})(\sum_{g_0\geq0}N_{g_0,d_0}^{\xi,\eta}u^{2g_0-2+l(\xi)+l(\eta)}).
\end{aligned}
$$
By Theorem \ref{thm:1} and \cite[Theorem 4.4]{bou2021}, we have
\begin{equation}\label{eq:7.22}
\begin{aligned}
&u^{-d\cdot[L]}N_q^{\mu_1,\mu_2}(X_k',d,g')\left((-i)(q^\frac{1}{2}-q^{1\frac{1}{2}})\right)^{n+g'-1+l(\mu_1)}\\
=&u^{-d\cdot[L]}\sum_{w(\mu_2)\neq0}w\cdot N_q^{\mu_1\cup(w),\mu_2\setminus(w)}(X_k',d,g')\left((-i)(q^\frac{1}{2}-q^{1\frac{1}{2}})\right)^{n+g'-1+l(\mu_1)}\\
&+\sum\Delta\left(u^{-d\cdot[L]+2}\prod_{j=1}^mN_q^{\mu_1^j,\mu_2^j}(X_k',d_j,g_j')\left((-i)(q^\frac{1}{2}-q^{-\frac{1}{2}})\right)^{n_j+g_j'-1+l(\mu_1^j)}\right)\\
&\cdot\left(u^{-2}\prod_{x\in\xi}\frac{[x]_q}{x}\prod_{y\in\eta}\frac{[y]_q}{y}\left((-i)(q^\frac{1}{2}-q^{-\frac{1}{2}})\right)^{l(\xi)+l(\eta)}\right).
\end{aligned}
\end{equation}
Note that
$$
\begin{aligned}
\sum_{j=1}^m(n_j+g_j'-1+l(\mu_1^j))+l(\xi)+l(\eta)&=\sum_{j=1}^mn_j+(\sum_{j=1}^mg_j'-m+l(\xi))+(\sum_{j=1}^ml(\mu_1^j)+l(\eta))\\
&=n-1+g'+l(\mu_1).
\end{aligned}
$$
From equation (\ref{eq:7.22}), we obtain the equality (\ref{eq:CH-2}).
\end{proof}

\section*{Acknowledgements}
The first author thanks Erwan Brugall\'e for valuable communications and suggestions during the early stage of this project. We thank H\"ulya Arg\"uz and Pierrick Bousseau for their explanation of equation (3.12) in \cite[Section 3.3]{ab25} and valuable comments by email. J. Hu is partially supported by the National Key R \& D Program of China (No.2023YFA1009801) and National Science Foundation of China (No. 1253103). Y. Ding was supported by National Science Foundation of China (No.12101565).

\appendix

\section{Some computations}\label{sec:appendix}
In this Appendix, we compute some refined counts of floor diagrams relative to a conic using two methods:
direct computation from the definition and the recursive formula in Theorem \ref{thm:CH}.

\subsection{Computation via floor diagrams}

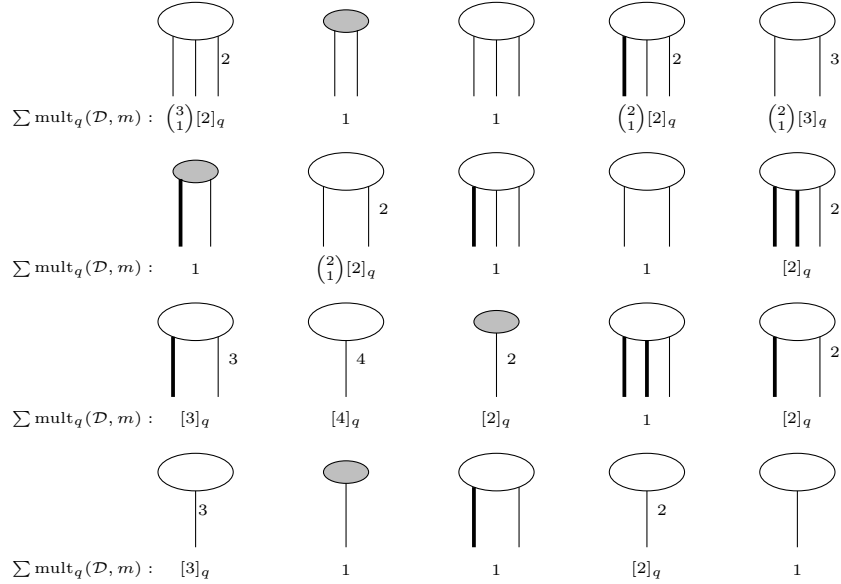
\begin{figure}[ht]
    \centering
    \begin{tikzpicture}
    \foreach \Point in {(0,0),(4,0),(6,0),(8,0)}
    \draw \Point ellipse (0.5 and 0.25);
    \foreach \Point in {(2,0)}
    \draw[fill=lightgray] \Point ellipse (0.3 and 0.15);
    \draw (-0.3,-0.2)--(-0.3,-1);
    \draw (0,-0.25)--(0,-1);
    \draw (0.3,-0.2)--(0.3,-1);
    \draw (1.85,-0.12)--(1.85,-1);
    \draw (2.15,-0.12)--(2.15,-1);
    \draw (3.7,-0.2)--(3.7,-1);
    \draw (4,-0.25)--(4,-1);
    \draw (4.3,-0.2)--(4.3,-1);
    \draw[line width=0.5mm] (5.7,-0.2)--(5.7,-1);
    \draw (6,-0.25)--(6,-1);
    \draw (6.3,-0.2)--(6.3,-1);
    \draw (8.3,-0.2)--(8.3,-1);
    \draw (7.7,-0.2)--(7.7,-1);
    \draw (0.4,-0.5) node{\tiny$2$} (6.4,-0.5) node{\tiny$2$}(8.5,-0.5) node{\tiny$3$};
    \draw (-1.5,-1.3) node{\tiny$\sum\mult_q(\dl,m):$} (0,-1.3) node{\tiny$\binom{3}{1}[2]_q$}  (2,-1.3) node{\tiny$1$}
     (4,-1.3) node{\tiny$1$} (6,-1.3) node{\tiny$\binom{2}{1}[2]_q$} (8,-1.3) node{\tiny$\binom{2}{1}[3]_q$};
    \foreach \Point in {(0,-2)}
    \draw[fill=lightgray] \Point ellipse (0.3 and 0.15);
    \foreach \Point in {(2,-2),(4,-2),(6,-2),(8,-2)}
    \draw \Point ellipse (0.5 and 0.25);
    \draw[line width=0.5mm] (-0.2,-2.1)--(-0.2,-3);
    \draw (0.2,-2.1)--(0.2,-3);
    \draw (1.7,-2.2)--(1.7,-3);
    \draw (2.3,-2.2)--(2.3,-3);
    \draw[line width=0.5mm] (3.7,-2.2)--(3.7,-3);
    \draw (4,-2.25)--(4,-3);
    \draw (4.3,-2.2)--(4.3,-3);
    \draw (5.7,-2.2)--(5.7,-3);
    \draw (6.3,-2.2)--(6.3,-3);
    \draw[line width=0.5mm] (7.7,-2.2)--(7.7,-3);
    \draw[line width=0.5mm] (8,-2.25)--(8,-3);
    \draw (8.3,-2.2)--(8.3,-3);
    \draw (2.5,-2.5) node{\tiny$2$} (8.5,-2.5) node{\tiny$2$};    
    \draw (-1.5,-3.3) node{\tiny$\sum\mult_q(\dl,m):$} (0,-3.3) node{\tiny$1$}  (2,-3.3) node{\tiny$\binom{2}{1}[2]_q$}
     (4,-3.3) node{\tiny$1$} (6,-3.3) node{\tiny$1$} (8,-3.3) node{\tiny$[2]_q$};
    \foreach \Point in {(0,-4),(2,-4),(6,-4),(8,-4)}
    \draw \Point ellipse (0.5 and 0.25);
    \foreach \Point in {(4,-4)}
    \draw[fill=lightgray] \Point ellipse (0.3 and 0.15);
    \draw[line width=0.5mm] (-0.3,-4.2)--(-0.3,-5);
    \draw (0.3,-4.2)--(0.3,-5);
    \draw (2,-4.25)--(2,-5);
    \draw (4,-4.15)--(4,-5);
    \draw[line width=0.5mm] (5.7,-4.2)--(5.7,-5);
    \draw[line width=0.5mm] (6,-4.25)--(6,-5);
    \draw (6.3,-4.2)--(6.3,-5);
    \draw[line width=0.5mm] (7.7,-4.2)--(7.7,-5);
    \draw (8.3,-4.2)--(8.3,-5);
    \draw (0.5,-4.5) node{\tiny$3$} (2.2,-4.5) node{\tiny$4$} (4.2,-4.5) node{\tiny$2$} (8.5,-4.4) node{\tiny$2$}; 
    \draw (-1.5,-5.3) node{\tiny$\sum\mult_q(\dl,m):$} (0,-5.3) node{\tiny$[3]_q$}  (2,-5.3) node{\tiny$[4]_q$}
     (4,-5.3) node{\tiny$[2]_q$} (6,-5.3) node{\tiny$1$} (8,-5.3) node{\tiny$[2]_q$};
    \foreach \Point in {(0,-6),(4,-6),(6,-6),(8,-6)}
    \draw \Point ellipse (0.5 and 0.25);
    \foreach \Point in {(2,-6)}
    \draw[fill=lightgray] \Point ellipse (0.3 and 0.15);
    \draw (0,-6.25)--(0,-7);
    \draw (2,-6.15)--(2,-7);
    \draw[line width=0.5mm] (3.7,-6.2)--(3.7,-7);
    \draw (4.3,-6.2)--(4.3,-7);
    \draw (6,-6.25)--(6,-7);
    \draw (8,-6.25)--(8,-7);
    \draw (0.1,-6.5) node{\tiny$3$} (6.2,-6.5) node{\tiny$2$}; 
    \draw (-1.5,-7.3) node{\tiny$\sum\mult_q(\dl,m):$} (0,-7.3) node{\tiny$[3]_q$}  (2,-7.3) node{\tiny$1$}
    (4,-7.3) node{\tiny$1$} (6,-7.3) node{\tiny$[2]_q$} (8,-7.3) node{\tiny$1$};
    \end{tikzpicture}
\caption{Floor diagrams in Example \ref{epl:7.1}.}
\label{fig:marked}
\end{figure}

We use Brugall\'e's convention \cite{bru2015} to depict floor diagrams: white ellipses denote floors of degree 2,
and grey ellipses denote floors of degree 1; vertices in $\vt^\infty(\dl)$ and edges in $m(\cup_{i=1}^kA_i)$ are not depicted; 
the edges of $\dl$ are represented by vertical lines
that are oriented from down to up; bold vertical lines denote the edges in $m(A_0)\cap\eg^\infty(\dl)$. 
We only mark the weights of the edges that are at least 2.

\begin{example}\label{epl:7.1}
From the floor diagrams depicted in Figure \ref{fig:marked}, we compute the values of some refined counts in the following.
The values of refined counts are listed according to the order of the floor diagrams in Figure \ref{fig:marked}.
\begin{align*}
&N_q^{\emptyset,(1,1,2)}(X_6',2[L],0)=3[2]_q, N_q^{\emptyset,(1,1)}(X_6',[L],0)=1, N_q^{\emptyset,(1,1,1)}(X_6',2[L]-[E_i],0)=1,\\
&N_q^{(1),(1,2)}(X_6',2[L],0)=2[2]_q, N_q^{\emptyset,(1,3)}(X_6',2[L],0)=2[3]_q, N_q^{(1),(1)}(X_6',[L],0)=1,\\
&N_q^{\emptyset,(1,2)}(X_6',2[L]-[E_i],0)=2[2]_q, N_q^{(1),(1,1)}(X_6',2[L]-[E_i],0)=1, \\
& N_q^{\emptyset,(1,1)}(X_6',2[L]-[E_i]-[E_j],0)=1, \text{ where } i,j\in\{1,\ldots,6\}, i\neq j.\\
& N_q^{(1,1),(2)}(X_6',2[L],0)=[2]_q, N_q^{(1),(3)}(X_6',2[L],0)=[3]_q, N_q^{\emptyset,(4)}(X_6',2[L],0)=[4]_q,\\
& N_q^{\emptyset,(2)}(X_6',[L],0)=[2]_q, N_q^{(1,1),(1)}(X_6',2[L]-[E_i],0)=1, N_q^{(1),(2)}(X_6',2[L]-[E_i],0)=[2]_q,\\
& N_q^{\emptyset,(3)}(X_6',2[L]-[E_i],0)=[3]_q, N_q^{\emptyset,(1)}(X_6',[L]-[E_i],0)=1,\\
& N_q^{(1),(1)}(X_6',2[L]-[E_i]-[E_j],0)=1, N_q^{\emptyset,(2)}(X_6',2[L]-[E_i]-[E_j],0)=[2]_q, i\neq j,\\
& N_q^{\emptyset,(1)}(X_6',2[L]-[E_i]-[E_j]-[E_k],0)=1, \text{ where } i,j, k\in\{1,\ldots,6\}, i\neq j, i\neq k, j\neq k.
\end{align*}
\end{example}

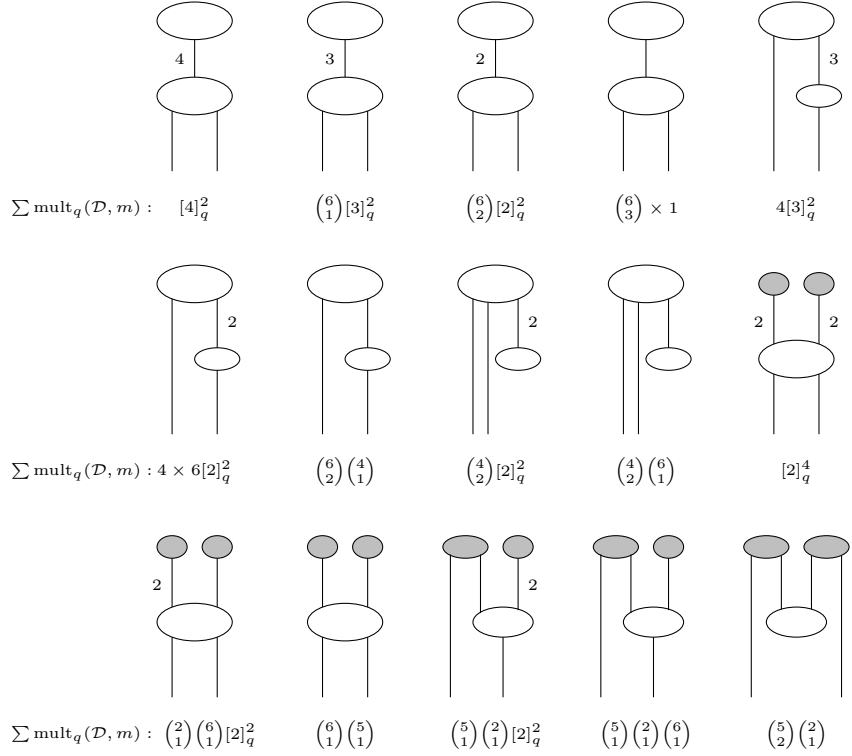
\begin{figure}[ht]
    \centering
    \begin{tikzpicture}
    \foreach \Point in {(0,0),(2,0),(4,0),(6,0),(8,0),(0,-1),(2,-1),(4,-1),(6,-1)}
    \draw \Point ellipse (0.5 and 0.25);
    \draw (0,-0.25)--(0,-0.75);
    \draw (2,-0.25)--(2,-0.75);
    \draw (4,-0.25)--(4,-0.75);
    \draw (6,-0.25)--(6,-0.75);
    \draw (-0.3,-1.2)--(-0.3,-2);
    \draw (0.3,-1.2)--(0.3,-2);
    \draw (1.7,-1.2)--(1.7,-2);
    \draw (2.3,-1.2)--(2.3,-2);
    \draw (3.7,-1.2)--(3.7,-2);
    \draw (4.3,-1.2)--(4.3,-2);
    \draw (5.7,-1.2)--(5.7,-2);
    \draw (6.3,-1.2)--(6.3,-2);    
    \draw (7.7,-0.2)--(7.7,-2);
    \draw (8.3,-0.2)--(8.3,-0.85);
    \draw (8.3,-1) ellipse (0.3 and 0.15);
    \draw (8.3,-1.15)--(8.3,-2);
    \draw (-0.2,-0.5) node{\tiny$4$} (1.8,-0.5) node{\tiny$3$} (3.8,-0.5) node{\tiny$2$} (8.5,-0.5) node{\tiny$3$};    
    \draw (-1.5,-2.5) node{\tiny$\sum\mult_q(\dl,m):$} (0,-2.5) node{\tiny$[4]_q^2$}  (2,-2.5) node{\tiny$\binom{6}{1}[3]_q^2$}
     (4,-2.5) node{\tiny$\binom{6}{2}[2]_q^2$} (6,-2.5) node{\tiny$\binom{6}{3}\times1$} (8,-2.5) node{\tiny$4[3]_q^2$};
     \foreach \Point in {(0,-3.5),(2,-3.5),(4,-3.5),(6,-3.5)}
    \draw \Point ellipse (0.5 and 0.25);
    \draw (-0.3,-3.7)--(-0.3,-5.5);
    \draw (0.3,-3.7)--(0.3,-4.35);
    \draw (0.3,-4.5) ellipse (0.3 and 0.15);
    \draw (0.3,-4.65)--(0.3,-5.5);
    \draw (1.7,-3.7)--(1.7,-5.5);
    \draw (2.3,-3.7)--(2.3,-4.35);
    \draw (2.3,-4.5) ellipse (0.3 and 0.15);
    \draw (2.3,-4.65)--(2.3,-5.5);
    \draw (3.7,-3.7)--(3.7,-5.5);
    \draw (3.9,-3.75)--(3.9,-5.5);
    \draw (4.3,-3.7)--(4.3,-4.35);
    \draw (4.3,-4.5) ellipse (0.3 and 0.15);
    \draw (5.7,-3.7)--(5.7,-5.5);
    \draw (5.9,-3.75)--(5.9,-5.5);
    \draw (6.3,-3.7)--(6.3,-4.35);
    \draw (6.3,-4.5) ellipse (0.3 and 0.15);
    \foreach \Point in {(7.7,-3.5),(8.3,-3.5)}
    \draw[fill=lightgray] \Point ellipse (0.2 and 0.15);
    \draw (8.3,-3.65)--(8.3,-4.3);
    \draw (7.7,-3.65)--(7.7,-4.3);
    \draw (8,-4.5) ellipse (0.5 and 0.25);    
    \draw (7.7,-4.7)--(7.7,-5.5);
    \draw (8.3,-4.7)--(8.3,-5.5);
    \draw (0.5,-4) node{\tiny$2$} (4.5,-4) node{\tiny$2$} (7.5,-4) node{\tiny$2$} (8.5,-4) node{\tiny$2$};    
    \draw (-1.5,-6) node{\tiny$\sum\mult_q(\dl,m):$} (0,-6) node{\tiny$4\times6[2]_q^2$}  (2,-6) node{\tiny$\binom{6}{2}\binom{4}{1}$}
     (4,-6) node{\tiny$\binom{4}{2}[2]_q^2$} (6,-6) node{\tiny$\binom{4}{2}\binom{6}{1}$} (8,-6) node{\tiny$[2]_q^4$};
    \foreach \Point in {(-0.3,-7),(0.3,-7),(1.7,-7),(2.3,-7),(4.3,-7),(6.3,-7)}
    \draw[fill=lightgray] \Point ellipse (0.2 and 0.15);
    \foreach \Point in {(3.6,-7),(5.6,-7),(7.6,-7),(8.4,-7)}
    \draw[fill=lightgray] \Point ellipse (0.3 and 0.15);
    \foreach \Point in {(0,-8),(2,-8)}
    \draw[] \Point ellipse (0.5 and 0.25);
    \foreach \Point in {(4.1,-8),(6.1,-8)}
    \draw[] \Point ellipse (0.4 and 0.2);
    \draw (-0.3,-7.15)--(-0.3,-7.8);
    \draw (0.3,-7.15)--(0.3,-7.8);
    \draw (-0.3,-8.2)--(-0.3,-9);
    \draw (0.3,-8.2)--(0.3,-9);
    \draw (1.7,-7.15)--(1.7,-7.8);
    \draw (2.3,-7.15)--(2.3,-7.8);
    \draw (1.7,-8.2)--(1.7,-9);
    \draw (2.3,-8.2)--(2.3,-9);
    \draw (3.8,-7.1)--(3.8,-7.88);
    \draw (3.4,-7.1)--(3.4,-9);
    \draw (4.3,-7.15)--(4.3,-7.84);
    \draw (4.1,-8.2)--(4.1,-9);
    \draw (5.8,-7.1)--(5.8,-7.88);
    \draw (5.4,-7.1)--(5.4,-9);
    \draw (6.3,-7.15)--(6.3,-7.84);
    \draw (6.1,-8.2)--(6.1,-9);
    \draw (7.4,-7.12)--(7.4,-9);
    \draw (7.8,-7.12)--(7.8,-7.83);
    \draw (8.6,-7.12)--(8.6,-9);
    \draw (8.2,-7.12)--(8.2,-7.83);
    \draw (8,-8) ellipse (0.4 and 0.2);    
    \draw (-0.5,-7.5) node{\tiny$2$} (4.5,-7.5) node{\tiny$2$};    
    \draw (-1.5,-9.5) node{\tiny$\sum\mult_q(\dl,m):$} (0.2,-9.5) node{\tiny$\binom{2}{1}\binom{6}{1}[2]_q^2$}  (2,-9.5) node{\tiny$\binom{6}{1}\binom{5}{1}$}
     (4,-9.5) node{\tiny$\binom{5}{1}\binom{2}{1}[2]_q^2$} (6,-9.5) node{\tiny$\binom{5}{1}\binom{2}{1}\binom{6}{1}$} (8,-9.5) node{\tiny$\binom{5}{2}\binom{2}{1}$};
    \end{tikzpicture}
\caption{$4[L]-\sum_{i=1}^6[E_i]$-marked floor diagrams of genus 0 and type $(\emptyset,(1,1))$ from \cite[Fig. 3]{bru2015}.}
\label{fig:4L-E_i-marked}
\end{figure}

\begin{example}\label{epl:7.2}
All $4[L]-\sum_{i=1}^6[E_i]$-marked floor diagrams of genus 0 and type $(\emptyset,(1,1))$ were listed in \cite[Fig. 3]{bru2015}.
From \cite[Fig. 3]{bru2015}, we obtain the value of the refined counts $N_q^{\emptyset,(1,1)}(X_6',4[L]-\sum_{i=1}^6[E_i],0)$.
For the readers' convenience,  $4[L]-\sum_{i=1}^6[E_i]$-marked floor diagrams of genus 0 and type $(\emptyset,(1,1))$ in \cite[Fig. 3]{bru2015} are also depicted in Figure \ref{fig:4L-E_i-marked}.
$$
N_q^{\emptyset,(1,1)}(X_6',4[L]-\sum_{i=1}^6[E_i],0)=[2]_q^4+[4]_q^2+10[3]_q^2+67[2]_q^2+226.
$$
\end{example}

\subsection{Computation via the recursive formula}
We use the Caporaso-Harris recursive formula in Theorem \ref{thm:CH} to compute the refined counts listed in Example \ref{epl:7.1} and Example \ref{epl:7.2}.

\begin{align*}
N_q^{\emptyset,(1,1,2)}(X_6',2[L],0)=&2N_q^{(2),(1,1)}(X_6',2[L],0)+N_q^{(1),(1,2)}(X_6',2[L],0)\\
=&2N_q^{(1,2),(1)}(X_6',2[L],0)+N_q^{(1,1),(2)}(X_6',2[L],0)\\
&+2N_q^{(1,2),(1)}(X_6',2[L],0)\\
=&6N_q^{(1,1,2),\emptyset}(X_6',2[L],0)\\
=&6\prod_{i=1}^6N_q^{\emptyset,(1)}(X_6',[E_i],0)\frac{[2]_q}{2}
=3[2]_q.\\
N_q^{\emptyset,(1,1)}(X_6',[L],0)=&N_q^{(1),(1)}(X_6',[L],0)=N_q^{(1,1),\emptyset}(X_6',[L],0)=1.\\
N_q^{\emptyset,(1,1,1)}(X_6',2[L]-[E_i],0)=&N_q^{(1),(1,1)}(X_6',2[L]-[E_i],0)\\
=&N_q^{(1,1),(1)}(X_6',2[L]-[E_i],0)\\
=&N_q^{(1,1,1),\emptyset}(X_6',2[L]-[E_i],0)\\
=&\prod_{j\neq i}N_q^{\emptyset,(1)}(X_6',[E_j],0)
=1.\\
N_q^{\emptyset,(1,3)}(X_6',2[L],0)=&N_q^{(1),(3)}(X_6',2[L],0)+3N_q^{(3),(1)}(X_6',2[L],0)\\
=&6N_q^{(1,3),\emptyset}(X_6',2[L],0)\\
=&6\prod_{j=1}^6N_q^{\emptyset,(1)}(X_6',[E_j],0)\frac{[3]_q}{3}
=2[3]_q.\\
N_q^{\emptyset,(1,2)}(X_6',2[L]-[E_i],0)=&N_q^{(1),(2)}(X_6',2[L]-[E_i],0)+2N_q^{(2),(1)}(X_6',2[L]-[E_i],0)\\
=&4N_q^{(1,2),\emptyset}(X_6',2[L]-[E_i],0)\\
=&4\prod_{j\neq i}N_q^{\emptyset,(1)}(X_6',[E_j],0)\frac{[2]_q}{2}
=2[2]_q.\\
N_q^{\emptyset,(1,1)}(X_6',2[L]-[E_i]-[E_j],0)=&N_q^{(1),(1)}(X_6',2[L]-[E_i]-[E_j],0)\\
=&N_q^{(1,1),\emptyset}(X_6',2[L]-[E_i]-[E_j],0)\\
=&\prod_{l\neq i,j}N_q^{\emptyset,(1)}(X_6',[E_l],0)
=1.\\
N_q^{\emptyset,(4)}(X_6',2[L],0)=&4N_q^{(4),\emptyset}(X_6',2[L],0)
=4\prod_{i=1}^6N_q^{\emptyset,(1)}(X_6',[E_i],0)\frac{[4]_q}{4}=[4]_q.\\
N_q^{\emptyset,(2)}(X_6',[L],0)=&2N_q^{(2),\emptyset}(X_6',[L],0)=[2]_q.\\
N_q^{(1),(2)}(X_6',2[L]-[E_i],0)=&2N_q^{(1,2),\emptyset}(X_6',2[L]-[E_i],0)\\
=&2\prod_{j\neq i}N_q^{\emptyset,(1)}(X_6',[E_j],0)\frac{[2]_q}{2}=[2]_q.\\
N_q^{\emptyset,(3)}(X_6',2[L]-[E_i],0)=&3N_q^{(3),\emptyset}(X_6',2[L]-[E_i],0)\\
=&3\prod_{j\neq i}N_q^{\emptyset,(1)}(X_6',[E_j],0)\frac{[3]_q}{3}=[3]_q.\\
N_q^{\emptyset,(2)}(X_6',2[L]-[E_i]-[E_j],0)=&2N_q^{(2),\emptyset}(X_6',2[L]-[E_i]-[E_j],0)\\
=&2\prod_{l\neq i,j}N_q^{\emptyset,(1)}(X_6',[E_l],0)\frac{[2]_q}{2}=[2]_q.\\
N_q^{\emptyset,(1)}(X_6',2[L]-[E_i]-[E_j]-[E_k],0)=&N_q^{(1),\emptyset}(X_6',2[L]-[E_i]-[E_j]-[E_k],0)\\
=&\prod_{l\neq i,j,k}N_q^{\emptyset,(1)}(X_6',[E_l],0)=1.
\end{align*}

The above computations also imply the equalities in the following.
\begin{align*}
N_q^{(1),(1,2)}(X_6',2[L],0)=2[2]_q, N_q^{(1,1),(2)}(X_6',2[L],0)=[2]_q,
N_q^{(1),(3)}(X_6',2[L],0)=[3]_q. 
\end{align*}

From the Caporaso-Harris recursive formula (\ref{eq:CH-2}) and the above computations, we obtain the following equalities.
\begin{align*}
&N_q^{(1,1),\emptyset}(X_6',4[L]-\sum_{i=1}^6[E_i],0)\\
&=N_q^{(1,1),(2)}(X_6',2[L],0)[2]_q+\binom{2}{1}N_q^{(1),(3)}(X_6',2[L],0)\cdot[3]_q+N_q^{\emptyset,(4)}(X_6',2[L],0)\cdot[4]_q\\
&+\frac{1}{2}\binom{2}{1}\binom{2}{1}\left(N_q^{(1),(1)}(X_6',[L],0)\right)^2
+\binom{2}{1}\binom{2}{1}N_q^{(1),(1)}(X_6',[L],0)N_q^{\emptyset,(2)}(X_6',[L],0)\cdot[2]_q\\
&+\frac{1}{2}\binom{2}{1}\left(N_q^{\emptyset,(2)}(X_6',[L],0)\right)^2\cdot[2]_q^2
+\sum_{i=1}^6N_q^{(1,1),(1)}(X_6',2[L]-[E_i],0)N_q^{\emptyset,(1)}(X_6',[E_i],0)\\
&+\sum_{i=1}^6\binom{2}{1}N_q^{(1),(2)}(X_6',2[L]-[E_i],0)N_q^{\emptyset,(1)}(X_6',[E_i],0)\cdot[2]_q\\
&+\sum_{i=1}^6N_q^{\emptyset,(3)}(X_6',2[L]-[E_i],0)N_q^{\emptyset,(1)}(X_6',[E_i],0)\cdot[3]_q\\
&+\sum_{i=1}^6\binom{2}{1}\binom{2}{1}N_q^{(1),(1)}(X_6',[L],0)N_q^{\emptyset,(1)}(X_6',[L]-[E_i],0)N_q^{\emptyset,(1)}(X_6',[E_i],0)\\
&+\sum_{i=1}^6\binom{2}{1}N_q^{\emptyset,(2)}(X_6',[L],0)N_q^{\emptyset,(1)}(X_6',[L]-[E_i],0)N_q^{\emptyset,(1)}(X_6',[E_i],0)\cdot[2]_q\\
&+\sum_{i\neq j}\binom{2}{1}N_q^{(1),(1)}(X_6',2[L]-[E_i]-[E_j],0)N_q^{\emptyset,(1)}(X_6',[E_i],0)N_q^{\emptyset,(1)}(X_6',[E_j],0)\\
&+\sum_{i\neq j}N_q^{\emptyset,(2)}(X_6',2[L]-[E_i]-[E_j],0)N_q^{\emptyset,(1)}(X_6',[E_i],0)N_q^{\emptyset,(1)}(X_6',[E_j],0)\cdot[2]_q\\
&+\sum_{i\neq j}\binom{2}{1}N_q^{\emptyset,(1)}(X_6',[L]-[E_i],0)N_q^{\emptyset,(1)}(X_6',[L]-[E_j],0)\\
&~~~~~\cdot N_q^{\emptyset,(1)}(X_6',[E_i],0)N_q^{\emptyset,(1)}(X_6',[E_j],0)\\
&+\sum_{i\neq j,i\neq k, j\neq k}N_q^{\emptyset,(1)}(X_6',2[L]-[E_i]-[E_j]-[E_k],0)N_q^{\emptyset,(1)}(X_6',[E_i],0)\\
&~~~~~\cdot N_q^{\emptyset,(1)}(X_6',[E_j],0)N_q^{\emptyset,(1)}(X_6',[E_k],0)\\
&=[2]_q^2+2[3]_q^2+[4]_q^2+2+4[2]_q^2+[2]_q^4+6+12[2]_q^2+6[3]_q^2+24+12[2]_q^2\\
&+30+15[2]_q^2+30+20\\
&=[2]_q^4+[4]_q^2+8[3]_q^2+44[2]_q^2+112.\\
&N_q^{(1),(1)}(X_6',4[L]-\sum_{i=1}^6[E_i],0)\\
=&N_q^{(1,1),\emptyset}(X_6',4[L]-\sum_{i=1}^6[E_i],0)+N_q^{(1),(1,2)}(X_6',2[L],0)\cdot[2]_q+N_q^{\emptyset,(1,3)}(X_6',2[L],0)\cdot[3]_q\\
&+\binom{3}{2}\binom{2}{1}N_q^{\emptyset,(1,1)}(X_6',[L],0)N_q^{\emptyset,(2)}(X_6',[L],0)[2]_q\\
&+\binom{3}{2}\binom{2}{1}N_q^{(1),(1)}(X_6',[L],0)N_q^{\emptyset,(1,1)}(X_6',[L],0)\\
&+\sum_{i=1}^6N_q^{\emptyset,(1,2)}(X_6',2[L]-[E_i],0)N_q^{\emptyset,(1)}(X_6',[E_i],0)\cdot[2]_q\\
&+\sum_{i=1}^6\binom{2}{1}N_q^{(1),(1,1)}(X_6',2[L]-[E_i],0)N_q^{\emptyset,(1)}(X_6',[E_i],0)\\
&+\sum_{i=1}^6\binom{3}{2}\binom{2}{1}N_q^{\emptyset,(1,1)}(X_6',[L],0)N_q^{\emptyset,(1)}(X_6',[L]-[E_i],0)N_q^{\emptyset,(1)}(X_6',[E_i],0)\\
&+\sum_{i\neq j}\binom{2}{1}N_q^{\emptyset,(1,1)}(X_6',2[L]-[E_i]-[E_j],0)N_q^{\emptyset,(1)}(X_6',[E_i],0)N_q^{\emptyset,(1)}(X_6',[E_j],0)\\
=&[2]_q^4+[4]_q^2+8[3]_q^2+44[2]_q^2+112+2[2]_q^2+2[3]_q^2
+6[2]_q^2+6+12[2]_q^2\\
&+12+36+30\\
=&[2]_q^4+[4]_q^2+10[3]_q^2+64[2]_q^2+196.\\
&N_q^{\emptyset,(1,1)}(X_6',4[L]-\sum_{i=1}^6[E_i],0)\\
=&N_q^{(1),(1)}(X_6',4[L]-\sum_{i=1}^6[E_i],0)+N_q^{\emptyset,(1,1,2)}(X_6',2[L],0)[2]_q\\
&+\frac{1}{2}\binom{4}{2,2}\binom{(1,1)}{(1)}\binom{(1,1)}{(1)}\left(N_q^{\emptyset,(1,1)}(X_6',[L],0)\right)^2\\
&+\sum_{i=1}^6\binom{3}{2}N_q^{\emptyset,(1,1,1)}(X_6',2[L]-[E_i],0)N_q^{\emptyset,(1)}(X_6',[E_i],0)\\
=&[2]_q^4+[4]_q^2+10[3]_q^2+64[2]_q^2+196+3[2]_q^2+12+18\\
=&[2]_q^4+[4]_q^2+10[3]_q^2+67[2]_q^2+226.
\end{align*}


\end{document}